\documentclass{article}
\usepackage[draft]{todonotes}   
\usepackage{amsthm,amsmath}
\usepackage{amscd,amssymb,latexsym,graphicx}
\usepackage{amscd}
\usepackage[all]{xy}
\usepackage{cite}
\usepackage{mathtools}
\usepackage{hycolor}
\usepackage{xcolor}
\usepackage[
                       colorlinks=true,
                       linkcolor=black, 
                       citecolor=black, 
                       urlcolor=blue,
                     ]{hyperref}
\title{Stable foliations and \\
         semi-flow Morse homology
       }
\author{Joa Weber\footnote{
        Financial support:
        FAPESP grant 2013/20912-4, 
        FAEPEX grant 1135/2013
                 }
                 \\
        IMECC UNICAMP\\
       }
\newtheorem{theoremABC}{Theorem}

\newtheorem{theorem}{Theorem}[section]
\newtheorem{corollary}[theorem]{Corollary}

\newtheorem{lemma}[theorem]{Lemma}
\newtheorem{proposition}[theorem]{Proposition}

\theoremstyle{definition}
\newtheorem{definition}[theorem]{Definition}
\newtheorem{hypothesis}[theorem]{Hypothesis}
\newtheorem{remark}[theorem]{Remark}

\theoremstyle{remark}

\renewcommand{\1}{{{\mathchoice {\rm 1\mskip-4mu l} {\rm 1\mskip-4mu l}
{\rm 1\mskip-4.5mu l} {\rm 1\mskip-5mu l}}}}
\renewcommand{\graph}{{\rm graph }}  
\newcommand{\D}{{\mathbb{D}}}

\newcommand{\N}{{\mathbb{N}}}

\newcommand{\R}{{\mathbb{R}}}
\renewcommand{\SS}{{\mathbb{S}}}

\newcommand{\Z}{{\mathbb{Z}}}
\newcommand{\Bb}{{\mathcal{B}}}
\newcommand{\Dd}{{\mathcal{D}}}
\newcommand{\Ee}{{\mathcal{E}}}
\newcommand{\Ff}{{\mathcal{F}}}
\newcommand{\Gg}{{\mathcal{G}}}   

\newcommand{\Ll}{{\mathcal{L}}}   
\newcommand{\Mm}{{\mathcal{M}}}   
\newcommand{\Nn}{{\mathcal{N}}}
\newcommand{\Oo}{{\mathcal{O}}}

\newcommand{\Ss}{{\mathcal{S}}}
\newcommand{\Tt}{{\mathcal{T}}}
\newcommand{\Uu}{{\mathcal{U}}}
\newcommand{\Vv}{{\mathcal{V}}}
\newcommand{\Ww}{{\mathcal{W}}}

\newcommand{\im}{{\rm im\, }}      
\newcommand{\range}{{\rm range\, }}  

\newcommand{\diag}{{\rm diag}}     
\newcommand{\cl}{{\rm cl\, }}         
\newcommand{\dist}{{\rm dist}}     
\newcommand{\INT}{{\rm int\, }}       
\newcommand{\IND}{{\rm ind}}       

\newcommand{\grad}{{\rm grad }}    
\newcommand{\Crit}{{\rm Crit}}        
\newcommand{\Ho}{{\rm H}}             
\newcommand{\Co}{{\rm C}}              
\newcommand{\HM}{{\rm HM}}          
\newcommand{\CM}{{\rm CM}}          

\newcommand{\W}{{\rm W}}

\newcommand{\norm}{{\rm norm}}

\newcommand{\eps}{{\varepsilon}}

\def\NABLA#1{{\mathop{\nabla\kern-.5ex\lower1ex\hbox{$#1$}}}}
\def\Nabla#1{\nabla\kern-.5ex{}_{#1}}
\def\Tabla#1{\Tilde\nabla\kern-.5ex{}_{#1}}
\def\abs#1{\mathopen|#1\mathclose|}   
\def\Abs#1{\left|#1\right|}            
\def\norm#1{\mathopen\|#1\mathclose\|}
\def\Norm#1{\left\|#1\right\|}

\renewcommand{\Tilde}{\widetilde}

\newcommand{\p}{{\partial}}

\begin{document}
\maketitle
%
%
%
\begin{abstract}
In case of the heat flow on the free loop
space of a closed Riemannian manifold
non-triviality of Morse homology for
semi-flows is established by
constructing a natural isomorphism
to singular homology
of the loop space. 
The construction is also
new in finite dimensions.
The main idea is to build a
Morse filtration using Conley pairs
and their pre-images
under the time-$T$-map of the heat flow. 
A crucial step is to contract each Conley pair
onto its part in the unstable manifold.
To achieve this we construct stable
foliations for Conley pairs using the recently
found backward
$\lambda$-Lemma~\cite{weber:2014a}.
These foliations are of independent
interest~\cite{weber:2014d}.
\end{abstract}
\tableofcontents

\section{Main results}
\label{sec:intro}
Consider a closed Riemannian manifold
$(M,g)$. A smooth
function $V\in C^\infty(S^1\times M)$,
called potential,
gives rise to the {\bf classical action
functional}
$$
     \Ss_V(\gamma)
     =\int_0^1\left(
     \frac12\Abs{\dot\gamma(t)}^2
     -V(t,\gamma(t))\right) dt
$$
defined on the
{\bf\boldmath free loop space of $M$},
that is the Hilbert manifold
$\Lambda M=W^{1,2}(S^1,M)$ which consists of all
absolutely continuous maps
$\gamma:S^1\to M$ whose first
derivative is square integrable.
Here and throughout we identify
$S^1=\R/\Z$ and
think of maps defined on $S^1$ as
$1$-periodic maps defined on $\R$.
Let $\nabla$ be the Levi-Civita
connection.
The set $\Crit$ of critical
points of $\Ss_V$ consists of the
$1$-periodic solutions of the ODE
\begin{equation}\label{eq:1-periodic}
     \Nabla{t}\dot x+\nabla V_t(x)=0
\end{equation}
where $V_t(q):=V(t,q)$.
For constant $V$ these are
the closed geodesics.
The negative $L^2$ gradient of $\Ss_\Vv$ is
given by the left hand side of~(\ref{eq:1-periodic})
and defined on a dense subset
$W^{2,2}$ of $\Lambda M$.
It generates a $C^1$ semi-flow
\begin{equation*}
       \varphi :(0,\infty)\times\Lambda M\to\Lambda M
\end{equation*}
which extends continuously to time zero,
preserves sublevel sets,
and is called the {\bf heat flow}; see
e.g.~\cite{henry:1981a,weber:2010a-link,weber:2013b}.
The semi-flow still exists for a class of abstract
perturbations, introduced in~\cite{salamon:2006a},
that take the form of smooth maps
$\Vv:\Lambda M\to\R$ which satisfy certain
axioms, say~{\rm (V0)--(V3)} in the notation
of~\cite{weber:2013b}. 
  These perturbations allow to achieve
  Morse-Smale transversality generically;
  see~\cite{weber:2013b}.
  They extend from the dense subset
  $\Ll M=C^\infty(S^1,M)$
  to $\Lambda M$ by~(V0).
Define $\varphi_s\gamma=u(s,\cdot)$
where $u:[0,\infty)\times S^1\to M$
solves the {\bf heat equation}
\begin{equation}\label{eq:heat}
     \p_su-\Nabla{t}\p_tu-\grad\Vv(u)=0
\end{equation}
with $u(0,\cdot)=\gamma$.
If $\Vv(\gamma)=\int_0^1 V_t(\gamma(t)) dt$,
then $\grad\Vv(u)=\nabla V_t(u)$;
see~\cite{weber:2013b}.

\subsection{Semi-flow Morse homology}
\label{sec:main-results}
From now on fix $V$ in the residual
(hence dense)
subset of $C^\infty(S^1\times M,\R)$ for
which $\Ss_V$ is a {\bf Morse function},
that is all critical points are
nondegenerate; see~\cite{weber:2002a}.
An {\bf\boldmath oriented critical
point} $\langle x\rangle$ or $o_x$
is a critical point $x$ together with
an orientation of the
maximal vector subspace
$E_x\subset T_x\Lambda M$
on which the Hessian of $\Ss_\Vv$
is negative definite. Recall that the
dimension of $E_x$, denoted by
$\IND_V(x)$, is finite and called the 
{\bf\boldmath Morse index of $x$};
see e.g.~\cite{weber:2002a}.

\subsubsection*{Chain groups}
Fix a regular value $a$ of
$\Ss_V$.
The set $\Crit^a$ of critical
points of the Morse function
$\Ss_V$ defined on the sublevel set
$$
     \Lambda^a M=\{\Ss_V\le a\}
$$
is a finite set, see e.g.~\cite{weber:2002a},
hence the set $\Crit$ is countable.
To avoid dependence of the Morse
chain complex on the (traditionally
taken and lamented)
apriori choices of orientations
a look at the construction of
simplicial homology is useful;
see e.g.~\cite[\S5]{munkres:1984a}.
In this theory all simplices are taken
oriented, because the \emph{algebraic}
boundary operator induces on
(or transports to) the faces
precisely the \emph{geometric} boundary
orientation which eventually
leads to $\p^2=0$. Then in a second step
one factors out opposite orientations.
In the context of Floer homology
a similar approach was
taken recently by Abbondandolo
and Schwarz~\cite{abbondandolo:2013b}
who use oriented critical points
as generators and then factor out
opposite orientations.
This requires a mechanism of
orientation transport, but
avoids having unnatural
orientations built into the chain
complex and therefore allows for a
natural isomorphism to singular
homology.

By definition the
{\bf\boldmath Morse chain group
$\CM_*^a=\CM_*^a(V)$}
is the free abelian group generated by
the (finite) set of oriented critical points
$\langle x\rangle$, likewise denoted by $o_x$, below level
$a$ and subject to the relations
\begin{equation}\label{eq:relation}
     o_x+\bar o_x=0
     ,\qquad
     \forall x\in\Crit^a,
\end{equation}
where $\bar o_x$ is the orientation opposite to $o_x$.
The Morse index provides a natural
grading and $\Crit^a_k\subset\Crit^a$
denotes the set
of critical points of Morse index $k$.

\subsubsection*{Boundary operator}
Fix an element
$v=v_a:\Lambda M\to\R$
of the set $\Oo^a_{reg}$ of regular
perturbations defined 
in~\cite[\S 5]{weber:2013b},
set
\begin{equation}\label{eq:perturbation-MS}
     \Vv(\gamma;V,v_a)
     =v_a(\gamma)
     +\int_0^1 V(t,\gamma(t))\, dt ,
\end{equation}
and note the following consequences.
Firstly, on $\Lambda^a M$ the critical
points of $\Ss_V$ and the
{\bf perturbed action} $\Ss_\Vv$, also called
{\bf Morse-Smale function}, given by
\begin{equation}\label{eq:action-fct}
     \Ss_\Vv(\gamma)
     =\frac12\int_0^1
     \Abs{\dot\gamma(t)}^2\, dt
     -\Vv(\gamma)
\end{equation}
coincide by~\cite[\S 5  Prop.~8]{weber:2013b}.
In abuse of notation we denote
the perturbed action
$\Ss_\Vv$ sometimes by $\Ss_{V+v_a}$.
In fact, both functionals
coincide on a neighborhood $U=U(V)$
in $\Lambda M$ of the set $\Crit$ of
\emph{all} critical points.
Therefore the subspaces $E_x$ do
not change under such perturbations.
Secondly, the perturbed action $\Ss_\Vv$
is {\bf\boldmath Morse Smale below level
$a$} in the functional analytic sense
of~\cite[\S 1]{weber:2013b}.

By~\cite[\S 6 Thm.~18]{weber:2013b}
the unstable manifold $W^u(x)=W^u(x;\Vv)$
of any critical point $x$ is a contractible,
thus orientable, smooth submanifold of
$\Lambda M$ whose dimension is given
by the Morse index $k=\IND_V(x)$.
On the other hand, for $\eps=\eps(a)>0$
small\footnote{
  As a consequence of the local stable
  manifold theorem,
  see e.g.~\cite[\S 2.5 Thm.~3]{weber:2014a},
  and the Palais-Morse Lemma there is
  a constant $\eps_a>0$ such that
  the assertion holds
  $\forall\eps\in(0,\eps_a]$.  
  }
the {\bf stable} or {\bf ascending disk}
\begin{equation}\label{eq:asc-disk}
     W^s_\eps(y)=W^s_\eps(y;\Vv)
     :=W^s(y;\Vv)\cap
     \{\Ss_\Vv<\Ss_\Vv(y)+\eps\}
\end{equation}
of any $y\in\Crit^a$
is a $C^1$ Hilbert submanifold of
$\Lambda M$ of finite codimension
$\ell=\IND_V(y)$.
Since $T_y W^u(y)$ is the orthogonal
complement of the tangent space
at $y$ to the ascending
disk $W^s_\eps(y)$,
an orientation of the unstable manifold
determines a co-orientation
of the (contractible) ascending disk and vice versa.

The functional
analytic characterization of the
{\bf\boldmath Morse-Smale condition
below level $a$} used in the
definition of $\Oo^a_{reg}$ translates
into the form common in dynamical systems,
namely that all intersections
\begin{equation}\label{eq:MS-DS}
     M_{xy}^\eps
     :=W^u(x)\pitchfork\W^s_\eps(y)
     ,\qquad
     \forall x,y\in\Crit^a,
\end{equation}
are cut out transversely from $\Lambda M$.
Consequently these intersections are
$C^1$ manifolds of dimension equal to
the Morse index difference $k-\ell$.
They are naturally oriented given
an orientation of $W^u(x)$ and a
co-orientation of $W^s_\eps(y)$.
More precisely,
condition~(\ref{eq:MS-DS})
implies that there is the pointwise
splitting
\begin{equation}\label{eq:or-ds}
     T_\gamma W^u(x)
     \cong T_\gamma M_{xy}^\eps\oplus
     \left(T_\gamma W^s_\eps(y)
     \right)^\perp
     ,\qquad
     \gamma\in M_{xy}^\eps,
\end{equation}
into two orthogonal subspaces.
Furthermore, for generic
$\delta\in(0,\eps)$ each set
\begin{equation}\label{eq:m_xy}
     m_{xy}:=M_{xy}^\eps\cap
     \left\{\Ss_\Vv=\Ss_\Vv(y)
     +\delta\right\}
     ,\quad
     \forall x,y\in\Crit^a,
\end{equation}
is cut out transversely from $M_{xy}^\eps$
and therefore inherits the structure
of a $C^1$ manifold of dimension
$k-\ell-1$.
By the gradient nature of the
heat flow each trajectory between $x$ and
$y$ intersects a level set
precisely once. Thus the elements of
$m_{xy}$ correspond precisely
to the heat flow lines from $x$ to $y$
(modulo time shift).
Therefore one calls $m_{xy}$
{\bf\boldmath manifold of connecting
trajectories between $x$ and $y$}.

Now consider the case of index
difference $1$. Fix an oriented
critical point $\langle x\rangle$ of Morse
index $k$. Then $m_{xy}$ is a finite
set for any $y\in\Crit_{k-1}$
by~\cite[Prop.~1]{weber:2013b}.\footnote{
  Identify $m_{xy}$ and the space
  $\Mm(x,y)/\R$ in~\cite{weber:2013b}
  via the bijection
  $\gamma\mapsto u(s,t)
  :=\left(\varphi_s\gamma\right)(t)$.
  Actually,
  if there are no critical points whose
  action lies between that of $x$ and
  $y$, then the finite set property is elementary:
  Because $m_{xy}$ is the transversal
  intersection -- inside the level
  hypersurface
  $\{\Ss_\Vv=\Ss_\Vv(y)+\eps/2\}$ --
  of a descending $k$-sphere $S^u(x)$
  and an ascending sphere of $y$
  of codimension $k$,
  finiteness of $m_{xy}$ follows from
  compactness of $S^u(x)$.
  }
The orientation $\langle x\rangle$
of $E_x=T_x W^u(x)$ extends to
an orientation of $W^u(x)$. Because
the dimension of $M_{xy}^\eps$ is one,
each of its components is a heat flow line
which runs to $y$ and, most importantly,
is \emph{naturally oriented}
by the forward/downward
flow. Because two of the vector spaces
in~(\ref{eq:or-ds}) are oriented,
declaring the direct sum an oriented
direct sum determines an orientation
of the third space.
More precisely, the identity
\begin{equation}\label{eq:or-ds-2}
  \left\langle T_\gamma W^u(x)\right\rangle
     _{\langle x\rangle}
  \cong
  \left\langle{\textstyle\frac{d}{ds}\varphi_s
     \gamma}\right\rangle_{\langle\mathrm{flow}\rangle}
  \oplus
  \left\langle T_\gamma W^s_\eps(y)^\perp\right\rangle
     _{u_*\langle x\rangle}
  ,\qquad
  \gamma\in m_{xy},
\end{equation}
determines a co-orientation of
$W^s_\eps(y)$, thus an orientation of
$W^u(y)$, depending on $\langle x\rangle$.
This orientation, denoted by
$u_*\langle x\rangle$ or by
$\langle y\rangle_{u_*\langle x\rangle}$ to
emphasize the target critical point
$y=y(u_\gamma)=u_\gamma(\infty)$,
is called the
{\bf transport} or
{\bf\boldmath push-forward of
$\langle x\rangle$ along the
trajectory $u=u_\gamma$} where
$u_\gamma(s)=\varphi_s\gamma$.
Already in the early days of finite dimensional
Morse homology a corresponding
procedure appeared in~\cite{salamon:1990a},
although it was used to compare, not
to transport, orientations.

The {\bf Morse boundary operator}
is defined on oriented critical points by
\begin{equation*}
\begin{split}
     \p^M_k=\p^M_k(V,v_a)
     :\CM_k^a(\Ss_V)
   &\to\CM_{k-1}^a(\Ss_V)
   \\
     \langle x\rangle
   &\mapsto
     \sum_{y\in\Crit_{k-1}}\sum_{u\in m_{xy}}
     u_*\langle x\rangle.
\end{split}
\end{equation*}
By~(\ref{eq:or-ds-2})
this definition respects the
relations~(\ref{eq:relation}).
Extend $\p^M_k$ by linearity.
\begin{theorem}\label{thm:d^2=0}
It holds that $\p^M_{k-1}\circ\p^M_k=0$
for every integer $k$.
\end{theorem}
\begin{proof}
Theorem~\ref{thm:Morse=triple}.
\end{proof}

\subsubsection*{Morse homology}
Assume $\Ss_V$ is Morse and $a\in\R$
is a regular value. For $v_a\in\Oo^a_{reg}$
define heat flow Morse homology of
the perturbed action by
\begin{equation}\label{eq:Morse-homology-1}
     \HM_k^a(\Lambda M,\Ss_{V+v_a})
     :=\frac{\ker \p^M_k}
     {\im \p^M_{k+1}}
\end{equation}
for every integer $k$. In~(\ref{eq:jhihi}) we
will establish isomorphisms
\begin{equation}\label{eq:gy6g5}
     \HM_*^a(\Lambda M,\Ss_{V+v})
     \cong 
     \mathrm{H}_*(\{\Ss_{V+v}\le a\})
     \cong
     \mathrm{H}_*(\{\Ss_{V}\le a\})
\end{equation}
for every $v\in\Oo^a_{reg}$ and where
the second isomorphism is natural
in $v\in\Oo^a$. Moreover,
given regular values $a<b$ and
a perturbation
$v\in\Oo^a_{reg}\cap\Oo^b_{reg}$,
the isomorphisms~(\ref{eq:gy6g5}) commute
with the inclusion induced
homomorphisms; see~(\ref{eq:gy656g}).
Throughout singular homology $\Ho_*$ is
taken with integer coefficients, unless
mentioned otherwise.

\begin{definition}
\label{def:Morse-homology}
{\bf\boldmath Heat semi-flow
homology below level $a$} of the Morse
function $\Ss_V:\Lambda M\to\R$
is defined by
$$
     \HM_*^a(\Lambda M,\Ss_{V})
     :=\HM_*^a(\Lambda M,\Ss_{V+v})
$$
where $v\in\Oo^a_{reg}$. By~(\ref{eq:gy6g5})
this definition does not depend on the
perturbation $v$
(which even leaves all critical
points including neighborhoods untouched;
cf.~(\ref{eq:action-fct})).
\end{definition}

The following result was announced
in~\cite[Thm.~A.7]{salamon:2006a}.

\begin{theoremABC}\label{thm:main}
Assume $\Ss_V$ is Morse and $a$ is either
a regular value of $\Ss_V$ or equal to infinity.
Then there is a natural isomorphism
$$
     \HM_*^a(\Lambda M,\Ss_V;R)
     \cong
     {\rm H}_*(\Lambda^a M;R)
$$
for every principal ideal domain $R$.
If $M$ is not simply connected,
then there is a separate isomorphism for
each component of the loop space.
The isomorphism commutes with the
homomorphisms
$
     \HM_*^a(\Lambda M,\Ss_V)
     \to
     \HM_*^b(\Lambda M,\Ss_V)
$
and
$
     {\rm H}_*(\Lambda^a M)
     \to
     {\rm H}_*(\Lambda^b M)
$
for $a<b$.
\end{theoremABC}

\subsection{Morse filtrations and natural isomorphism}
\label{sec:intro-filtration-iso}

Theorem~\ref{thm:main} relates a purely
topological object with one whose construction
relies heavily on analysis and geometry.
Thus it is a natural idea to look for a family
of intermediate objects -- all encoding the
same homology --
which is flexibel enough so one is able to
relate some member to the Morse side.
A good choice for the family are
cellular filtrations of a topological
space. Indeed by~\cite[V \S 1]{dold:1995a}
cellular homology relates naturally to
singular homology.
This idea was applied successfully
already by Milnor~\cite{milnor:1965a}
in finite dimensions and, more recently, for
flows on Banach manifolds by Abbondandolo
and Majer~\cite{abbondandolo:2006a}.

\begin{definition}\label{def:cellular-filtration}
A sequence of subspaces 
{\boldmath $
     \Ff(\Lambda)
     =\left( F_k\right)_{k\in\Z}
$}
of a topological space $\Lambda$
is called a 
{\bf cellular filtration of \boldmath$\Lambda$}
if
\begin{enumerate}
\item[(i)]
     $F_k\subset F_{k+1}$ for every $k\in\Z$;
\item[(ii)]
     every singular simplex in $\Lambda$
     is a simplex in $F_k$ for some $k$;
\item[(iii)]
     relative singular homology
     ${\rm H}_\ell(F_k,F_{k-1})$ vanishes
     whenever $\ell\not= k$.
\end{enumerate}
\end{definition}

The {\bf cellular complex}
$
     \Co\Ff (\Lambda)
     =(\Co_*\Ff(\Lambda),\p_*^{trip})
$
of a cellular filtration $\Ff(\Lambda)=(F_k)_{k\in\Z}$
of a topological space $\Lambda$
consists of the {\bf cellular chain groups}
$$
     \Co_k\Ff(\Lambda)
     :=\Ho_k(F_k,F_{k-1})
$$
and the {\bf cellular boundary operator}
$$
     \p_k^{trip}:\Co_k\Ff(\Lambda)
     \to\Co_{k-1}\Ff(\Lambda)
$$
given by the connecting homomorphism
in the homology sequence
of the triple $(F_k,F_{k-1},F_{k-2})$.
In fact, the triple boundary operator is
the composition
\begin{equation}\label{eq:triple-boundary}
     \p_k^{trip}:
     \Ho_k(F_k,F_{k-1})
     \stackrel{\p\:}{\longrightarrow}
     \Ho_{k-1}(F_{k-1})
     \stackrel{j_*\:}{\longrightarrow}
     \Ho_{k-1}(F_{k-1},F_{k-2})
\end{equation}
of the connecting homomorphism $\p$
associated to the pair $(F_k,F_{k-1})$
and the quotient induced homomorphism $j_*$
associated to the pair $(F_{k-1},F_{k-2})$.
It is well known that {\bf cellular homology
\boldmath $\Ho_*\Ff(\Lambda)$},
that is the homology associated to the
cellular complex, is naturally\footnote{
  {\bf Natural} in the usual sense that
  these isomorphisms commute
  with the homomorphisms induced
  by {\bf cellular maps}, that is continuous
  maps $f:\Lambda\to\Lambda^\prime$
  such that $f(F_k)\subset F_k^\prime$ $\forall k$.
  }
isomorphic to singular homology of the
topological space
$\Lambda$ itself; see 
e.g.~\cite[Section~V.1]{dold:1995a} or~\cite{milnor:1965a}.
\begin{definition}\label{def:Morse-filtration}
A cellular filtration
$\Ff^a=\left( F_k\right)_{k\in\Z}$ of $\Lambda^a M$
is called a {\bf\boldmath Morse filtration
associated to the action
$\Ss_\Vv$ on $\Lambda^a M$}
if each relative homology group
$\Ho_k(F_k,F_{k-1})$ is generated by
(the classes of appropriate disks $D^u_x$ contained in)
the unstable manifolds of the
critical points of Morse index $k$
and, in addition, every $x\in\Crit^a_k$ lies in
$F_k\setminus F_{k-1}$. Consequently
$F_k\cap\Crit^a=\Crit^a_{\le k}$.
\end{definition}

Observe that for a Morse filtration
$\Ho_\ell(F_k,F_{k-1})$ is isomorphic
to $\Z^{\Crit^a_k}$, if $\ell=k$, although not
naturally and it is trivial, otherwise. 
By $a_k$ we denote the {\bf positive generator}
of $\Ho_k(\D^k,\SS^{k-1})$, that is the class
$[\D^k_{\langle \mathrm{can}\rangle}]$ of the unit
disk equipped with the canonical orientation;
see Definition~\ref{def:canonical-orientation}.

\begin{theoremABC}[Morse filtration and natural
isomorphism]\label{thm:cellular-filtration}
\mbox{ }
\begin{enumerate}
\item[a)]
Consider the Morse-Smale function $\Ss_\Vv$ on
$\Lambda^a M$ given by~(\ref{eq:action-fct}).
There exists an associated Morse filtration,
namely the sequence of subsets
$\Ff(\Lambda^a M)=\left( F_k\right)$
defined by~(\ref{eq:F_k}--\ref{eq:F_m}).
Furthermore, for every regular value $b\le a$
there is a Morse filtration
$\Ff(\Lambda^b M)=\left( F_k^b\right)$
such that the inclusion map $\iota:\Lambda ^b M
\hookrightarrow\Lambda^a M$ is cellular.

\item[b)] 
Let $\Ff^a=\Ff(\Lambda^a M)$ be given by~a).
Pick an integer $k\ge 0$ and a (finite) list
$\vartheta=(\vartheta^x)$ of diffeomorphisms
$\vartheta^x:(\D^k,\Ss^{k-1})\to (D^u_x,S^u_x)$
between the unit disk and certain descending
disks $D^u_x$,
see~(\ref{eq:vartheta-2}), one for each
$x\in\Crit^a_k$. Then there is an isomorphism
$\Theta_k$ 
determined by
\begin{equation}\label{eq:Theta-k}
\begin{split}
     \Theta_k=\Theta^a_k(\vartheta):
     \CM_k^a(\Ss_\Vv)
   &\to 
     \Ho_k(F_k,F_{k-1}) =\Co_k\Ff^a
   \\
     \langle x\rangle
   &\mapsto\bar\vartheta^x_*
     (\sigma_{\langle x\rangle} a_k)
     =[D^u_{\langle x\rangle}]
\end{split}
\end{equation}
where $\bar\vartheta^x:
\D^k\stackrel{\vartheta^x}{\cong} D^u_x
\stackrel{\iota}{\hookrightarrow} N_x
\stackrel{\iota^x}{\hookrightarrow} N_k
\stackrel{\iota}{\hookrightarrow} F_k$
denotes the diffeomorphism
composed with inclusions, 
cf.~(\ref{eq:nat-iso}). The sign
$\sigma_{\langle x\rangle}$ of $\vartheta^x$
is defined by~(\ref{eq:sign-vartheta})
and $D^u_{\langle x\rangle}$ denotes the disk
$D^u_x$ oriented by $\langle x\rangle$;
see Figure~\ref{fig:fig-D_u_x}
and~(\ref{eq:vartheta-or}).
\end{enumerate}
\end{theoremABC}

The main point of
Theorem~\ref{thm:cellular-filtration}
is existence of a Morse filtration. The proof
in section~\ref{sec:Morse filtration}
is constructive and relies on the following
key properties.
\begin{itemize}
\item[\bf (F1)]
          Finite Morse index
\item[\bf (F2)]
          $\Ss_\Vv$ is bounded below
\item[\bf (F3)]
          $\Ss_\Vv$ satisfies the Palais-Smale condition
\item[\bf (F4)]
          Morse-Smale on
          neighborhoods (Lemma~\ref{le:34})
\item[\bf (SF1)]
          Suitable definition of a Conley pair
          $(N_x,L_x)$ for every critical point
\item[\bf (SF2)]
          Taking pre-images $(\varphi_s)^{-1}$ substitutes
          non-existing backward flow $\varphi_{-s}$
\end{itemize}
For an overview of the construction
of the Morse filtration
we refer to our survey~\cite{weber:2014b}
in which we also discuss related
previous work~\cite{abbondandolo:2006a}
of Abbondandolo and Majer. For instance,
once one has a Morse filtration
the proof of the following result
is essentially based on their arguments.

\begin{theorem}\label{thm:Morse=triple}
Let the Morse filtration $\Ff^a$
associated to the Morse-Smale function $\Ss_\Vv$
and the isomorphisms
$\Theta_k:\CM_k^a(\Ss_\Vv)\to\Co_k\Ff^a$
be as in Theorem~\ref{thm:cellular-filtration},
then
\begin{equation*}
     \bigl(\p_k^{trip}\circ\Theta_k\bigr)
     \langle x\rangle
     =\sum_{y\in\Crit^a_{k-1}}\sum_{u\in m_{xy}}
     \bar\vartheta^{u(\infty)}_*
     \left(\sigma_{u_*\langle x\rangle} a_{k-1}\right)
     =\bigl(\Theta_{k-1}\circ\p^M_k\bigr)
     \langle x\rangle
\end{equation*}
for every oriented critical point
$\langle x\rangle$, where
$\bar\vartheta^{u(\infty)}_*
\left(\sigma_{u_*\langle x\rangle} a_{k-1}\right)
=\Theta_{k-1}\left( u_*\langle x\rangle\right)$.
\end{theorem}

\subsection{Stable foliations for Conley pairs}
\label{sec:intro-stab-fol}

The proof that the filtration $\Ff^a=(F_k)$
defined by~(\ref{eq:F_k}--\ref{eq:F_m})
is \emph{Morse} hinges on two properties
of the subsets $F_k\subset \Lambda^a M$:
\emph{openness} and \emph{semi-flow invariance}.
Suppose $F_0\subset\Lambda M$ is
open and semi-flow invariant and consider,
for instance, a local sublevel set about some
nondegenerate local minimum $y$. Then
the pre-image ${\varphi_s}^{-1} F_0$ is open
by continuity of the time-$s$-map. It is also
semi-flow invariant, because $F_0$ is.
Now suppose $x$ is a nondegenerate
critical point of Morse index one.
Its unstable manifold connects to such $y$.
The problem is that~$x$, although approximated
for large $s$, will never be included in the pre-image.
Now the basic idea of Conley theory~\cite{conley:1978a}
enters, namely the notion of an isolating neighborhood
$N$ with exit set $L$.
Suppose $N_x$ is an open neighborhood
of $x$ which admits a subset $L_x$
through which any trajectory leaving
$N_x$ has to go first. Suppose further
that there is some large time $T$ such
that the pre-image ${\varphi_T}^{-1} F_0$
contains $L_x$.
Then the union ${\varphi_s}^{-1} F_0 \cup N_x$
has both desired properties.

\begin{figure}
  \centering
  \includegraphics{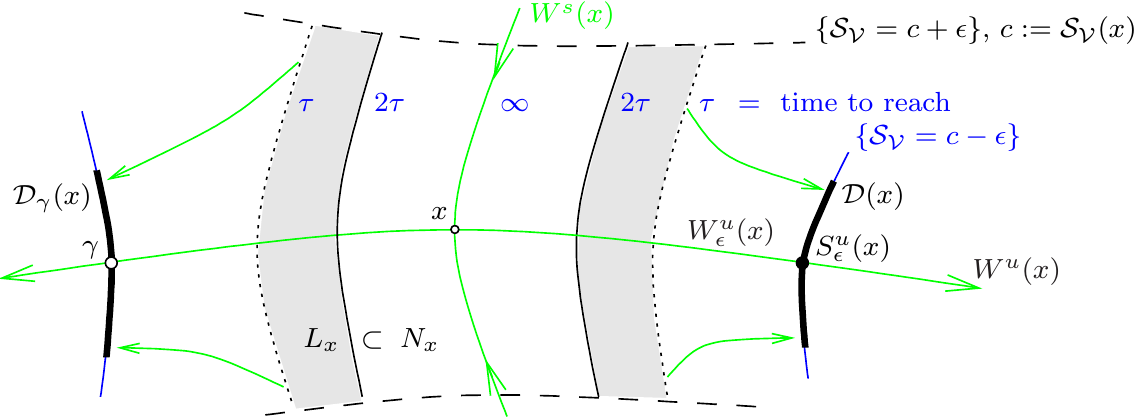}
  \caption{Conley pair $(N_x,L_x)$ for
           critical point $x$
           }
  \label{fig:fig-Conley-pair}
\end{figure}

\begin{definition}\label{def:index-pair}
A {\bf\boldmath Conley pair $(N,L)$
for a critical point $x$} of $\Ss_\Vv$
consists of an open subset $N\subset\Lambda M$
and a closed subset $L\subset N$ which satisfy
\begin{enumerate}
  \item[(i)]
     $x\in N\setminus L$
  \item[(ii)]
     $\cl N\cap\Crit \Ss_\Vv=\{x\}$
  \item[(iii)]
     $\gamma\in L$ and $\varphi_{[0,s]}\gamma\subset N
     \,\Rightarrow\,
     \varphi_s\gamma\in L$
  \item[(iv)]
     $\gamma\in N$ and 
     $\varphi_T\gamma\notin N
     \,\Rightarrow\,
     \exists \sigma\in(0,T):
     \varphi_\sigma\gamma\in L$ and
     $\varphi_{[0,\sigma]}\gamma\subset N$
\end{enumerate}
In particular,
conditions~(i) and~(ii) tell that $N$
is an open neighborhood of $x$ which
contains no other critical points in its closure.
Condition~(iii) says that {\bf\boldmath$L$
is positively invariant in $N$}
and~(iv) asserts that
every semi-flow line which leaves
$N$ goes through $L$ first.
Hence we say that {\bf\boldmath $L$
is an exit set of $N$}.
\end{definition}

Given a nondegenerate critical point $x$ of
$\Ss_\Vv$, set $c:=\Ss_\Vv(x)$. Borrowing from
finite dimensions~\cite{salamon:1990a}
we define the two sets
\begin{equation}\label{eq:Conley-set-NEW}
\begin{split}
     N_x=N_x^{\eps,\tau}
   :&=
     \left\{\gamma\in\Lambda M\mid
     \text{$\Ss_\Vv(\gamma)<c+\eps$,
     $\Ss_\Vv(\varphi_\tau\gamma)>c-\eps$}
     \right\}_x ,
\end{split}
\end{equation}
where $\{\ldots\}_x$ denotes the
{\bf path connected component}
that contains $x$, and
\begin{equation}\label{eq:L_x}
     L_x=L_x^{\eps,\tau}
     :=\{\gamma\in N_x\mid
     \Ss_\Vv(\varphi_{2\tau}\gamma)\le c-\eps\}.
\end{equation}
Note that $L_x$ is a relatively closed subset of the open
subset $N_x$ of $\Lambda M$.

\begin{theorem}[Conley pair]\label{thm:Conley-pair}
The pair $(N_x,L_x)$ defined
by~(\ref{eq:Conley-set-NEW}-\ref{eq:L_x})
is a Conley pair for the nondegenerate
critical point $x$ for all $\eps>0$ small
and $\tau>0$ large.
\end{theorem}

  Theorem~\ref{thm:Conley-pair}
  holds for all $\eps\in(0,\mu]$ and $\tau>\tau_0$
  with $\mu$ and $\tau_0$
  as in~(H4) of
  Hypothesis~\ref{hyp:local-setup-N}.
  In this case all ascending/descending
  disks $W^{s,u}_\eps$ and spheres
  $S^{s,u}_\eps$ are manifolds.

Figure~\ref{fig:fig-Conley-pair} shows a typical
Conley pair, illustrates the exit set property
of $L_x$, and indicates
hypersurfaces which are characterized by the
fact that each point reaches the level set
$\{\Ss_\Vv=c-\eps\}$ in the same time.
The points on the stable manifold
never reach level $c-\eps$, so they are
assigned the time label $\infty$. By the
Backward $\lambda$-Lemma~\cite{weber:2014b}
locally near $x$ these hypersurfaces fiber
over descending disks into diffeomorphic
copies of the local {\bf stable} manifold.
This provides a foliation of small
neighborhoods of $x$ the leaves of which,
apriori, have no global meaning.
It is the main content of
Theorem~\ref{thm:inv-fol}
to express such neighborhoods and leaves in
terms of (globally defined) level sets of the
 action functional. The
difficulty being infinite dimension. Concerning
the naming {\bf invariant stable foliation}
note the boldface 'stable' above and~a)
below, whereas {\bf invariant} refers to~b).
Parts~c) and~d) are quite useful as they
allow to contract $N_x$ onto the ascending
disk or even fit $N_x$ into any given
neighborhood of $x$.

\begin{theoremABC}[Invariant stable foliation]
\label{thm:inv-fol}
Pick a nondegenerate critical point $x$
of $\Ss_\Vv$ and set $c:=\Ss_\Vv(x)$.
Then for every small $\eps>0$ the following is true.
Consider the {\bf descending sphere} and the
{\bf descending disk} given by
\begin{equation}\label{eq:Desc-sphere}
     S^u_\eps(x):=W^u(x)\cap\{\Ss_\Vv=c-\eps\}
     ,\quad
     W^u_\eps(x):=W^u(x)\cap\{\Ss_\Vv>c-\eps\}.
\end{equation}
Pick a tubular neighborhood $\Dd(x)$
(associated to a radius $r$ normal disk bundle)
over $S^u_\eps(x)$ in the level hypersurface
$\{\Ss_\Vv=c-\eps\}$.
Denote the fiber over
$\gamma\in S^u_\eps(x)$ by $\Dd_\gamma(x)$;
see Figure~\ref{fig:fig-Conley-pair}.
Then the following holds
for every large $\tau>0$.\footnote{
  Hypothesis~\ref{hyp:local-setup-N} (H4) specifies
  the precise ranges of $\eps$ and $\tau$.
  }
\begin{enumerate}
\item[a)]
  The set $N_x=N_x^{\eps,\tau}$
  defined by~(\ref{eq:Conley-set-NEW})
  contains in its closure no critical points
  except $x$.
  Moreover, it carries the
  structure of a codimension-$k$
  foliation\footnote{
    For the precise degree of smoothness we refer
    to the backward
    $\lambda$-Lemma~\cite[Thm.~1]{weber:2014a}.
    }
 whose leaves are parametrized by the $k$-disk
  $\varphi_{-\tau} W^u_\eps(x)$ where $k$ is
  the Morse index of $x$.
  The leaf $N_x(x)$ over $x$ is the
  ascending disk $W^s_\eps(x)$. The other leaves
  are the codimension-$k$ disks given by
  \begin{equation*}
     N_x(\gamma_T)
     =\left({\varphi_T}^{-1}\Dd_\gamma (x)\cap
     \{\Ss_\Vv<c+\eps\}\right)_{\gamma_T}
     ,\qquad
     \gamma_T:=\varphi_{-T}\gamma,
  \end{equation*}
  whenever $T>\tau$ and $\gamma\in S^u_\eps(x)$.
\item[b)]
  Leaves and semi-flow are compatible in the sense that
  \begin{equation*}
     z\in N_x(\gamma_T)
     \quad\Rightarrow\quad
     \varphi_\sigma z\in N_x(\varphi_\sigma\gamma_T)
     \quad \forall \sigma\in[0,T-\tau).
  \end{equation*}
\item[c)]
  The leaves converge uniformly to the ascending disk
  in the sense that
  \begin{equation}\label{eq:unif-exp-dist}
     \dist_{W^{1,2}}\left(N_x(\gamma_T),W^s_\eps(x)\right)
     \le e^{-T\frac{\lambda}{16}}
  \end{equation}
  for all $T>\tau$ and $\gamma\in S^u_\eps(x)$;
  see~(H4) below for $\lambda$.
  If $U$ is a neighborhood
  of the closure of $W^s_\eps(x)$ in $\Lambda M$, then
  $N_x^{\eps,\tau_*}\subset U$ for some constant $\tau_*$.
\item[d)]
  Assume $U$ is a neighborhood of $x$ in
  $\Lambda M$. Then there are constants $\eps_*$ and
  $\tau_*$ such that $N_x^{\eps_*,\tau_*}\subset U$.
\end{enumerate}
\end{theoremABC}

\begin{figure}
  \centering
  \includegraphics{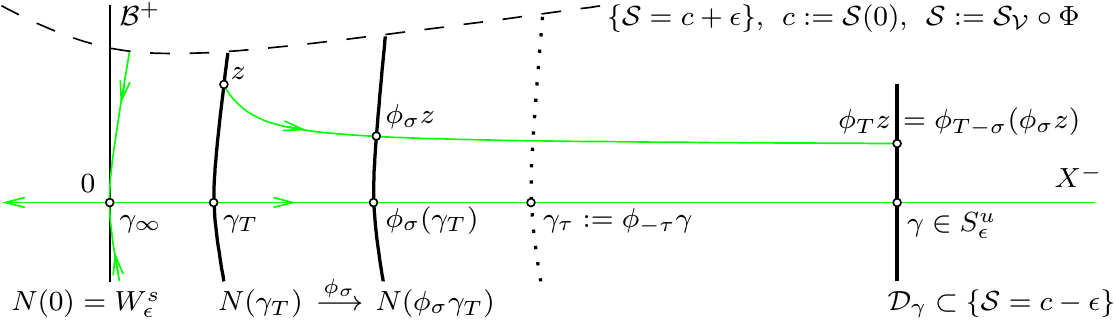}
  \caption{Invariant foliation of
                 $N=N^{\eps,\tau}$
                 in local coordinates of
                 Hypothesis~\ref{hyp:local-setup-N}}
  \label{fig:fig-flowinv-foliation}
\end{figure}

\begin{theoremABC}[Strong deformation retract]\label{thm:deformation-retract}
Pick one of the Conley pairs $(N_x,L_x)$
in Theorem~\ref{thm:Conley-pair} and abbreviate by
$$
     N^u_x:=N_x\cap W^u(x),\qquad
     L^u_x:=L_x\cap W^u(x).
$$
the corresponding parts in the unstable manifold.
Then the pair of spaces $(N_x,L_x)$ strongly
deformation retracts to $(N_x^u,L_x^u)$.
Moreover, the latter pair consists of an
open disk whose dimension $k$ is the Morse index
of $x$ and an annulus which arises by removing a
smaller open disk from the larger one. 
\end{theoremABC}

\begin{corollary}\label{cor:conley-index-N-L}
Given a Conley pair $(N_x,L_x)$
as in Theorem~\ref{thm:Conley-pair}, then
\begin{equation}\label{eq:N-L}
   \Ho_\ell(N_x,L_x)\cong
   \begin{cases}
       \Z&,\ell=\IND_\Vv(x),
       \\
       0&,\text{otherwise.}
   \end{cases}
\end{equation}
\end{corollary}

\begin{proof}
Isomorphism~(\ref{eq:nat-iso}).
\end{proof}

The task to prove~(\ref{eq:N-L}) triggered the
discovery of the Backward
$\lambda$-Lemma in~\cite{weber:2014b}.
  Luckily it was afterwards that we have been
  informed by Kell~\cite{kell:2012a}
  that~(\ref{eq:N-L}) should follow from
  Rybakowski's theory~\cite{rybakowski:1987a}.
The $\lambda$-Lemma,
therefore Theorem~\ref{thm:inv-fol},
both highly depend on finiteness of the Morse
index. Furthermore, it is the proof of
Theorem~\ref{thm:deformation-retract} in
section~\ref{sec:deformation-Conley-pair} which
requires the extension of the linearized graph
maps in the Backward
$\lambda$-Lemma~\cite{weber:2014b}
from $W^{1,2}$ to $L^2$; see
Remark~\ref{rem:changes}
and~\cite[Rmk.~1]{weber:2013b}.


\subsection{Past and future}
\label{sec:outlook}

The Morse complex goes back to the work of
Thom~\cite{thom:1949a},
Smale~\cite{smale:1960a,smale:1961a}, and
Milnor~\cite{milnor:1965a} in the 40's, 50's and 60's,
respectively.
The geometric formulation in terms of
flow trajectories was re-discovered
by Witten in his influential 1982
paper~\cite{witten:1982a}.
He studied a supersymmetric
quantum mechanical system
related to the Laplacian
$\Delta_s=d_sd_s^*+d_s^*d_s$
which involves the deformed
Hodge differential $d_s=e^{-sf} d e^{sf}$
acting on differential forms.
Here $f:M\to\R$ denotes a Morse function 
on a closed Riemannian manifold $M$ and
$s\ge 0$ is a real parameter.
The Morse complex
arises as the adiabatic limit of the
quantum mechanical system,
as the parameter $s$ tends to infinity.
In the early 90's the details of
the construction
have been worked out, among others,
by Po\'{z}niak~\cite{po-zniak:1991a},
by Schwarz~\cite{schwarz:1993a}
who developed the functional analytic framework,
and by the author~\cite{weber:1993a-link}
who developed the dynamical systems framework.
In the past decade Abbondandolo and
Majer~\cite{abbondandolo:2006a}
extended the Morse complex to flows on
Banach manifolds.

Morse homology for semi-flows
was constructed only recently
in~\cite{weber:2010a-link,weber:2013b}
where the functional analytic (moduli space)
framework has been worked out for the heat flow.
Being based on Sard's theorem,
the theory could be trivial.
The present paper develops the dynamical systems
framework and, above all, establishes non-triviality
of the theory by calculating it in terms of
singular homology.

Key tools are the invariant stable foliations
provided by Theorem~\ref{thm:inv-fol}
which are of independent interest.
For instance, the (non obvious) \emph{global}
stable manifold theorem for \emph{forward}
semi-flows will be a corollary of the main result
of our forthcoming paper~\cite{weber:2014d}
whose base is Theorem~\ref{thm:inv-fol}
together with the pre-image idea
-- in a different guise though --
which founded~\cite{weber:2014a} and
the present text.

An extremely rich source of semi-flows
is obviously geometric analysis. For instance,
although the present theory only deals
with harmonic spheres of dimension one,
it could be a first step in one of various
possible directions.

Returning to present time, consider
the finite dimensional case in which
there is, of course, no need to consider
semi-flow Morse homology. But there are
(too) many choices which one can take
while constructing the Morse complex.
For instance, should one orient stable or
unstable manifolds? Or even $M$ itself?
Should we use the forward or the backward flow?
The heat flow eliminates these questions
alltogether
-- only the unstable manifolds
are of finite dimension and there is no
backward flow in general.
We saw above that one even gets away with
embedded ascending disks $W^s_\eps(x)$,
no manifold structure needed on all of $W^s(x)$.
Furthermore, our construction
of the natural isomorphism to singular homology
applies correspondingly and is new in finite dimensions.

Finite Morse index is one of the most heavily used
ingredients in this paper. Already the Backward
$\lambda$-Lemma~\cite{weber:2014a}
hinges on it via well posedness of the mixed
Cauchy problem. So does existence of the
backward flow on unstable manifolds.
That the action is bounded below and
satisfies the Palais-Smale condition is
also used frequently. The Abbondandolo-Majer
extension of Morse-Smale to
neighborhoods~\cite[Lemma~2.5]{abbondandolo:2006a}
carries over to the present setup and is quite
useful. Remarkably, in the very last step
of our construction suddenly the need for a
\emph{forward} $\lambda$-Lemma
arises; see Figure~\ref{fig:fig-forward-lambda}.

\section{Conley pairs and stable foliations}
\label{sec:stab-fol}
In section~\ref{sec:stab-fol} we study the heat
flow locally near a given nondegenerate critical
point $x$ of $\Ss_\Vv$ of Morse index $k$.
The perturbation $\Vv$ is only required to
satisfy axioms~{\rm (V0)--(V3)}
in the notation of~\cite{weber:2013b}.
Throughout section~\ref{sec:stab-fol}
we use heavily results and notation
of~\cite{weber:2014a}. The reader
may wish to have a copy at hand.

\begin{remark}[Backward flow on unstable manifold]
\label{rmk:diffeo-unstable-NEW}
The unstable manifold $W^u(x)$ carries a
{\bf\boldmath backward flow} $\varphi_{-s}$.
Thus the time-$s$-map $\varphi_s$
restricted to the unstable manifold
is a diffeomorphism of $W^u(x)$ and its inverse
is given by $\varphi_{-s}$.
To see this recall that by definition,
see e.g.~\cite[\S 6.1]{weber:2013a},
each element $\gamma$ of $W^u(x)$
is of the form $u_\gamma(0,\cdot)$
where $u_\gamma:(-\infty,0]\times S^1\to M$
solves the heat equation~(\ref{eq:heat})
and $u_\gamma(s,\cdot)$
converges to $x$, as $s\to-\infty$.
Given $s>0$, obviously
$\varphi_{-s}\gamma:=u_\gamma(-s,\cdot)$ lies in
the pre-image ${\varphi_s}^{-1}(\gamma)$ which
contains no other element by backward unique
continuation~\cite[Thm.~17]{weber:2013a}.
\end{remark}

\subsubsection*{Outline}
In section~\ref{sec:iso-block} we define an open
subset $N_c=N_c^{\eps,\tau}\subset\Lambda M$
associated to a critical value $c$ of the action
and reals $\eps,\tau>0$.
If the action of $x$ is $c$,
then $N_x=N_x^{\eps,\tau}$ is the path
connected component of $N_c^{\eps,\tau}$ that
contains $x$.
Lemma~\ref{le:N_x} asserts that $N_x$
intersects the stable manifold $W^s(x)$
in the ascending disk $W^s_\eps(x)$ and
the descending disk $W^u_\eps(x)$
in the $k$-disk $\varphi_{-\tau} W^u_\eps(x)$.
The inclusions~(\ref{eq:N-inclusions}) suggest
that $N_x$ contracts onto $x$,
as $\eps\to 0$ and
$\tau\to\infty$.
Thus by nondegeneracy of $x$ the closure of
$N_x$ contains no critical point except $x$
whenever $\eps>0$ is sufficiently small
and $\tau>0$ is sufficiently large.
Inspired by Conley~\cite{conley:1978a}
such $N_x$ is called an
isolating block for $x$.

Section~\ref{sec:stab-fol-level-sets} shows
that an isolating block $N_x$ is foliated
by disks diffeomorphic to the ascending disk
$W^s_\eps(x)$ via the graph maps $\Gg^T_\gamma$
and $\Gg^\infty$ provided by the Backward 
$\lambda$-Lemma~{\cite[Thm.~1]{weber:2014a}}
and the Local Stable Manifold
Theorem~{\cite[Thm.~3]{weber:2014a}}.
More precisely, the leaves of the foliation are
parametrized by the elements of the $k$-disk
$\varphi_{-\tau} W^u_\eps(x)$. In particular, the leaf
over its center $x$ 
is the ascending disk $W^s_\eps(x)$.
Furthermore, the heat flow $\varphi_s$ maps
leaves to leaves and the isolating block
$N_x$ contracts onto $W^s_\eps(x)$,
as $\tau\to\infty$.

In section~\ref{sec:deformation-Conley-pair}
we extend the heat flow on the ascending disk
$W^s_\eps(x)$ artificially to the other leaves
of the isolating block $N_x$ using
the diffeomorphisms mentioned in the former
paragraph. This way we prove that the part $N_x^u$
of $N_x$ in the unstable manifold
is a strong deformation retract of $N_x$.
This seems obvious. So why is there a long
calculation? Because we need to make sure that
the deformation takes place \emph{inside} $N_x$
and the dimension of each leaf is infinite.

In section~\ref{sec:Conley-pairs} we introduce
the notion of an exit set $L_x=L_x^{\eps,\tau}$
associated to an isolating block $N_x=N_x^{\eps,\tau}$.
The pair $(N_x,L_x)$ is called a Conley pair and
we state and prove key properties that will be used
in section~\ref{sec:filtration-iso}. In particular
we show that the homology of the pair $(N_x,L_x)$
coincides with the homology of the pair
$(\D^k,\SS^{k-1})$ where $k$ is the Morse index
of $x$ and $\SS^{k-1}$ denotes the boundary
of the closed unit disk $\D^k\subset\R^k$.

\subsubsection*{Local coordinate setup and choices}

\begin{hypothesis}\label{hyp:local-setup-N}
Fix a perturbation $\Vv$ that satisfies
the axioms~{\rm (V0)--(V3)}
in~\cite{weber:2013b} and a nondegenerate
critical point $x$ of $\Ss_\Vv$ of Morse index $k$
and action~$c$.
\begin{enumerate}
\item[(H1)]
We use the local setup of~\cite{weber:2014a},
see Figure~\ref{fig:fig-loc-stab-mf-thm}.
Fix a local parametrization
$$
     \Phi:\exp_x: X\supset\Uu\supset\Bb_{\rho_0} \to\Lambda M,
     \qquad X=T_x\Lambda M=W^{1,2}(S^1,x^*TM),
$$
of a neighborhood of $x$ in $\Lambda M$
and consider the orthogonal splitting
$$
     X=T_x W^u(x)\oplus T_x W^s_\eps(x)
     =X^-\oplus X^+
$$
with corresponding orthogonal
projections $\pi_\pm$.
By a standard argument we
assume that $\Uu$ is of the form
$W^u\times\Oo^+$ where $W^u\subset X^-$
represents the unstable manifold near $x$
and $\Oo^+\subset X^+$ is an open
ball about $0$. The constant $\rho_0>0$
is provided by~\cite[Hyp.~1]{weber:2014a}
and $\Bb_{\rho_0}$ denotes the closed radius
$\rho_0$ ball in $X$ centered at the origin.

By $\phi$ we denote the local
semi-flow on $\Uu$ which represents
the heat flow with respect to
$\Phi$; see~\cite[(5)]{weber:2014a}.
In these coordinates $0\in X$ represents $x$
and $\Ss:=\Ss_\Vv\circ\Phi^{-1}$ the action functional.
In general, our coordinate notation will be the
global notation with $x$ omitted, for example
$W^s_\eps$ abbreviates $\Phi^{-1} W^s_\eps(x)$.
\item[(H2)]
   Due to nondegeneracy of the critical point $x$
   we assume that the radius $\rho_0>0$
   has been chosen sufficiently small such that
   the coordinate patch $\Phi(\Bb_{\rho_0})$ about $x$
   contains no other critical points.
\item[(H3)]
  Fix a constant $\mu>0$ sufficiently small
  such that the ascending disk $W^s_{2\mu}(x)$
  defined by~(\ref{eq:asc-disk}) and the
  descending disk $W^u_{2\mu}(x)$
  defined by~(\ref{eq:Desc-sphere})
  are contained in the coordinate
  patch $\Phi(\Bb_{\rho_0})$
  and such that their closures are diffeomorphic
  to the closed unit disks in $\R^k$ and $X^+$,
  respectively; cf.
  Lemma~\ref{le:descending-disk}
  and Lemma~\ref{le:asc-disk}.
\item[(H4)]
  The following are the hypotheses of 
  Theorem~\ref{thm:inv-fol} which allow
  to apply the Backward 
  $\lambda$-Lemma~\cite[Thm.~1]{weber:2014a}.
  Fix an element $\lambda\in(0,d)$ in the {\bf spectral
  gap}\footnote{
  distance $d$ between zero and the spectrum of
  the Jacobi operator $A_x$ associated to $x$
  }
  of the Jacobi operator $A_x$ associated to $x$.
  Pick $\eps\in(0,\mu]$
  where $\mu$ is the constant in~(H3).
  Choose $r=r(\eps)>0$ sufficiently small such
  that the tubular neighborhood $\Dd(x)$ associated
  to the radius $r$ normal disk bundle of the
  descending sphere $S^u_\eps(x)$ in the level
  hypersurface $\{\Ss_\Vv=c-\eps\}$ of the Hilbert
  manifold $\Lambda M$ exists and is contained
  in the coordinate patch $\Phi(\Bb_{\rho_0})$.
  Denote the fiber over $\gamma\in S^u_\eps(x)$
  by $\Dd_\gamma(x)$; see
  Figures~\ref{fig:fig-Conley-pair}
  or, in coordinates, 
  Figure~\ref {fig:fig-flowinv-foliation}.
  Then there is a constant
  $\tau_0=\tau_0(\eps,r,\lambda)>0$ such that
  the assertions of Theorem~\ref{thm:inv-fol}
  hold true whenever $\tau>\tau_0$.
\end{enumerate}
\end{hypothesis}

\subsection{Isolating blocks}
\label{sec:iso-block}
As some results in this section do not require
nondegeneracy we use the notation $y$
for arbitrary critical points of $\Ss_\Vv$.
In contrast $x$ always denotes the nondegenerate
critical point that has been fixed at the very
beginning of section~\ref{sec:stab-fol}.

\begin{definition}\label{def:isol-block}
Assume $\eps>0$ and $\tau>0$ are constants.
\begin{enumerate}
\item[\rm (a)]
  Given a critical value $c$ of
  the action functional $\Ss_\Vv$
  consider the set\,\footnote{We borrow
     definition~(\ref{eq:N_c})
     from the finite dimensional
     situation~\cite[p.~119]{salamon:1990a}.
  }
  \begin{equation}\label{eq:N_c}
  \begin{split}
     N_c=N_c^{\eps,\tau}
   :&=
     \left\{\gamma\in\Lambda M\mid
     \text{$\Ss_\Vv(\gamma)<c+\eps$,
     $\Ss_\Vv(\varphi_\tau\gamma)>c-\eps$}
     \right\}
   \\
    &=\{\Ss_\Vv<c+\eps\}\cap
     \varphi^{-1}_{(\tau,\infty]}\{\Ss_\Vv=c-\eps\}
  \end{split}
  \end{equation}
  where by definition
  $\varphi^{-1}_\infty\{\Ss_\Vv=c-\eps\}$
  denotes those points of $\Lambda M$
  above action level $c-\eps$
  which never reach that
  level.\,\footnote{If $\Ss_\Vv$ is Morse below level
     $c+\eps$
     then $N_c^{\eps,\tau}=\cup_y W^s_\eps(y)$
     where the union is over all critical
     points $y$ whose action lies in the interval
     $(c-\eps,c+\eps)$.
     (In this case there are no limit cycles.)
    }

\item[\rm (b)]
  Suppose $y$ is a critical point of action
  $c=\Ss_\Vv(y)$. By $N_y=N_y^{\eps,\tau}$ we
  denote the path connected component of
  $N_c^{\eps,\tau}$ that contains $y$;
  compare~(\ref{eq:Conley-set-NEW}).

\item[\rm (c)]
  Suppose $x$ is a nondegenerate critical point
  and there are no other critical points in the
  closure of $N_x^{\eps,\tau}$. Then $N_x^{\eps,\tau}$
  is called  an {\bf isolating block}.
\end{enumerate}
\end{definition}
Figure~\ref{fig:fig-N-components}
shows a set $N_c$ that consists of three path
connected components
one of which is an isolating block.

\begin{figure}
  \centering
  \includegraphics{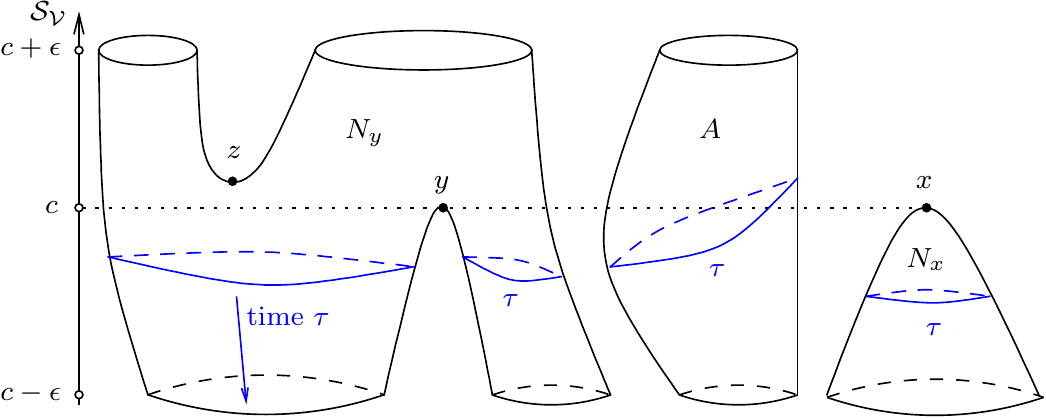}
  \caption{A set $N_c$
                 with three path connected components
                 $N_y$, $A$, $N_x$
                 }
  \label{fig:fig-N-components}
\end{figure}

\begin{lemma}\label{le:Conley-set}
The set $N_c^{\eps,\tau}$ defined by~(\ref{eq:N_c})
is an open subset of
$\Lambda^{c+\eps} M$ and contains all
critical points with action values in the interval
$(c-\eps,c+\eps)$.
\end{lemma}

\begin{proof}
Openness is due to continuity of the action
functional $\Ss_\Vv$ and Lipschitz continuity
of the time-$s$-map $\varphi_s$ when restricted to
sublevel sets. The latter follows from a mild extension
of~\cite[Thm.~9.1.5]{weber:2010a-link};
see~\cite{weber:heat-book}.
The second assertion is true since
critical points of $\Ss_\Vv$ and fixed points
of $\varphi_s$ coincide.
\end{proof}

\begin{lemma}[Descending disks]
\label{le:descending-disk}
Given a nondegenerate critical point
$x$ of $\Ss_\Vv$, there is a constant 
$\eps_0>0$ such that the following is true.
For each $\eps\in(0,\eps_0]$ the closure
of the descending disk $W^u_\eps(x)$ defined
by~(\ref{eq:Desc-sphere}) is
diffeomorphic to the closed unit disk in $\R^k$ where 
$k$ is the Morse index of $x$.
Furthermore, any open neighborhood $U$ of $x$
in the unstable manifold $W^u(x)$
contains the closure of some
descending disk $W^u_\eps(x)$.
\end{lemma}

\begin{proof}
Unstable Manifold 
Theorem~\cite[Thm.~18]{weber:2013b}
and Morse-Lemma~\cite{milnor:1963a}.
\end{proof}

\begin{lemma}\label{le:N_x}
Assume $N_y^{\eps,\tau}$ is given by
Definition~\ref{def:isol-block}~(b), then
\begin{equation}\label{eq:N-inclusions}
\begin{split}
     \delta<\eps\;\;
     \Rightarrow\;\;
     N_y^{\delta,\tau}\subset N_y^{\eps,\tau},
     \qquad
     T>\tau\;\;
     \Rightarrow\;\;
     N_y^{\eps,T}\subset N_y^{\eps,\tau}.
\end{split}
\end{equation}
Assume $x$ is a nondegenerate critical point
of $\Ss_\Vv$, then
\begin{equation}\label{eq:N-cap-W}
\begin{split}
     N_x^{\eps,\tau}\cap W^s(x)
    &=W^s_\eps(x),
   \\
     N_x^{\eps,\tau}\cap W^u(x)
    &=\varphi_{-\tau} W^u_\eps(x)\\
    &=\{ x\}\cup
       \bigcup_{T>\tau}\varphi_{-\tau} S^u_\eps(x).
\end{split}
\end{equation}
for every $\eps\in(0,\eps_0]$ where $\eps_0$
is given by the descending disk
Lemma~\ref{le:descending-disk}.
\end{lemma}

\begin{proof}
The first inclusion in~(\ref{eq:N-inclusions})
is trivial and the second one follows from the
fact that the action does not increase along
heat flow trajectories.

Consider the first identity
in~(\ref{eq:N-cap-W}). Since
$W^s_\eps(x):=W^s(x)\cap\{\Ss_\Vv<c+\eps\}$
the inclusion ``$\subset$'' is trivial.
To see ``$\supset$'' note that
$W^s_\eps(x)$ is a subset of $N_c$.
Given $\gamma\in W^s_\eps(x)$ 
the trajectory $\varphi_{[0,\infty]}\gamma$
connects $\gamma$ and $x$ in
$W^s_\eps(x)$, hence in $N_c$.
Thus $\gamma$ lies in the component
of $N_c$ that contains $x$.

Recall that
$W^u_\eps(x):=W^u(x)\cap\{\Ss_\Vv>c-\eps\}$.
By flow invariance of the unstable
manifold
$
     \varphi_{-\tau} W^u_\eps(x)
     =W^u(x)\cap
     \{z\in\Lambda M\mid
     \Ss_\Vv(\varphi_\tau z)>c-\eps\}
     \subset N_c
$.
Now the second identity
in~(\ref{eq:N-cap-W}) follows by a similar
argument as the first identity, just
use backward trajectories.
To see the third identity observe that
any flow trajectory in
$W^s(x)\setminus\{ x\}$ hits
$S^u_\eps(x)$ precisely once.
Obviously $W^u_\eps(x)$ is diffeomorphic
to its image under the diffeomorphism
$\varphi_{-\tau}$ of $W^u(x)$.
On the other hand, it is
diffeomorphic to the open unit disk in $\R^k$
by the descending disk
Lemma~\ref{le:descending-disk}
where $k$ denotes the Morse index of $x$.
\end{proof}

\begin{remark}[Open problem]\label{rem:shrink-N}
The inclusions~(\ref{eq:N-inclusions}) suggest
that one could fit $N_x$ into any given
neighborhood of $x$ by choosing
$\eps>0$ sufficiently small\footnote{
  so the ascending disk
  $W^s_\eps(x)$ contracts
  to $x$ by the Palais-Morse Lemma
}
and $\tau>0$ sufficiently large.\footnote{
  so $N_x^{\eps,\tau}$ contracts to $W^s_\eps(x)$
  by the Backward 
  $\lambda$-Lemma~{\cite[Thm.~1]{weber:2014a}}
}
By Theorem~\ref{thm:inv-fol} part~(d)
this is indeed possible.
Can this also be achieved by shrinking only $\eps$?
\end{remark}

\subsection{Stable foliations associated to
level sets} \label{sec:stab-fol-level-sets}

\subsubsection*{Local non-intrinsic foliation}
Assume~(H1) and~(H2) of
Hypothesis~\ref{hyp:local-setup-N}.
We start with an investigation of the foliation
property provided by the Backward
$\lambda$-Lemma~{\cite[Thm.~1]{weber:2014a}
for a disk family
$\Dd=S^u_\eps\times B_\kappa^+\subset\Bb_{\rho_0}$,
not necessarily related to level sets, but
which still has the no {\bf return property}
with respect to the local flow $\phi$, that is
$$
     \Dd\cap{\phi_s}^{-1}\Dd=\emptyset
$$
for all $s>0$ for which $\phi$ is defined.

\begin{corollary}[to the Backward
$\lambda$-Lemma~{\cite[Thm.~1]{weber:2014a}}]
\label{cor:inv-fol}
Given~(H1) and~(H2), the assumptions
of~\cite[Thm.~1]{weber:2014a},
and the additional assumption that $(\Dd,\phi)$ has
the no return property, then the following is true.
Let $\Gg,\Gg^\infty:\Bb^+\to X$ be the graph maps
provided by Theorems~1 and~3
in~\cite{weber:2014a}, respectively. Then the subset
$$
     F=F^{\eps,T_0}
     :=\left(\im\Gg\cup\im\Gg^\infty\right)
     \subset\Bb_{\rho_0}\subset\Uu
$$
of the Banach space $X$ carries the structure of a
codimension $k$ foliation; see
  Figure~\ref{fig:fig-flowinv-foliation}
  for the part $N$ of $F$ below level $c+\eps$.
The leaves
are given by the subset $F(0):=\Gg^\infty(\Bb^+)$ of the
local stable manifold $W^s(0,\Uu)$,
defined in Lemma~\ref{le:asc-disk},
and by the graphs
$F(\gamma_T):=\Gg^T_\gamma(\Bb^+)$ for all $T>T_0$
and $\gamma\in S^u_\eps$. Leaves and semi-flow are
compatible in the sense that
$$
     z\in F(\gamma_T)
     \quad\Rightarrow\quad
     \phi_\sigma z\in F(\phi_\sigma\gamma_T)
     \qquad,\gamma_T
     :=\phi_{-T}\gamma=\Gg^T_\gamma(0),
$$
whenever the semi-flow trajectory from $z$ to 
$\phi_\sigma z$ remains inside $F$.
\end{corollary}

\begin{proof}[Proof of Corollary~\ref{cor:inv-fol}]
Assume that the leaves $F(\gamma_T)$
and $F(\beta_S)$ are disjoint whenever
$\gamma_T\not=\beta_S$. Then the Lipschitz
continuous $C^1$ maps $\Gg^T_\gamma:\Bb^+\to X$
and $\Gg^\infty:\Bb^+\to X$
endow $F$ with the structure
of a codimension $k$ foliation.
\\
To prove the assumption suppose
$(T,\gamma)\not= (S,\beta)$.
Because $T\ge T_0\ge T_1$, the endpoint
conditions~\cite[(21)]{weber:2014a}
are satisfied by the
choice of $T_1$ in~\cite[(19)]{weber:2014a}.
Assume by contradiction that 
$\Gg^T_\gamma(z_+)=\Gg^S_\beta(z_+)=:z$ for some 
$z_+\in\Bb^+$. Then by~\cite[(31)]{weber:2014a} the
point $z$ is the initial value of a heat flow 
trajectory $\xi^T$ ending at time $T$ on the fiber
$\Dd_\gamma$ and also of a heat flow trajectory $\xi^S$
ending at time $S$ on $\Dd_{\beta}$.
By uniqueness of the solution to the Cauchy 
problem~\cite[(5)]{weber:2014a}
with initial value $z$
the two trajectories coincide until time $\min\{T,S\}$.
If $T=S$, then $\gamma=\beta$ and we are done. Now 
assume without loss of generality that $T<S$, 
otherwise rename. Hence $\xi^S$ meets $\Dd_\gamma$
at time $T$ and $\Dd_\beta$ at the later time $S$.
But this contradicts the no return property of $\Dd$.

We prove compatibility of leaves and
semi-flow. The fixed
point $0$ is semi-flow invariant. Its neighborhood
$F(0)$ in the local stable manifold is trivially
semi-flow invariant in the required sense, namely up to
leaving $F(0)$.
Pick $z\in F(\gamma_T):=\Gg^T_\gamma(\Bb^+)$.
By~\cite[(31)]{weber:2014a}
the point $z$ is the initial value
of a heat flow trajectory $\xi^T$ ending at time $T$ on
the fiber $\Dd_\gamma$.
Assume the image $\xi^T([0,T])=\phi_{[0,T]} z$ is
contained in $F:=\im\Gg\cup\im\Gg^\infty$. Pick
$\sigma\in[0,T-T_0]$. This implies that
$z_+:=\pi_+\phi_\sigma z\in\Bb^+$. The
flow line $\phi_{[0,T-\sigma]}\phi_\sigma z$ runs
from $\phi_\sigma z$ to $\phi_T z\in\Dd_\gamma$.
Hence this flow line coincides with the
fixed point $\xi_{\gamma,z_+}^{T-\sigma}$ of the strict
contraction $\Psi_{\gamma,z_+}^{T-\sigma}$. But
$\phi_\sigma z=\xi_{\gamma,z_+}^{T-\sigma}(0)$ is equal to
$\Gg^{T-\sigma}_\gamma(z_+)$ again
by~\cite[(31)]{weber:2014a} and
$\Gg^{T-\sigma}_\gamma(\Bb^+)=:F(\gamma_{T-\sigma})
=F(\phi_\sigma\gamma_T)$ by definition of $F$ and
$\gamma_{T-\sigma}$.
\end{proof}

\subsubsection*{Ascending disks}
Since nondegeneracy of $x$ is equivalent
to a strictly positive spectral gap $d$,
the following two results are based
on the Palais-Morse Lemma~\cite{palais:1969a}
and the Local Stable Manifold
Theorem~\cite[Thm.~3]{weber:2014a}
whose neighborhood assertion uses
the non-trivial fact that convergence implies
exponential convergence.

\begin{figure}
  \centering
  \includegraphics{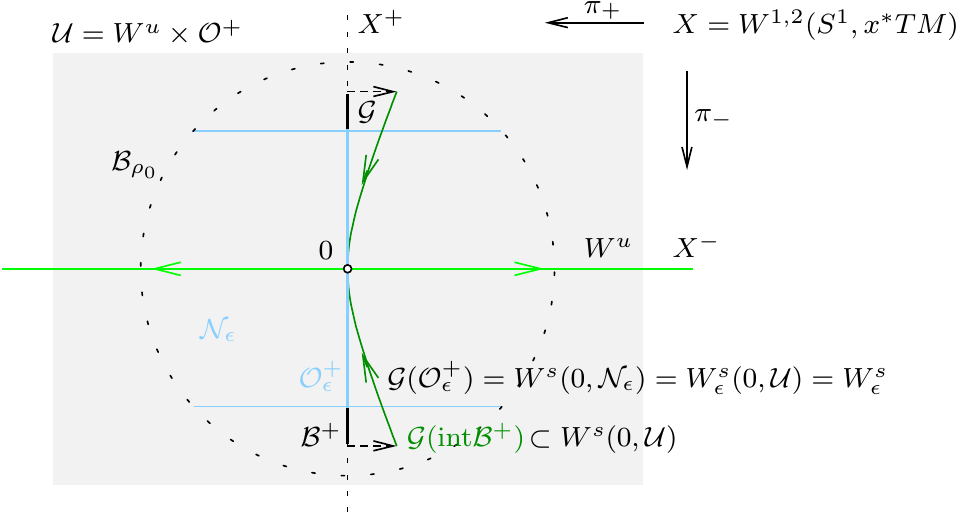}
  \caption{The local ascending disk
           $W^s_\eps(0,\Uu)$ is a graph
            and equal to $W^s_\eps$}
  \label{fig:fig-loc-stab-mf-thm}
\end{figure}

\begin{lemma}[Ascending disks]\label{le:asc-disk}
Assume~(H1) and~(H2) of
Hypothesis~\ref{hyp:local-setup-N}.
The Local Stable Manifold
Theorem~\cite[Thm.~3]{weber:2014a}
provides the closed ball $\Bb^+$ about
$0\in X^+$ of radius $r>0$.
Then there is a constant $\eps_0=\eps_0(r)>0$
such that the following is true whenever
$\eps\in(0,\eps_0]$.
\begin{enumerate}
\item[\rm (i)]
  The {\bf local ascending disk} defined by
  $$
     W^s_\eps(0,\Uu)
     :=W^s(0,\Uu)\cap\{\Ss<\Ss(0)+\eps\}
  $$
  is, firstly, a graph $\Gg^\infty(\Oo^+_\eps)$
  over the subset
  $\Oo^+_\eps:=\pi_+W^s_\eps(0,\Uu)\subset\Bb^+$
  which, secondly, is diffeomorphic to an open disk
  in $X^+$.
  Thirdly, that graph also
  coincides with the {\bf local stable manifold}
  $$
     W^s(0,\Nn_\eps)
     :=\left\{ z\in\Nn_\eps\mid\text{$\phi(s,z)\in\Nn_\eps$
     $\forall s>0$ and $\lim_{s\to\infty}\phi(s,z)=0$}\right\}
  $$
  of the set $\Nn_\eps
  :=\INT\,\Bb_{\rho_0} \cap{\pi_+}^{-1}\Oo^+_\eps
  \subset\Uu$ illustrated
  in Figure~\ref{fig:fig-loc-stab-mf-thm}.

\item[\rm (ii)]
  Any neighborhood $\Ww$ of $0$ in
  $W^s(0,\Uu)$ contains a local ascending disk.

\item[\rm (iii)]
  The local coordinate representative
  $W^s_\eps:=\Phi^{-1} W^s_\eps(x)$
  of the ascending disk $W^s_\eps(x)$
  defined by~(\ref{eq:asc-disk}) coincides
  with the local ascending disk $W^s_\eps(0,\Uu)$.
\end{enumerate}
\end{lemma}

\begin{corollary}\label{cor:loc-stab-mf}
In the notation of Lemma~\ref{le:asc-disk}
assume that $\Nn\subset\Uu$ is an open subset
which contains the hyperbolic fixed point $0$.
Then the local stable manifold $W^s(0,\Nn)$
is an open neighborhood of $0$
in $W^s(0,\Uu)$.
\end{corollary}

\begin{proof}[Proof of Lemma~\ref{le:asc-disk}]
(Ascending disks).
By the Local Stable Manifold
Theorem~\cite[Thm.~3]{weber:2014a}
a neighborhood of $0$ in $W^s(0,\Uu)$,
say $\Ww\subset\range\Gg$,
is embedded in $\Lambda M$ and its tangent space
at $0$ is $X^+=\pi_+(X)$.
Observe that the restriction $f:=\Ss|$ of
the action to $\Ww$ is a Morse function.
Apply the Palais-Morse
Lemma~\cite{palais:1969a} to
obtain a coordinate system on
$\Ww$ (choose $\Ww$ smaller if necessary)
modelled on $T_0\Ww=X^+$ and
such that
$$
     f(y)=\sum_{j=1}^\infty\lambda_{k+j} y_j^2
$$
for every $y\in\Ww$. Here
$y=\sum_{j=1}^\infty y_j\xi_{k+j}$
and $0<\lambda_{k+1}<\lambda_{k+2}<\ldots$ are the
positive eigenvalues of the Jacobi operator $A$
associated to the critical point $0$ of $\Ss$
with corresponding normalized eigenvectors $\xi_{k+j}$;
see e.g.~\cite[(2)]{weber:2014a}.

In these coordinates the local ascending disk
$W^s_\eps(0,\Uu)$ takes the form of an open ellipse
in $X^+$ which is given by
\begin{equation*}
\begin{split}
     \Ee_\eps&:=\Ee\left(a_1,a_2,\ldots\right)
     =\biggl\{ y\in X^+\colon
     \sum_{j=1}^\infty \lambda_{k+j} y_j^2
     <\eps\biggr\}
     \subset\Oo^+_R
   \\
     a_j&:=\sqrt{\frac{\eps}{\lambda_{k+j}}}
\end{split}
\end{equation*}
and contained in the {\bf open} ball
$\dot\Bb^+_{R}\subset X^+$
of radius $R=a_1(\eps)$.
Since any neighborhood of $0$
contains a ball of sufficiently
small radius this proves part~(ii).

To prove~(i)
fix the radius $\eps_0>0$ sufficiently small such
that the open ball $\dot\Bb^+_{\eps_0}$ is contained,
firstly, in the domain
of our Palais-Morse parametrization, secondly,
in the Palais-Morse representative of $\Ww$ 
and, thirdly, in the Palais-Morse
representative of the ball $\Bb^+\subset X^+$
of radius $r>0$.
The second assertion in part~(i) follows
since $\dot\Bb^+_{\eps_0}$ represents the manifold
$W^s_{\eps_0}(0,\Uu)$ which is diffeomorphic
under $\pi_+$ to
$$
     \Oo^+_{\eps_0}
     :=\pi_+ W^s_{\eps_0}(0,\Uu)
     \subset\Bb^+.
$$
Here the diffeomorphism property follows
from the fact that $W^s_{\eps_0}(0,\Uu)$
is tangent to $X^+$ at $0$ and by choosing
$\eps_0>0$ smaller, if necessary.
The tangency argument also justifies the assumption
that $W^s_{\eps_0}(0,\Uu)\subset\INT\,\Bb_{\rho_0}$,
otherwise choose $\eps_0>0$ smaller.
The same arguments work for each $\eps\in(0,\eps_0]$
and $\Gg(\Oo^+_\eps)$ is well defined.

To prove the remaining assertions one and three in~(i)
we show that
\begin{equation}\label{eq:inclusions-Gg}
     \Gg(\Oo^+_\eps)
     \subset W^s(0,\Nn_\eps)
     =W^s_\eps(0,\Uu)
     \subset \Gg(\Oo^+_\eps)
     ,\quad
     \Nn_\eps
     :=\INT\,\Bb_{\rho_0} \cap{\pi_+}^{-1}\Oo^+_\eps,
\end{equation}
whenever $\eps\in(0,\eps_0]$.
To understand the middle identity observe that
the inclusion '$\subset$' is obvious since
$\Nn_\eps\subset\Bb_{\rho_0}\subset\Uu$.
To see the reverse '$\supset$'
note that
$$
     W^s_\eps(0,\Uu)
     \subset\left(\INT\,\Bb_{\rho_0}\cap
    {\pi_+}^{-1}\pi_+ W^s_\eps(0,\Uu)\right)
     =:\Nn_\eps.
$$
By semi-flow invariance of local ascending disks
the elements of $W^s_\eps(0,\Uu)$
converge to $0$ without leaving $W^s_\eps(0,\Uu)$,
hence without leaving $\Nn_\eps$.
But this means that $W^s_\eps(0,\Uu)\subset
W^s_\eps(0,\Nn_\eps)$.
To prove the second inclusion in~(\ref{eq:inclusions-Gg})
observe that $\Nn:=\Gg(\Oo^+_{\eps_0})$ is a neighborhood
of $0$ in $W^s(0,\Uu)$. Apply part~(ii)
proved above and readjust $\eps_0$, if necessary.
This proves that $W^s_\eps(0,\Uu)\subset\Gg(\Oo^+_\eps)$.
To prove the first inclusion in~(\ref{eq:inclusions-Gg})
pick $z\in\Gg(\Oo^+_\eps)$,
that is 
$$
     z=(Gz_+,z_+)=\Gg(z_+)\in \Gg(\Oo^+_\eps)
$$
for some $z_+\in \Oo^+_\eps$.
To see that $z\in W^s_\eps(0,\Uu)$ consider the
(unique) element $z_*$ of $W^s_\eps(0,\Uu)$
which projects under the diffeomorphism
$\pi_+:W^s_\eps(0,\Uu)\to \Oo^+_\eps$ to $z_+$.
Since we already know that 
$W^s_\eps(0,\Uu)\subset\Gg(\Oo^+_\eps)$
the point $z_*\in W^s_\eps(0,\Uu)$
is of the form $z_*=\Gg(z_+)$. But $\Gg(z_+)=z$.

The key information to prove part~(iii)
is the fact shown above using
the Palais-Morse lemma, namely
that the local ascending disk $W^s_\eps(0,\Uu)$
is contained in the interior of the ball $\Bb_{\rho_0}$
which itself is contained in the domain $\Uu$
of the parametrization $\Phi$. But $\Phi$ intertwines
the local semi-flows $\phi_s$ on $\Uu$ and $\varphi_s$
on $\Phi(\Uu)$ by its very definition; cf.~\cite[(5)]{weber:2014a}.
\end{proof}

\begin{proof}[Proof of Corollary~\ref{cor:loc-stab-mf}]
Obviously
$0\in W^s(0,\Nn)\subset W^s(0,\Uu)$.
It remains to show that the subset
$W^s(0,\Nn)$ of $W^s(0,\Uu)$ is open.
Fix $z\in W^s(0,\Nn)\subset\Nn$.
It suffices to prove existence of an
open ball $\Oo(z)\subset\Uu$ about $z$
such that the (open) subset $\Oo(z)\cap W^s(0,\Uu)$
of $W^s(0,\Uu)$ is contained in $W^s(0,\Nn)$.
Assume by contradiction that no such ball exists.
In this case there is a sequence $(z_i)$ contained
in $W^s(0,\Uu)$ and in $\Nn$,\footnote{
  We may assume that $z_i\in\Nn$ since $z$ lies
  in the open subset $\Nn$ of $\Uu$.
  }
but disjoint to $W^s(0,\Nn)$, and which
converges to $z$ in the $W^{1,2}$ topology.
Consequently for each $z_i$ there is a time $s_i>0$
such that $\phi_{s_i} z_i\notin\Nn$. Taking subsequences,
if necessary, we  distinguish two cases:
\\
In {\bf case one} the sequence $(s_i)$ is contained in some
bounded interval $[0,T]$. Now $\phi$ restricted
to a sublevel set is uniformly Lipschitz on a fixed
interval $[0,T]$ by a slightly improved version
of~\cite[Thm.~9.15]{weber:2010a-link};
see~\cite{weber:heat-book}.
Thus the sequence of continuous maps
$[0,T]\to\Uu: s\mapsto w_{z_i}(s):=\phi_s z_i$
converges uniformly to the map
$w_z:[0,T]\to\Nn\subset\Uu$.
But this implies that the image of $w_{z_i}$
is also contained in $\Nn$ for all sufficiently large $i$
which contradicts the fact that $\phi_{s_i} z_i\notin\Nn$.
\\
In {\bf case two} $s_i\to\infty$, as $i\to\infty$.
By openness of $\Nn$ there is a sufficiently small
open ball $\Oo_\rho$ of radius $\rho$ about
$0\in\Uu$ which is contained in $\Nn$.
By Lemma~\ref{le:asc-disk}~(ii) there is a local
ascending disk $W^s_\eps(0,\Uu)$
contained in the open neighborhood
$\Ww:=W^s(0,\Uu)\cap\Oo_\rho$
of $0$ in $W^s(0,\Uu)$.
Fix $\tau>0$ large such that
$\phi(\tau,z)\in W^s_{\eps/2}(0,\Uu)$.
Then the following is true for every
sufficiently large~$i$: The point
$\phi(\tau,z_i)$ lies in $W^s_\eps(0,\Uu)$
by continuity of $\phi$.
But $W^s_\eps(0,\Uu)$ is semi-flow invariant
and contained in $\Oo_\rho\subset\Nn$.
So $\phi(s,z_i)\in\Nn$
for $s\in[\tau,\infty)$
which contradicts $s_i\to\infty$.
\end{proof}

\subsubsection*{Proof of Theorem~\ref{thm:inv-fol} -- intrinsic foliation}
Assume
Hypothesis~\ref{hyp:local-setup-N}~(H1--H4).
In particular, by definition of $\mu$
in~(H3) both the descending disk $W^u_{2\mu}(x)$
and the ascending disk $W^s_{2\mu}(x)$
are manifolds and lie in the
coordinate patch $\Phi(\Bb_{\rho_0})$ about
the nondegenerate critical point $x$ of Morse index $k$.
The Local Stable Manifold
Theorem~{\cite[Thm.~3]{weber:2014a}}
provides the graph map
$\Gg^\infty:\Bb^+\to X$ defined on the
closed ball $\Bb^+=\Bb^+_r$ about $0\in X^+$
whose radius $r$ we write in the form
\begin{equation}\label{eq:def-R}
      r=:2R.
\end{equation}
Again by~{\cite[Thm.~3]{weber:2014a}}
the set $\Nn:=\Gg^\infty(\dot\Bb_{R}^+)$ is an
open neighborhood of $0$ in the local stable
manifold $W^s(0,\Uu)$. Thus $\Nn$ contains an
ascending disk by the ascending disk
Lemma~\ref{le:asc-disk}~(ii).
Choosing $\mu>0$ smaller, if
necessary, we assume without loss of generality that
there is the inclusion of the ascending disk
coordinate representative
\begin{equation}\label{eq:mu-N}
      W^s_{\mu}\subset\Nn:=\Gg^\infty(\dot\Bb_{R}^+).
\end{equation}

The coordinate representative $\Dd$ of the
tubular neighborhood $\Dd(x)$ intersects the unstable
manifold transversally in $S^u_\eps$.
Use the implicit function theorem, if
necessary, to modify the coordinate system
locally near $\Dd$ to make sure that $\Dd$
is an open neighborhood
of $S^u_\eps$ in $S^u_\eps\times X^+$.
Pick a radius $\varkappa\in(0,\rho_0)$ sufficiently
small such that $S^u_\eps\times\Bb_\varkappa ^+$
is contained in $\Dd$ and in
$\Bb_{\rho_0}$. Next diminish $\Dd$ setting
\begin{equation}\label{eq:D-def}
     \Dd:=S^u_\eps\times\Bb_\varkappa^+
     ,\qquad
     \Dd\cap\Crit=\emptyset,
\end{equation}
where the latter observation holds by~(H2).
Since $\Dd$ is contained in an action level set
and $\phi$ is a gradient semi-flow, the pair
$(\Dd,\phi)$ has the no return property.
Consider the constant
$T_0=T_0(x,\lambda,\eps,\varkappa)>0$
and the graph maps $\Gg^T_\gamma$ provided
by the Backward
$\lambda$-Lemma~{\cite[Thm.~1]{weber:2014a}}
for all $T\ge T_0$ and elements $\gamma$ of the
descending $(k-1)$-disk $S^u_\eps$;
see Figure~\ref{fig:fig-foliation-2R}.

\begin{figure}
  \centering
  \includegraphics{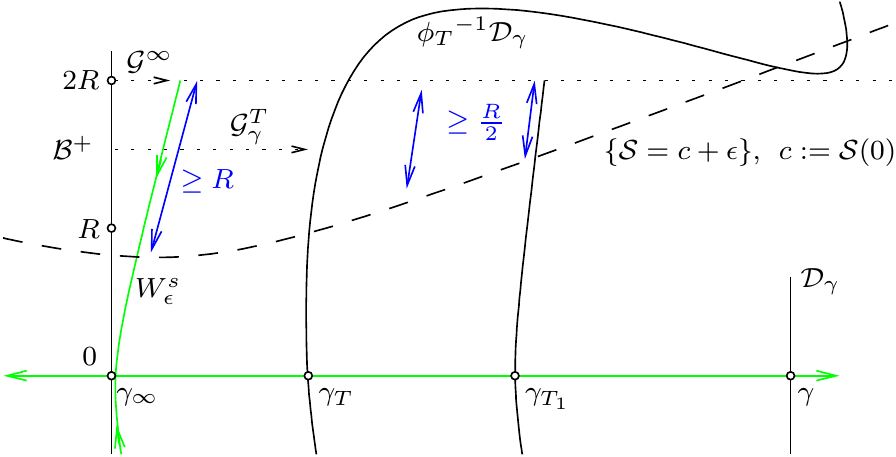}
  \caption{The disk
           $\Gg^T_\gamma(\Bb^+)\cap\{\Ss<c+\eps\}
           =\left({\phi_T}^{-1}\Dd_\gamma\cap
           \{\Ss<c+\eps\}\right)_{\gamma_T}$
           }
  \label{fig:fig-foliation-2R}
\end{figure}

\vspace{.1cm}
\noindent
{\sc Step 1. (Graphs)}
{\it
There is a constant $T_1\ge T_0$ such that the
     following is true. Assume
     $T\in[T_1,\infty]$ and $\gamma\in S^u_\eps$. Then
     the set $\Gg^T_\gamma(\Bb^+)\cap\{\Ss<c+\eps\}$
     is diffeomorphic to the open unit disk
     in $X^+$.
}

\begin{proof}
Case~1. ($T=\infty$)
The graph $\Gg^\infty(\Bb^+)$ -- which is a
neighborhood of $0$ in the local stable manifold
$W^s(0,\Uu)$ by the Local Stable Manifold
Theorem~\cite[Thm.~3]{weber:2014a}
-- intersects
the sublevel set $\{\Ss<c+\eps\}$ transversally
in the ascending disk $W^s_\eps$.
But $W^s_\eps$ is diffeomorphic to the open $\eps$-disk
in $X^+$ by the Palais-Morse lemma using the fact that
the positive part of the spectrum of the Jacobi
operator $A_x$ is bounded away from zero
(by its smallest positive eigenvector $\lambda_{k+1}$).
For the above assertions see  Lemma~\ref{le:asc-disk}.

Case~2. ($T<\infty$) By the Backward
$\lambda$-Lemma~{\cite[Thm.~1]{weber:2014a}}
the family
of disks $T\mapsto \Gg^T_\gamma(\Bb^+)$ is uniformly
$C^1$ close to the disk $\Gg^\infty(\Bb^+)$.
Transversality of the intersection with 
$\{\Ss<c+\eps\}$ is automatic
since the sublevel set is an open subset of
the loop space.
However, since the graphs $\Gg^T_\gamma(\Bb^+)$
are manifolds with boundaries
we need to make sure that these boundaries
stay away from $\{\Ss<c+\eps\}$
in order to conclude that any intersection
$\Gg^T_\gamma(\Bb^+)\cap \{\Ss<c+\eps\}$
is diffeomorphic to the intersection
$\Gg^\infty(\Bb^+)\cap \{\Ss<c+\eps\}=W^s_\eps$.
But the latter is diffeomorphic to the open unit disk
in $X^+$ by Case~1.
\\
Concerning boundaries recall that
$\pi_+\Gg^\infty(\Bb^+)=
\pi_+\Gg^T_\gamma(\Bb^+)=\Bb^+=\Bb^+_{2R}$.
Here the second identity holds by step~5 in the proof
of~{\cite[Thm.~1]{weber:2014a}}.
On the other hand, the topological boundary
of $W^s_\eps$ projects into $\Bb^+_{R}$
by the choice of $\mu$ in~(\ref{eq:mu-N}); see
Figure~\ref{fig:fig-foliation-2R}.
Thus the distance between the boundary of
$\Gg^\infty(\Bb^+)$ and the intersection
$\Gg^\infty(\Bb^+)\cap \{\Ss<c+\eps\}=W^s_\eps$
is at least $R$.
Since $\Gg^T_\gamma\to\Gg^\infty$, as $T\to\infty$,
uniformly on $\Bb^+$ and uniformly in
$\gamma\in S^u_\eps$,
there is a time $T_1>0$ such that
the distance between the boundary of
$\Gg^T_\gamma(\Bb^+)$ and the intersection
$\Gg^T_\gamma(\Bb^+)\cap \{\Ss<c+\eps\}$
is at least $R/2$ for all $\gamma$ and
$T\ge T_1$.
\end{proof}

\noindent
{\sc Step 2. (Pre-Images)}
{\it
      For all $T\ge T_1$ and $\gamma\in S^u_\eps$
     the following is true.
     \begin{enumerate}
     \item[a)]
       The disk
       $\Gg^T_\gamma(\Bb^+)\cap\{\Ss<c+\eps\}=:D$
       is a neighborhood of $\gamma_T$ in the pathwise
       connected component $P_{\gamma_T}$ of the set
       $P:={\phi_T}^{-1}\Dd_\gamma\cap\{\Ss<c+\eps\}$.
     \item[b)]
     The disk $\Gg^T_\gamma(\Bb^+)\cap\{\Ss<c+\eps\}$
     equals $P_{\gamma_T}:=\left({\phi_T}^{-1}\Dd_\gamma\cap
     \{\Ss<c+\eps\}\right)_{\gamma_T}$.
     \end{enumerate}
}

\begin{proof}
a) That $\gamma_T$ is contained in $P$ is obvious
and that it is contained in $D$ is asserted
by the Backward $\lambda$-Lemma~{\cite[Thm.~1]{weber:2014a}}.
To see that $D\subset P_{\gamma_T}$ pick $z\in D$.
Then the heat flow takes $z$ in time $T$ into
$\Dd_\gamma$ by definition of $\Gg^T_\gamma$ and
the identity~\cite[(31)]{weber:2014a}. Hence $z\in P$
and therefore $D\subset P$.
Thus to prove that $D\subset P_{\gamma_T}$ it
suffices to show that $z$ path connects
to $\gamma_T$ inside $D$. But this is trivial,
because $D$ is diffeomorphic to a disk by Step~1.
To see the neighborhood property of $D$ pick
$z\in P_{\gamma_T}$ and connect
$z$ to $\gamma_T$ inside $P$ through
a continuous path.
Of course, since $\pi_+\gamma_T=0$
the elements of the path
near $\gamma_T$ project under $\pi_+$ into
$\Bb^+$ and are therefore in the image of the
map $\Gg^T_\gamma$ defined
by~\cite[(25)]{weber:2014a}.

b) By part~a) it remains to prove the inclusion '$\supset$'.
Pick $z\in P_{\gamma_T}$ and connect $z$ to $\gamma_T$
inside $P$ through a continuous path. Note that all
points on this path have action strictly less than
$c+\eps$. Now if z was not in the disk $D$, this path
would have to cross the
topological boundary of $D$ by the
neighborhood property in a). But $\p D$ is
contained in the level set $\{\Ss=c+\eps\}$.
Contradiction.
\end{proof}

\vspace{.1cm}
\noindent
{\sc Step 3.}
{\it Set $\tau_0:=2T_1$. Assume from now on that
     $\tau>\tau_0$. Recall that
     Corollary~\ref{cor:inv-fol} provides
     the codimension $k$ foliation
     $F=F^{\eps,\tau}:=\im \Gg^{(\tau,\infty]}$. Then
     $$
       A:=F^{\eps,\tau}\cap\{\Ss<c+\eps\}
       =N^{\eps,\tau}=:N,
     $$
     that is the part $A$ below level $c+\eps$
     of the foliation $F^{\eps,\tau}$
     is equal to the coordinate representative 
     of the set $N_x^{\eps,\tau}$
     defined by~(\ref{eq:Conley-set-NEW});
     see Figure~\ref{fig:fig-ABC}.
     The point is that $A$ is essentially the image
     of a family of maps, but the definition of $N$
     requires each point being path connectable to $0$.
}

\begin{figure}
  \centering
  \includegraphics{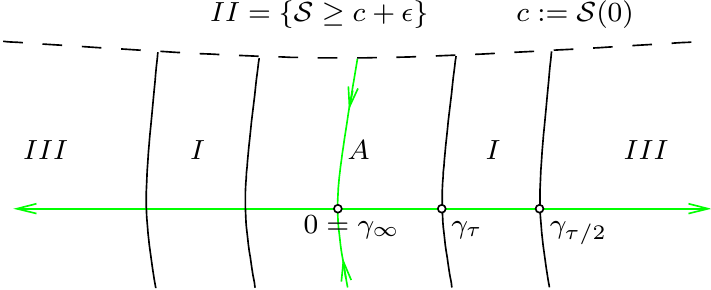}
  \caption{The set $A$ in step~3
           with neighborhood $A\cup I\cup II$
           }
  \label{fig:fig-ABC}
\end{figure}

\begin{proof}
$A\subset N$: Pick $z\in A$. Then $\Ss(z)<c+\eps$ and
$z$ is of the form $\Gg^T_\gamma(z_+)$ for some time
$T\in(\tau,\infty]$ and elements $\gamma\in S^u_\eps$
and $z_+\in\Bb^+$.
But $\Gg^T_\gamma(z_+)=\xi^T_{\gamma,z_+}(0)$
by~\cite[(31)]{weber:2014a}
and therefore $z$ runs under the
heat flow in time $T>\tau$ into the subset $\Dd$ of the
level set $\{\Ss=c-\eps\}$. Thus
$\Ss(\phi_\tau z)>c-\eps$ by the downward gradient
flow property and the fact that by~(\ref{eq:D-def})
there is no critical point of $\Ss$ on $\Dd$.
To conclude the proof
that $z\in N$ it remains to show that there is
a continuous path in $N$ between $z$ and $0$.
By Step~1 the set $\Gg^T_\gamma(\Bb^+)$ is a disk and
therefore path connected. Connect $z$ and $\gamma_T$
by a continuous path in this disk. Any point on this
path lies in
$\{\Ss<c+\eps\}\cap\{\Ss(\phi_\tau\cdot)>c-\eps\}$
by the argument just given for $z$. Connect
$\gamma_T$ and $\gamma_\infty=0$ by the obvious
backward flow line. Repeat the argument for the
points on this second path. Hence we have connected
$z$ and $0$ by a continuous path in $N$.

$A\supset N$: Assuming $z\notin A$ we prove that
$z\notin N$. To be not in $A$ we distinguish three cases;
see Figure~\ref{fig:fig-ABC}.
In case one $z$ lies in the set
$I:=\im \Gg^{(\tau/2,\tau]}\cap\{\Ss<c+\eps\}$.
But this means that $z$ reaches level $c-\eps$ in some
time $T\le\tau$. Hence $\Ss(\phi_\tau)\le c-\eps$ and
therefore $z\notin N$.
In case two $z$ lies in the set $II:=\{\Ss\ge c+\eps\}$
which is obviously disjoint to $N$.
In case three $z$ lies in the set
$III:=\{\Ss<c+\eps\}\cap
\{\Ss(\phi_{\tau/2}\cdot)\le c-\eps\}$
shown in Figure~\ref{fig:fig-ABC}.
Assume by contradiction $z\in N$.
Then $z$ and $0$ connect through a
continuous path in $N$. Note that
$0\in A$ since $\Gg^\infty(0)=0$.
Since $A\cup I\cup II$ is a
neighborhood of $A$, the path must run through
$I\cup II$ which is impossible by cases one and two.
\end{proof}

\begin{proof}[Proof of~a). (Foliation)]
By Step~3 and Corollary~\ref{cor:inv-fol}
there are the inclusions
$N^{\eps,\tau}\subset F^{\eps,\tau}\subset \Bb_{\rho_0}$.
But by~(H2) the ball $\Bb_{\rho_0}$
contains no critical point except the origin.
Thus $N_x$ is an isolating block for $x$;
this also follows from part~d).

By Corollary~\ref{cor:inv-fol} the set $F=F^{\eps,\tau}$
carries the structure of a codimension $k$ foliation.
By Step 3 the set $N=N^{\eps,\tau}$ is an open subset
of $F$ and therefore inherits the foliation structure
of $F$. We define the leaves of $N$ by
$N(0):=F(0)\cap\{\Ss<c+\eps\}
=\Gg^\infty(\Bb^+)\cap\{\Ss<c+\eps\}$ and by
$N(\gamma_T):=F(\gamma_T)\cap\{\Ss<c+\eps\}
=\Gg^T_\gamma(\Bb^+)\cap\{\Ss<c+\eps\}$
where $T\in(\tau,\infty)$ and $\gamma\in S^u_\eps$.
The second identities are just by definition of $F(0)$
and $F(\gamma_T)$ in Corollary~\ref{cor:inv-fol}.
Since the right hand sides are disks by Step~1
the leaves of $N$ are indeed parametrized
by the disjoint union of $\{0\}$ and
$(\tau_0,\infty)\times S^u_\eps$.
Hence the leaves of $N$ and $F$ are in 1-1
correspondence. They are of the asserted form by
Step 2 b).
\end{proof}

\begin{proof}[Proof of~b). (Compatibility of leaves
and semi-flow)]
That leaves and semi-flow are compatible follows from
Corollary~\ref{cor:inv-fol} as soon as we prove
that semi-flow trajectories starting and ending in
$N=N^{\eps,\tau}$ cannot leave $N$ (hence not $F$) at
any time in between. To see this decompose the
(topological) boundary of the set
$N=F\cap\{\Ss<c+\eps\}$ into
the top part $\p^+ N$ which lies in the level set
$\{\Ss=c+\eps\}$ and its complement the side part
$
     \p^- N
     =\bigcup_{\gamma\in S^u_\eps}
     \Gg^\tau_\gamma(\Bb^+)\cap\{\Ss<c+\eps\}
$
as illustrated by Figure~\ref{fig:fig-induced-flow}
below.
The downward gradient property implies, firstly, that
$\p^+ N$ cannot be reached from lower action levels
(thus not from $N$) and, secondly, that $\p^- N$
cannot be crossed twice. To prove the latter
assume by contradiction
that there are two elements $z_1\not= z_2$ of 
$$
     \p^- N
     =\left({\phi_\tau}^{-1}\Dd
     \cap\{\Ss<c+\eps\}\right)_{\phi_{-\tau} S^u_\eps}
$$
that lie on the same semi-flow trajectory starting
at, say $z_1$. Now on one hand,
the time needed from either one element to $\Dd$ is
$\tau$. On the other hand, getting from $z_1$ to $z_2$
requires the extra time $T>0$. By uniqueness of the
solution to the Cauchy problem it follows that
$\tau+T=\tau$ which contradicts $T>0$.
\end{proof}

\begin{proof}[Proof of~c). (Uniform convergence of leaves)]
Uniform and exponential convergence of leaves
follows from the exponential estimate
in~\cite[Thm.~1]{weber:2014a},
in which we can actually eliminate the constant
$\rho_0$ by choosing $T_0$ larger,
together with the inclusion
$
     N(\gamma_T)
     =\Gg^T_\gamma(\Bb^+)\cap\{\Ss_\Vv<c+\eps\}
     \subset\Gg^T_\gamma(\Bb^+)
$
and the corresponding one for $T=\infty$; for the
identity see proof of~a). This
proves~(\ref{eq:unif-exp-dist}).
Given $U$ as in the second assertion,
pick a $\delta$-neighborhood
$U_\delta\subset\Phi^{-1}(U)$
of $W^s_\eps$ in $\Bb_{\rho_0}$
for some $\delta\in(0,1)$.
Estimate~(\ref{eq:unif-exp-dist})
shows that $N^{\eps,\tau_*}\subset U_\delta$
whenever $\tau_*>-\frac{16}{\lambda}\ln \delta$.
\end{proof}

\begin{proof}[Proof of~d). (Localization of $N_x$)]
The two key ingredients are that the ascending
disk $W^s_\eps(x)$ localizes
near $x$ for small $\eps$ by the Palais-Morse Lemma
and that the isolating block $N_x^{\eps,\tau}$
contracts onto $W^s_\eps(x)$ by
estimate~(\ref{eq:unif-exp-dist}) in part~c).

Replacing the neighborhood $U$ of $x$ in
$\Lambda M$ by a smaller neighborhood,
if necessary, we solve the problem in the local
coordinate patch $\Phi(\Bb_{\rho_0})$ about $x$.
Thus we assume that $U$ is a neighborhood of
$0$ in $\Bb_{\rho_0}\subset X$.
By~(\ref{eq:def-R}) the radius of the ball $\Bb^+$
on which the stable manifold graph map
$\Gg^\infty$ is defined is $2R>0$; see
Figure~\ref{fig:fig-foliation-2R}.
Pick $\rho\in(0,R]$ sufficiently small such
that the ball $B_{2\rho}(0)$ is contained in $U$. By
the ascending disk Lemma~\ref{le:asc-disk}~(ii) the open
neighborhood $\Nn:=W^s_\eps\cap\INT B_\rho(0)$
of $0$ in the ascending disk $W^s_\eps$ contains
an ascending disk $W^s_{\eps_*}$ for some
$\eps_*\in(0,\eps)$.
Note that $W^s_{\eps_*}\subset\Nn\subset B_\rho(0)$.
Pick $\delta\in(0,\rho)$ and apply part~c)
for $W^s_{\eps_*}$ and its $\delta$-neighborhood
$U_\delta$ to obtain a constant
$\tau_*$ and the first of the inclusions
$
     N^{\eps_*,\tau_*} 
     \subset U_\delta(W^s_{\eps_*})
     \subset U_\delta(B_{\rho}(0))
     \subset B_{2\rho}(0)
     \subset U
$.
\end{proof}

\noindent
This completes the proof of
Theorem~\ref{thm:inv-fol}.

\subsection{Strong deformation retract}
\label{sec:deformation-Conley-pair}

\begin{proof}[Proof of Theorem~\ref{thm:deformation-retract}]
Assume Hypothesis~\ref{hyp:local-setup-N}.
Our construction of a strong deformation retraction
$\theta$ of $N$ onto its part $A$
in the unstable manifold
is \emph{motivated by the following observation}:
On the stable manifold the semi-flow
$\{\phi_s\}_{s\in[0,\infty]}$ itself does the job.
Indeed $\phi_\infty$ pushes
the whole leaf $N(0)$, that is the ascending disk
$W^s_\eps$ by Theorem~\ref{thm:inv-fol},
into the origin -- which
lies in the unstable manifold. Since $\phi_s$
restricted to the origin is the identity, the origin
is a strong deformation retract of $N(0)$.
If the Morse index $k$ is zero, then
$N=N(0)$ and we are done.

Assume from now on that $k>0$. In this case
the Backward $\lambda$-Lemma comes in.
It implies that $N$ is a foliation whose leaves are
$C^1$ modelled on the ascending disk $W^s_\eps$;
see Theorem~\ref{thm:inv-fol}.
The main and by now obvious idea is to use the graph
maps $G^T_\gamma$ and $\Gg^\infty$ of
Theorems~1 and~3 in~{\cite{weber:2014a}},
respectively,
and their left inverse $\pi_+$ to extend the good
retraction properties of  $\phi_s$ on the
ascending disk $N(0)$ to all the other leaves
$N(\gamma_T)$ where
$\gamma_T:=\phi_{-T}\gamma$.

\begin{figure}
  \centering
  \includegraphics{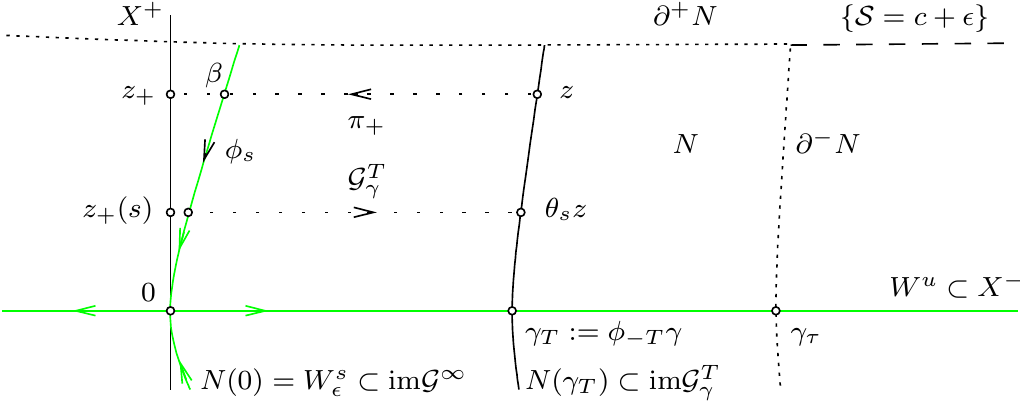}
  \caption{Leaf preserving semi-flow
           $\theta_sz
           :=\Gg^T_\gamma\pi_+\phi_s\Gg^\infty\pi_+ z$
           on foliation $N$
          }
  \label{fig:fig-induced-flow}
\end{figure}

\begin{definition}[Induced semi-flow]
\label{def:induced-semi-flow}
By Theorem~\ref{thm:inv-fol}
each $z\in N=N^{\eps,\tau}$ lies on a leaf
$N(\gamma_T)$ for some $T>\tau$ and some
$\gamma$ in the ascending disk $S^u_\eps$.
Set
$$
     z_+:=\pi_+ z,\qquad
     \beta:=\Gg^\infty(z_+),\qquad
     z_+(s):=\pi_+ \phi_s\beta,
$$
for $s\ge 0$.
Then the continuous map
$\theta:[0,\infty]\times N\to X$ given by
\begin{equation}\label{eq:str-def-retract}
     \theta_s z
     :=\Gg_\gamma^T
     \pi_+ \phi_s
     \Gg^\infty \pi_+ z
\end{equation}
is called the
{\bf induced semi-flow on \boldmath $N$};
see Figure~\ref{fig:fig-induced-flow}.
It is of class $C^1$ on $(0,\infty)\times N$ and
juxtaposition of maps means composition.
\end{definition}

Observe that $\theta$ takes values
in the image $F\supset N$ of the graph maps
and that it preserves the leaves of $F$;
see Corollary~\ref{cor:inv-fol}.
Continuity on $[0,\infty)\times N$
follows from continuity of the maps involved.
Existence of the asymptotic limit
$\phi_s\beta\to 0$, as $s\to\infty$,
for any $\beta\in W^s_\eps=N(0)$
has the following two consequences.
Assume $z\in N(\gamma_T)$. Then,
firstly, the limit
\begin{equation*}
     \theta_\infty z
     :=\lim_{s\to\infty} \theta_s z
     =\Gg^T_\gamma\pi_+
     \lim_{s\to\infty}\phi_s\beta
     =\Gg^T_\gamma(0)
     =\gamma_T
\end{equation*}
exists and lies in the unstable manifold indeed.
Here we used continuity of $\Gg^T_\gamma$
and $\pi_+$ and the fact that
$\beta=\Gg^\infty(z_+)$ lies in the stable
manifold of the origin.
The final identity holds by~\cite[Thm.~1]{weber:2014a}.
Secondly, $\theta_s z\to\theta_\infty z$,
as $s\to\infty$. The first consequence shows that
\begin{equation}\label{eq:strong-def-retract}
     \theta_\infty:N\to A
     ,\qquad
     A:=\phi_{-\tau} W^u_\eps
     \cong\{0\}\cup
     \left((\tau,\infty)\times S^u_\eps\right),
\end{equation}
is a retraction and the second one extends continuity
to $[0,\infty]\times N$.
The fact that the origin is a fixed point of $\phi_s$
implies that
\begin{equation*}
     \theta_s\gamma_T=\Gg^T_\gamma\pi_+\phi_s 0
     =\Gg^T_\gamma(0)=\gamma_T,
\end{equation*}
hence $\theta_s|_A=id_A$, for every $s\in[0,\infty]$.

To conclude the proof it remains to show
that $\theta_s$ preserves $N$. In fact, we
show that $\theta_s$ preserves the leaves
of the foliation
$$
     N=N(0)\cup
     \bigcup_{\substack{T>\tau\\\gamma\in S^u_\eps}}
     N(\gamma_T).
$$
By Theorem~\ref{thm:inv-fol} these
leaves are infinite dimensional open disks.
The idea is to show that
the function $(0,\infty)\ni s\mapsto\Ss(\theta_s z)$
strictly decreases whenever $z$ lies in the
topological  boundary of a leaf. 
This implies preservation of leaves as follows.
Firstly, note that $\theta$ is actually defined on a
neighborhood of $N(\gamma_T)$ in
$F(\gamma_T):=\Gg^T_\gamma(\Bb^+)$.
Secondly, the topological boundary of each leaf
lies on action level $c+\eps$ whereas the leaf itself
lies strictly below that level. Thus the induced
semi-flow points inwards along the boundary. So
$\theta_s$ preserves leaves and therefore
the foliation $N$.
Thus $A$ is a strong deformation retract of
$N$.\footnote{
  A deformation retraction of a topological space $N$
  onto a subspace $A$ is a homotopy between
  the identity map on $N$ and a retraction.
  More precisely, it is a continuous map
  $\theta:[0,\infty]\times N\to N$
  such that $\theta_0=id_N$, $\theta_\infty|_A=id_A$,
  ($\theta_s|_A=id_A$ for every $s\in[0,\infty]$,)
  and $\theta_\infty:N\to A$ is called
  a {\bf (strong) deformation retraction}.
  Here $[0,\infty]$ denotes the
  one point compactification.
  In this case we say {\bf\boldmath $A$ is
  a (strong) deformation retract of $N$}.
  }

\vspace{.1cm}
{\it In the remaining part of the proof
we show that the function
$s\mapsto\Ss(\theta_s z)$
strictly decreases in $s>0$ whenever $z$ lies in the
topological  boundary of a leaf.}

\vspace{.1cm}
\noindent
To see this decompose the {\bf topological boundary},
that is closure take away interior,
of the isolating block $N=N^{\eps,\tau}$ in two
parts. The {\bf upper boundary} $\p^+ N$
is the part which intersects the level set
$\{\Ss=c+\eps\}$. Similarly the
{\bf lower boundary} $\p^- N$
is the part on which the action is strictly
less than $c+\eps$; see
Figure~\ref{fig:fig-induced-flow}. The
lower part is foliated by the leaves $N(\gamma_\tau)$
where $\gamma\in S^u_\eps$.
\\
Denote the $L^2$-gradient of $\Ss$ as usual
by $\grad\Ss$ and
note that it is defined only on loops of regularity
at least $W^{2,2}$. However, for $s>0$ the loops
$\phi_s z:S^1\to M$ and, slightly less obvious, also
$\theta_s z$ are $C^\infty$ smooth and therefore of
class $W^{2,2}$. 
Figure~\ref{fig:fig-neighborhood-Ww}
illustrates the closed neighborhood
$$
     \Ww:=\Bb_{\rho_0}\cap\{\Ss\le c+\eps/2\}
$$
of $0\in X$.
\begin{figure}
  \centering
  \includegraphics{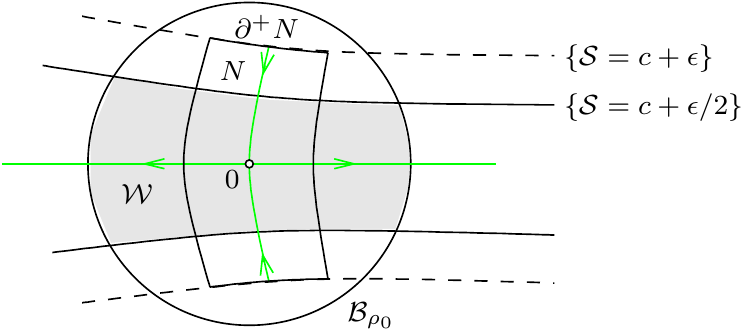}
  \caption{The complement of $\Ww$ in $\Bb_{\rho_0}$ is
           used to define $\alpha>0$}
  \label{fig:fig-neighborhood-Ww}
\end{figure}
Note that $\Ww$ is disjoint to the closed set $\p^+ N$.
Moreover, the constant
\begin{equation*}
     \alpha=\alpha(\rho_0,\eps):=
     \inf_{z\in\left(\Bb_{\rho_0}\cap W^{2,2}\right)\setminus \Ww}
     \Norm{\grad\Ss(z)}_2
     >0
\end{equation*}
is strictly positive. To see this assume $\alpha=0$.
Since $\Ss:W^{1,2}\to\R$ satisfies the
Palais-Smale condition there is a sequence $(z_k)$ in
$\left(\Bb_{\rho_0}\cap W^{2,2}\right)\setminus \Ww$
converging in $W^{1,2}$ to a critical point of $\Ss$
in $\Bb_{\rho_0}\setminus \Ww$. But this contradicts the
fact that, by our choice of $\rho_0$, the only
critical point in $\Bb_{\rho_0}$ is the origin
which lies in $\Ww$.

Assume $z$ is in the closure of $N$, that is
$z$ is in the closure of a leaf $N(\gamma_T)$
for some $T\ge\tau$ and $\gamma\in S^u_\eps$.
Recall from~\cite[(5)]{weber:2014a}
that in our coordinates $\grad\Ss$ is represented
by $A-f$ where $A=A_x$ is the Jacobi operator and
$f$ is the nonlinearity defined
by~\cite[(6)]{weber:2014a}.
By~\cite[Prop.~1~(b)]{weber:2014a} the operator $A$
preserves the vector space $X^-:=\pi_- X$ of
dimension $k>0$. The restriction $A^-$ lies in
$\Ll(X^-)$ and satisfies $\norm{A^-}=\abs{\lambda_1}$
where $\lambda_1<0$ denotes the smallest eigenvalue of
$A$. By definition of $\Gg_\gamma^T$ and
$\Gg^\infty$ in Theorems~1 and~3
in~\cite{weber:2014a} the difference
$$
     \theta_s z-\phi_s q
     =\Gg_\gamma^T(z_+(s))-\Gg^\infty(z_+(s))
     =\left(G_\gamma^T(z_+(s))-G^\infty(z_+(s)),0\right)
$$
lies in $X^-\subset C^\infty$.
This implies the first identity in the
estimate
\begin{equation}\label{eq:grad-diff}
\begin{split}
    &\Norm{\grad\Ss(\phi_s q)-\grad\Ss(\theta_s z)}_2
   \\
    &=\Norm{A^-(\phi_s q-\theta_s z)
     +f(\theta_s z)-f(\phi_s q)}_2
   \\
    &\le\left(\Abs{\lambda_1}+\kappa_0\right)
     \Norm{\theta_s z-\phi_s q}_{1,4}
   \\
    &=c_1
     \Norm{\Gg_\gamma^T(z_+(s))-\Gg^\infty(z_+(s))}_{1,4}
   \\
    &\le\rho_0 c_1 e^{-T\frac{\lambda}{16}}
\end{split}
\end{equation}
which holds for every $s>0$
and where $c_1:=(\abs{\lambda_1}+\kappa_0)$.
The first inequality also uses 
the Lipschitz Lemma~\cite[Le.~1]{weber:2014a}
for $f$ and $p=2$ with constant
$\kappa_0:=\kappa(\rho_0)$.
The final inequality is by~\cite[Thm.~1]{weber:2014a}.
Choose $\tau$ larger, if necessary, such that
\begin{equation}\label{eq:alpha/100}
     \rho_0 c_1e^{-\tau\frac{\lambda}{16}}
     \le\frac{1}{16}
     ,\qquad
     3\rho_0 c_1e^{-\tau\frac{\lambda}{16}}
     \le\frac{\alpha}{100}
     ,\qquad
     12\rho_0c_1
     e^{-\tau\frac{\lambda}{16}}
     \le\frac{\alpha^2}{8},
\end{equation}
and abbreviate
\begin{equation*}
     v_\pm=v_\pm(s):=\pi_\pm\grad\Ss(\theta_s z).
\end{equation*}
Apply the identity $\pi_-+\pi_+=\1$ and add
twice zero to obtain the estimate
\begin{equation}\label{eq:pi_-gradS}
\begin{split}
     \Norm{v_-}_2
    &=\Norm{\grad\Ss(\theta_s z)-v_+}_2
   \\
    &\le
     \Norm{\grad\Ss(\theta_s z)
     -\1\,\grad\Ss(\phi_s q)}_2
   \\
    &\quad
     +\Norm{d\Gg^\infty|_{z_+(s)}\pi_+
     \left(\grad\Ss(\phi_s q)
     -\grad\Ss(\theta_s z)\right)}_2
   \\
    &\quad
     +\Norm{d\Gg^\infty|_{z_+(s)} v_+-v_+}_2
   \\
    &\le
     3\Norm{\grad\Ss(\theta_s z)-\grad\Ss(\phi_s q)}_2
     +\frac{1}{4}\Norm{v_+}_2
   \\
    &\le
     3\rho_0 c_1e^{-T\frac{\lambda}{16}}
     +\frac{1}{4}\Norm{v_+}_2.
\end{split}
\end{equation}
To see the first zero which has been added
recall that (by definition of $\Gg^\infty$)
the projection $\pi_+$ restricted to the image $N(0)$
of $\Gg^\infty$ is the identity map on $N(0)$.
Linearization at the point $\phi_s q\in N(0)$
shows that
$
     d\Gg^\infty|_{z_+(s)}\pi_+
     =\1_{T_{\phi_s q} N(0)}
$.
The second inequality uses the two estimates
provided by~\cite[Prop.~3]{weber:2014a}.
The final inequality is by~(\ref{eq:grad-diff}).

From now on fix $z\in \p^+ N=\p^+ N^{\eps,\tau}$.
Observe that $z$
lies on action level $c+\eps$ and in the image of a
graph map $\Gg^T_\gamma$ where $\gamma\in S^u_\eps$
and $T>\tau$. (For $T=\tau$ there is nothing to prove.)
By continuity of $\theta$,
the downward gradient property, and openness
of $N$ there is a time $T_z>0$ such that
for each $s\in(0,T_z)$ the following holds.
The path $s\mapsto\theta_s z$ remains, firstly,
in $N$ and, secondly, above level $c+\frac{\eps}{2}$.
Thus $\theta_s z$, firstly, satisfies
estimates~(\ref{eq:grad-diff})--(\ref{eq:pi_-gradS})
and, secondly, remains in the complement of
$\Ww$ used to define $\alpha$.
By~(\ref{eq:pi_-gradS}) we get
\begin{equation}\label{eq:grad-Ss}
     \Norm{\grad\Ss(\theta_s z)}_2
     \le\Norm{v_-}_2+\Norm{v_+}_2
     \le 3\rho_0 c_1e^{-T\frac{\lambda}{16}}
     +\frac{5}{4}\Norm{v_+}_2
\end{equation}
which together with $T>\tau$
and the second assumption in~(\ref{eq:alpha/100})
implies that
\begin{equation}\label{eq:v_+-alpha}
     \Norm{v_+}_2
     >\frac{4}{5}\left(\Norm{\grad\Ss(\theta_s z)}_2
     -\frac{\alpha}{100}\right)
     >\frac{3}{4}\alpha
\end{equation}
for every $s\in(0,T_z)$. The final step
is by definition of $\alpha$.
Observe that
\begin{equation*}
\begin{split}
     \frac{d}{ds}\Ss(\theta_s z)
    &=d\Ss|_{\theta_s z} \,
     d\Gg_\gamma^T|_{z_+(s)} \,
     \pi_+  
     \tfrac{d}{ds}
     \left(\phi_s\Gg^\infty\pi_+ z\right)
   \\
    &=-\left\langle
     \grad\Ss|_{\theta_s z},
     d\Gg_\gamma^T|_{z_+(s)}
     \pi_+
     \grad\Ss|_{\phi_s q}
     \right\rangle_{L^2}
\end{split}
\end{equation*}
for every $s\in(0,T_z)$.
Here the second identity uses the definition
of the $L^2$-gradient and the fact that the semi-flow
$\phi_s$ is generated by $-\grad\Ss$.
Add three times zero to obtain that
\begin{equation}\label{eq:L2-extension-used}
\begin{split}
     \frac{d}{ds}\Ss(\theta_s z)
    &=
     -\left\langle
     \grad\Ss|_{\theta_s z},
     d\Gg_\gamma^T|_{z_+(s)}\pi_+
     \left(
       \grad\Ss|_{\phi_s q}
       -\grad\Ss|_{\theta_s z}
     \right)
     \right\rangle_{L^2}
   \\
    &\quad
     -\left\langle
     \grad\Ss|_{\theta_s z},
     \left(
       d\Gg_\gamma^T|_{z_+(s)}-d\Gg^\infty|_{z_+(s)}
     \right)\pi_+
     \grad\Ss|_{\theta_s z}
     \right\rangle_{L^2}
   \\
    &\quad
     -\left\langle
     \grad\Ss|_{\theta_s z},
     \left(
       d\Gg^\infty|_{z_+(s)}-\1
     \right)\pi_+
     \grad\Ss|_{\theta_s z}
     \right\rangle_{L^2}
   \\
    &\quad
     -\left\langle
     \grad\Ss|_{\theta_s z},\pi_+
     \grad\Ss|_{\theta_s z}
     \right\rangle_{L^2}
\end{split}
\end{equation}
for every $s\in(0,T_z)$.
At this point the $L^2$ extension
of the linearized graph maps enters.
Namely, use the difference
estimate~(\ref{eq:grad-diff}), the uniform
estimates for the linearized graph maps provided
by~\cite[Prop.~3]{weber:2014a}
and~\cite[Thm.~2]{weber:2014a}, and the
identity $\grad\Ss|_{\theta_s z}=v_-+v_+$ to get
\begin{equation*}
\begin{split}
     \frac{d}{ds}\Ss(\theta_s z)
    &\le
     \Norm{\grad\Ss(\theta_s z)}_2
     \left(
     2\rho_0 c_1e^{-T\frac{\lambda}{16}}
     +e^{-T\frac{\lambda}{16}}\Norm{v_+}_2
     \right)
   \\
    &\quad
     +\left(\Norm{v_-}_2+\Norm{v_+}_2\right)
     \frac{\Norm{v_+}_2}{4}
     -\Norm{v_+}_2^2
   \\
    &\le
     \left(
     3\rho_0 c_1e^{-T\frac{\lambda}{16}}+\frac{5}{4}\Norm{v_+}_2
     \right)
     \left(
     2\rho_0 c_1e^{-T\frac{\lambda}{16}}
     +e^{-T\frac{\lambda}{16}}\Norm{v_+}_2
     \right)
   \\
    &\quad
     +3\rho_0 c_1e^{-T\frac{\lambda}{16}}
     -\left(1-\frac{1}{4}-\frac{1}{16}
     \right)\Norm{v_+}_2^2
   \\
    &\le
     6\rho_0 c_1e^{-T\frac{\lambda}{16}}
     +6\rho_0 c_1 e^{-T\frac{\lambda}{16}}\Norm{v_+}_2
     -\frac{11}{16}\Norm{v_+}_2^2
   \\
    &\le
     12\rho_0 c_1e^{-T\frac{\lambda}{16}}
     -\frac{1}{2}\Norm{v_+}_2^2
   \\
    &\le-\tfrac{1}{4}\alpha^2
\end{split}
\end{equation*}
for every $s\in(0,T_z)$.
Consider the two lines after the first inequality.
Line one corresponds to the first two lines
in~(\ref{eq:L2-extension-used})
and line two corresponds to the last two lines;
in the last line orthogonality of $\pi_\pm$ enters.
Inequality two is by
estimate~(\ref{eq:grad-Ss}) for $\grad\Ss$
and~(\ref{eq:pi_-gradS}) for $v_-$.
To obtain inequality three we multiplied out the product
and used the first assumption 
in~(\ref{eq:alpha/100}). Inequality four
uses for the middle term
Young's inequality $ab\le\frac12 a^2+\frac12 b^2$
for $b=2^{-1}\norm{v_+}_2$ together
with the first assumption in~(\ref{eq:alpha/100}).
The final step uses the third assumption
in~(\ref{eq:alpha/100}) and
estimate~(\ref{eq:v_+-alpha}) for $v_+$.

This proves that the induced semi-flow $\theta_s$ is
inward pointing along the boundary of each leaf
$N(\gamma_T)$
and thereby completes the proof of
Theorem~\ref{thm:deformation-retract}.
\end{proof}

\begin{remark}\label{rem:changes}
The downward $L^2$-gradient nature of the heat
equation~(\ref{eq:heat}) causes the $L^2$ norm
to appear in estimates~(\ref{eq:grad-diff})
and~(\ref{eq:L2-extension-used}). The first
estimate involves the nonlinearity $f$ of the heat
equation.
To make sure that $f$ takes values in $L^2$ the
domain $W^{1,4}$ is the right choice;
see~\cite[(6)]{weber:2014a}.
The second estimate leads to the $L^2$ norms of
the linearized graph maps.
Cf.~\cite[Rmk.~1]{weber:2014a}.
\end{remark}

\subsection{Conley pairs}\label{sec:Conley-pairs}

\begin{proof}[Proof of Theorem~\ref{thm:Conley-pair}]
We need to verify properties~(i--iv)
in Definition~\ref{def:index-pair}.

{(i)}~Since $x$ is a fixed point of
the heat flow $\varphi$ and
$c:=\Ss_\Vv(x)=\Ss_\Vv(\varphi_{2\tau} x)$
it follows immediately that
$x\in N_x$ and $x\notin L_x$.
The latter conclusion also
uses continuity of the function
$\Ss_\Vv\circ\varphi_{2\tau}:\Lambda M\to \R$.
We only used $\eps,\tau>0$.

{(ii)}~For $\eps\in(0,\mu]$ and $\tau>\tau_0$
with $\mu$ and $\tau_0$
as in~(H4) of
Hypothesis~\ref{hyp:local-setup-N}
assertion~(ii) holds by
Theorem~\ref{thm:inv-fol}, that is
$N_x$ is an isolating block for~$x$.

{(iii)}~To prove that $L_x$ is positively 
invariant in $N_x$ it suffices to assume
$\gamma\in L_x$ and
$\varphi_s\gamma\in N_x$ for some
$s\ge 0$.\,\footnote{
  Using the downward gradient flow property
  this is equivalent to the usual hypothesis
  $\gamma\in L_x$ and
  $\varphi_{[0,s]}\gamma\subset N_x$ for some
  $s\ge 0$. (Use that our
  $N_x$ is path connected by definition.)
  }
It follows that $\varphi_s\gamma\in L_x$, because
$$
     \Ss_\Vv(\varphi_{2\tau}(\varphi_s\gamma))
     =\Ss_\Vv(\varphi_{2\tau+s}\gamma)
     \le\Ss_\Vv(\varphi_{2\tau}\gamma)
     \le c-\eps.
$$
Indeed the first step holds by
the semigroup property and the second step
by the downward gradient flow property.
The final step uses the assumption
$\gamma\in L_x$.

{(iv)}~Let $\eps$ and $\tau$ be as in~(H4)
Hypothesis~\ref{hyp:local-setup-N}.
Then Theorem~\ref{thm:inv-fol} applies, in
particular, there are no critical points other than
$x$ in the closure of $N_x$.
We need to verify that semi-flow trajectories
can leave $N_x$ only through $L_x$.
If $\gamma\in L_x$ and 
$\varphi_T\gamma\notin N_x$ the
assertions follow immediately from
openness of $N_x$, continuity of $\varphi$,
and the fact that $L_x$ is positively
invariant in $N_x$ by~(iii).
Now assume
that $\gamma\in N_x\setminus L_x$
and $\varphi_T\gamma\notin N_x$
for some time $T>0$.
Hence $\gamma\not= x$ and
$$
     \Ss_\Vv(\gamma)<c+\eps,\qquad
     \Ss_\Vv(\varphi_{2\tau}\gamma)>c-\eps,\qquad
     \Ss_\Vv(\varphi_{\tau+T}\gamma)\le c-\eps.
$$
Inequality three
excludes the case that $\gamma$
is in the ascending disk $W^s_\eps(x)$.
Thus by Theorem~\ref{thm:inv-fol} part~a)
the semi-flow trajectory through $\gamma$
reaches the action level $c-\eps$ 
in some finite time $T_*>\tau$.
In fact $T_*> 2\tau$
by inequality two.
Set $a:=T_*-2\tau>0$ to obtain that
$
          c-\eps
          =\Ss_\Vv(\varphi_{T_*}\gamma)
          =\Ss_\Vv(\varphi_{2\tau+a}\gamma)
$.
Set $b:=\tau+a>a$ 
to obtain that $T_*=2\tau+a=\tau+b$.
So the identity reads
$
          c-\eps
          =\Ss_\Vv(\varphi_{\tau+b}\gamma)
$.
Thus $b\le T$ by inequality three.
Next we show that
$a$ is the unique time
at which the orbit through
$\gamma$ enters $L_x$
and $b$ is the unique time
when it leaves $L_x$.

More precisely, we show that
$\varphi_s\gamma\in N_x$
if and only if $s\in[0,b)$ and that
$\varphi_s\gamma\in L_x$
if and only if $s\in[a,b)$.
To see the first of these two statements
pick $s\in[0,b)$.
Then 
$
     \Ss_\Vv(\varphi_s\gamma)
     \le\Ss_\Vv(\gamma)
     <c+\eps
$
since $\gamma\in N_x$.
Furthermore, note that
$\tau+s<\tau+b=2\tau+a=T_*$.
So
$
     \Ss_\Vv(\varphi_\tau(\varphi_s\gamma))
     =\Ss_\Vv(\varphi_{\tau+s}\gamma)
     >\Ss_\Vv(\varphi_{T_*}\gamma)
     =c-\eps.
$
The inequality is strict
since $\gamma\not= x$.
Vice versa, assume $\varphi_s\gamma\in N_x$.
Since this only makes sense for $s\ge0$
it remains to show $s<b$, equivalently
$s+\tau<T_*$. The latter follows from the fact
that $\Ss_\Vv(\varphi_{\tau+s}\gamma)>c-\eps$
since $\varphi_s\gamma\in N_x$
and the fact that
$\Ss_\Vv(\varphi_{T_*}\gamma)=c-\eps$
together with the downward gradient flow property.
\\
To see the second statement
pick $s\in[a,b)$.
Since $[a,b)\subset[0,b)$, the first statement
tells $\varphi_s\gamma\in N_x$.
So it remains to show
$
     \Ss_\Vv(\varphi_{2\tau}(\varphi_s\gamma))
     \le c-\eps
$
which is equivalent to $2\tau+s\ge T_*$. Indeed
$2\tau+s\ge 2\tau+a=T_*$
by our choice of $s$ and definition of $a$.
Vice versa, assume $\varphi_s\gamma\in L_x$
for some $s>0$. Then we get the two inequalities
$
     \Ss_\Vv(\varphi_\tau(\varphi_s\gamma))
     >c-\eps
$
and
$
     \Ss_\Vv(\varphi_{2\tau}(\varphi_s\gamma))
     \le c-\eps
$
by definition of $L_x$.
If $s\ge b$, equivalently
$\tau+s\ge \tau+b=T_*$,
we get
$
     \Ss_\Vv(\varphi_{s+\tau}\gamma)
     \le \Ss_\Vv(\varphi_{T_*}\gamma)
     =c-\eps
$ which contradicts inequality one.
In the case $s\in(0,a)$
we get
$
     \Ss_\Vv(\varphi_{2\tau+s}\gamma)
     >\Ss_\Vv(\varphi_{T_*}\gamma)
     =c-\eps
$ which contradicts inequality two.

Pick any $\sigma\in[a,b)\subset(0,T)$
to conclude the proof of~(iv).
Indeed $\varphi_{[0,\sigma]}\gamma\subset N_x$
by the first statement
(and the assumption
$\varphi_0\gamma\in L_x\subset N_x$)
and $\varphi_\sigma\gamma\in L_x$
by the second statement.
This concludes the proof of
Theorem~\ref{thm:Conley-pair}.
\end{proof} 

\begin{proposition}[Strong deformation retract]
\label{prop:N-L}
The Conley pair $(N_x,L_x)$ in
Theorem~\ref{thm:Conley-pair}
strongly deformation retracts
to its part $(N_x^u,L_x^u)$
in $W^u(x)$, i.e.
\begin{equation*}
\begin{split}
     \left(N_x,L_x\right)
     \simeq
     \left(N_x^u,L_x^u\right)
     =\left(
     \varphi_{-\tau} W^u_\eps(x),
     \varphi_{[-2\tau,-\tau)} S^u_\eps(x)\right).
\end{split}
\end{equation*}
Here the final pair of spaces consists of an open
$k$-disk, see~(\ref{eq:N-cap-W}), and
a (relatively) closed annulus which arises by
removing the smaller $k$-disk
$\varphi_{-2\tau} W^u_\eps(x)$.
\end{proposition}

\begin{proof}
The assertions for $N_x=N_x^{\eps,\tau}$ are true
by Theorem~\ref{thm:deformation-retract}
and~(\ref{eq:N-cap-W}).
Concerning $L_x=L_x^{\eps,\tau}$
pick $z\in N_x\setminus\{ x\}$.
By Theorem~\ref{thm:inv-fol} part~a)
this means that
$$
     z\in N_x(\gamma_T)
     =\left({\varphi_T}^{-1}\Dd_\gamma (x)\cap
     \{\Ss<c+\eps\}\right)_{\gamma_T}
     ,\qquad
     \gamma_T:=\varphi_{-T}\gamma,
$$
for some $\gamma\in S^u_\eps(x)$ and $T>\tau$.
Thus $z$ reaches action level $c-\eps$
under the semi-flow in time
$T\in(\tau,2\tau]$ if and only if
$\Ss_\Vv(\varphi_{2\tau} z)\le c-\eps$.
This shows that
$$
     L_x
     =\bigcup_{(T,\gamma)\in(\tau,2\tau]\times S^u_\eps}
     N_x(\gamma_T)
$$
since $L_x\subset N_x$.
Therefore $L_x$ carries the structure of a foliation
whose leaves are given by the corresponding leaves
of $N_x$. Thus the restriction to $L_x$
of the (leaf preserving)
strong deformation retraction $\theta$
of $N_x$ onto $N_x\cap W^u(x)$
given by~(\ref{eq:str-def-retract})
is a strong deformation retraction
of $L_x$ onto its part in the unstable
manifold.
This proves the first assertion.
Intersect the second identity
in~(\ref{eq:N-cap-W}) with $L_x$
to obtain the second assertion.
Concerning dimensions note that
the disks and the annulus
are open subsets of the unstable manifold
$W^u(x)$ whose dimension is the Morse index
$k$ of $x$ by~\cite[Thm.~18]{weber:2013b}.
\end{proof}

\subsubsection*{Homology of Conley pairs}\label{sec:Conley-homotopy-index}

\begin{definition}[Canonical orientations]
\label{def:canonical-orientation}
Given $k\ge 1$ we denote by
$\D^k$ the closed unit disk in $\R^k$.
The {\boldmath\bf canonical orientations
of $\R^k$ and 
$\D^k$} are provided by the
(ordered) canonical basis
$\Ee=(e_1,\dots,e_k)$ of $\R^k$.
The induced orientation of the boundary
$\p \D^k=\SS^{k-1}$, called
{\bf canonical boundary orientation}, is given
by putting the outward normal in
slot one, that is by declaring the sum
\begin{equation}\label{eq:standard-orientation}
     \R^k=\R \xi\oplus T_\xi \SS^{k-1}
\end{equation}
an oriented sum for each
$\xi\in\SS^{k-1}\subset\R^k$.
By definition an {\bf orientation of a point}
is a sign. With this convention the
canonical orientation of each point of
the 0-sphere $\SS^0=\{-1,+1\}\subset\R^1$
is provided by its own sign.
By definition 
$
     \D^0=\{0\}=\R^0
$
and
$
     \SS^{-1}=\p\D^0=\emptyset
$.
For $k\ge 1$ the {\bf positive generators}
$$
     a_k =[\D^k_{\langle\mathrm{can}\rangle}]
     \in\Ho_k(\D^k,\SS^{k-1})
     ,\qquad
     b_{k-1}=[\SS^{k-1}_{\langle\mathrm{can}\rangle}]
     \in\Ho_{k-1}(\SS^{k-1}),
$$
are given, respectively, by the class of the relative
cycle $\D^k$ equipped with its canonical orientation
and the class of $\SS^{k-1}$ with its canonical
orientation .
The 0-sphere $\SS^0=\{q,p\}\subset\R^1$,
where $q=-1$ and $p=+1$, is canonically oriented
by the boundary orientation of $\D^1=[-1,1]$.
The connecting homomorphism $\p$ maps $a_1$
to $b_0=[p-q]\in\Ho_0(\SS^0)\cong\Z^2$.
\end{definition}

\begin{figure}
  \centering
  \includegraphics{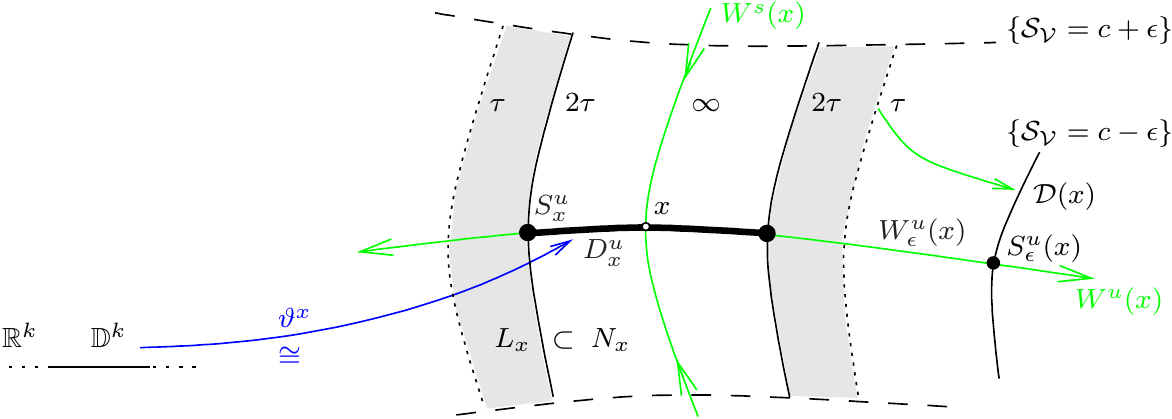}
  \caption{The $k$-disk $D^u_x\subset N_x$
      and its bounding sphere $S^u_x\subset L_x$
           }
  \label{fig:fig-D_u_x}
\end{figure}

\begin{theorem}[Homology of Conley pairs]
\label{thm:Conley-index}
Given a nondegenerate critical point $x$
of Morse index $k$ and one of the Conley pairs
$(N_x,L_x)=(N_x^{\eps,\tau},L_x^{\eps,\tau})$
provided by Theorem~\ref{thm:Conley-pair}.
Fix a diffeomorphism\footnote{
  Use the Morse Lemma to
  define a diffeomorphism 
  $\D^k\cong \overline{W^u_\eps(x)}$ and
  recall from Remark~\ref{rmk:diffeo-unstable-NEW}
  that restricted to the unstable manifold $W^u(x)$
  the heat flow turns into a genuine flow,
  then apply the diffeomorphism
  $\varphi_{-2\tau}|_{W^u(x)}$.
  }
\begin{equation}\label{eq:vartheta-2}
     \vartheta^{x}:\D^k\to D^u_x
     :=\varphi_{-2\tau}\overline{W^u_\eps(x)}
\end{equation}
between the closed unit disk $\D^k\subset\R^k$
and the disk $D^u_x$ which is contained in
$N_x\cap W^u(x)$ and whose boundary is given by
$
     S^u_x
     :=\p D^u_x=\varphi_{-2\tau} S^u_\eps(x)
$
and lies in the exit set $L_x$; see
Figure~\ref{fig:fig-D_u_x}.
Then there are the isomorphisms
\begin{equation}\label{eq:nat-iso}
\xymatrix{
     \Ho_*(\D^k,\SS^{k-1})
     \ar[r]^{\vartheta^x_*}_\cong
   &
     \Ho_*(D^u_x,S^u_x)
     \ar[r]^{\iota_*}_\cong
   &
     \Ho_*(N_x,L_x)
     }
\end{equation}
which are non-trivial only in degree
$k=\IND_\Vv(x)$ and where $\iota$
denotes inclusion. Furthermore, it holds that
$(\iota\circ\vartheta^x)_*:
[\D^k]\mapsto [D^u_x]\mapsto [D^u_x]$.
\end{theorem}

\begin{proof}
Since $\vartheta^x:\D^k\to D^u_x$ is a diffeomorphism
which maps $\p\D^k$ to $S^u_x$ it induces
an isomorphism on relative homology.
Thus the image $D^u_x$ of the relative
cycle $\D^k$ represents
one of two generators
of $\Ho_*(D^u_x,S^u_x)\cong\Z$.
  To distinguish them
  one needs to specify an orientation
  of $D^u_x$; see
  Definition~\ref{def:sign-vartheta}.
By~(\ref {eq:N-cap-W}) the boundary $S^u_x$ of
$D^u_x$ is $\varphi_{-2\tau} S^u_\eps(x)$ and it
lies in $L_x$ by Proposition~\ref{prop:N-L}.
Hence the inclusion
$\iota:(D^u_x,S^u_x)\hookrightarrow (N_x,L_x)$
provides an element of $\Ho_k(N_x,L_x)$ denoted
by $\iota_*[D^u_x]=[\iota(D^u_x)]$ or simply by
$[D^u_x]$.
To see that $\iota_*[D^u_x]$ is actually a basis --
in other words, that the inclusion $\iota$ induces
an isomorphism --
recall that
$
     (N^u_x,L_x^u)=(N_x\cap W^u(x),L_x\cap W^u(x))
$
and consider the homomorphisms
\begin{equation}\label{eq:idea-incl}
\xymatrix{
     \Ho_*(D^u_x,S^u_x)
     \ar[r]^{\iota_*}
   &
     \Ho_*(N_x,L_x)
     \ar[r]^{\theta_*}_\cong
   &
     \Ho_*(N_x^u,L_x^u)
     \ar[r]^{r_*}_\cong
   &
      \Ho_*(D^u_x,S^u_x).
     }
\end{equation}
Here $\theta:=\theta^\infty:N_x\to N^u_x$
is the strong deformation
retraction~(\ref{eq:str-def-retract})
referred to by
Theorem~\ref{thm:deformation-retract}
and $r=h_1:N^u_x\to D^u_x$ is
the strong deformation retraction to be
defined below.
Because both deformation retractions are strong,
we get that $r_*\theta_*\iota_*[D^u_x]
=[id(id(\iota(D^u_x))]=[D^u_x]$. But $[D^u_x]$
generates $\Ho_*(D^u_x,S^u_x)$ and so $\iota_*$
has to be injective. Moreover, since
isomorphisms map bases to bases and
${\theta_*}^{-1}{r_*}^{-1}([D^u_x])
=\iota_*[D^u_x]$ it follows that $\iota_*$
is surjective, thus an isomorphism.

It remains to construct a map
$h:[0,1]\times N^u_x\to N^u_x$,
$(\lambda,\gamma)\mapsto h_\lambda(\gamma)$,
providing a homotopy between $h_0=id_{N^u_x}$
and $r:=h_1:N^u_x\to D^u_x$ and such that
$h_\lambda|_{D^u_x}=id_{D^u_x}$ for every
$\lambda\in[0,1]$. Consider the annuli
$X\supset A$ given by
$$
     X
     :=W^u(x)\setminus\INT D^u_x
     =W^u(x)\setminus\varphi_{-2\tau} W^u_\eps(x)
     ,\qquad
     A:=W^u(x)\setminus W^u_\eps(x),
$$
and the entrance time function
$
     \Tt_A:X\mapsto[0,2\tau]
$
as defined by~(\ref{eq:entrance-time}) below
while constructing the third
isomorphism in the proof of
Theorem~\ref{thm:cellular-filtration}.
By arguments analogous to the ones used
during that construction
$\Tt_A$ is lower semi-continuous
by closedness of $A\subset X$
and upper semi-continuous by
(forward) semi-flow invariance of $A$ in $X$.
Then the map defined by
\begin{equation*}
     h_\lambda(\gamma):=
     \begin{cases}
        \gamma
        &\text{, $\gamma\in D^u_x$,}
        \\
        \varphi_{\lambda(\Tt_A(\gamma)-2\tau)}\gamma
        &\text{, $\gamma\in N^u_x\setminus\INT D^u_x$,}
     \end{cases}
\end{equation*}
has all the desired properties. It is well defined
since $\Tt_A$ vanishes on $\p D^u_x$.
\end{proof} 

\begin{definition}\label{def:sign-vartheta}
(i)~In the setting of Theorem~\ref{thm:Conley-index}
assume $\D^k$ carries the canonical orientation.
Pick an orientation $\langle x\rangle$
of $W^u(x)$. Then
\begin{equation}\label{eq:sign-vartheta}
     \sigma_{\langle x\rangle}
     :=
     \begin{cases}
       +1&\text{, if $\vartheta^x:\D^k\to W^u(x)$
       preserves orientation,}\\
       -1&\text{, otherwise.}
     \end{cases}
\end{equation}
is called the
{\bf sign} of $\vartheta^x$
with respect to $\langle x\rangle$.

(ii)~Consider the linear transformation
$\mu:=\diag(-1,1,\dots,1)\in\R^{k\times k}$.
It is an orientation reversing
diffeomorphism of $\R^k$ and of $\D^k$.
With the conventions
\begin{equation}\label{eq:sign-kappa}
     \mu^0=\1,\qquad
     \kappa_{\langle x\rangle}
     =\tfrac12 \left(1+\sigma_{\langle x\rangle}\right)
     \in\{0,1\}
\end{equation}
we get the identity of induced isomorphisms
\begin{equation}\label{eq:vartheta-or}
\begin{split}
     \sigma_{\langle x\rangle}\vartheta^x_*=
     (\vartheta^x\circ\mu^{\kappa_{\langle x\rangle}})_k
     :\Ho_*(\D^k,\SS^{k-1})
     \to\Ho_k\bigl(D^u_x,S^u_x\bigr)
\end{split}
\end{equation}
which map the positive generator
$a_k=[\D^k_{\langle\mathrm{can}\rangle}]$ is
to the generator $[D^u_{\langle x\rangle}]$
of $\Ho_k\bigl(D^u_x,S^u_x\bigr)\cong\Z$.
Here $D^u_{\langle x\rangle}$ denotes
the relative cycle $D^u_x$ oriented by
$\langle x\rangle$.
\end{definition}

\section{Morse filtration and natural isomorphism}
\label{sec:filtration-iso}

In section~\ref{sec:filtration-iso}
we construct the natural isomorphism
in Theorem~\ref{thm:main},
in other words, we calculate singular homology
of the sublevel set $\Lambda^a M$
in terms of the homology of the Morse~complex
$
     \left(\CM^a_*(V),
     \p^M_*(V,v_a)\right)
$
defined in section~\ref{sec:main-results}.
Recall that the chain group $\CM^a_*(V)$
is the free Abelian group generated by
oriented critical points
$\langle x\rangle\in\Crit^a$
of the Morse function $\Ss_V$
-- without assigning
the role of a distinct generator to one
of the two possible orientations
since we divide out subsequently by the
relation~(\ref{eq:relation}).
The Morse boundary operator counts
heat flow trajectories $u$ between
critical points of Morse index difference
one according to how the corresponding
push-forward orientations
$u_*\langle x\rangle$
match at the lower end.

The key idea is to consider an intermediate
chain complex associated to a cellular filtration
which, on the level of homology,
is already known to be naturally 
isomorphic to singular homology.
On the other hand, the additional geometric data 
provided by the Morse-Smale function $\Ss_\Vv$
given by~(\ref{eq:action-fct})
gives rise to a very particular filtration,
namely, a Morse filtration
whose associated cellular chain complex
equals the Morse complex up to natural identification.
In the case of a finite dimensional manifold
this idea has been used by
Milnor~\cite{milnor:1965a} in the context of
a \emph{self-indexing}\footnote{
  Self-indexing means that $f(x)=k$
  whenever $x$ is a critical
  point of $f$ of Morse index $k$.
  }
Morse function $f:M\to\R$
in which case just the sublevel sets
$F_k:=f^{-1}((-\infty,k+\tfrac12])$
itself provide a Morse filtration.
For a Banach manifold with a genuine
\emph{flow} generated by a $C^1$ vector field
a suitable filtration has been constructed by
Abbondandolo and Majer~\cite{abbondandolo:2006a}
who, moreover, provide full details of their
construction of an isomorphism (depending on
choices of orientations) between Morse
and singular homology.

Obviously the Hilbert manifold
of $W^{1,2}$ loops in $M$ is the natural
domain of the action functional $\Ss_\Vv$
and its Hilbert manifold structure
facilitates the analysis. Moreover, the space
$\Lambda^a M$ of $W^{1,2}$ loops in $M$
whose action is less or equal than $a$
is homotopy equivalent
to its subset $\Ll^a M$ of smooth loops
(see e.g.~\cite[\S~17]{milnor:1963a}
or footnote\footnote{
   Theorem (Palais, \cite[Thm.~16]{palais:1966a}).
   Given a Banach space $\Lambda$, a dense subspace
   $\Ll$, and an open subset
   $\Lambda^a\subset\Lambda$.
   Then the inclusion
   $\Lambda^a\cap\Ll\hookrightarrow\Lambda^a$
   is a homotopy equivalence.
   }).
Thus singular homology of both spaces
is naturally isomorphic and
Theorem~\ref{thm:main}
covers~\cite[Thm.~A.7]{salamon:2006a}.
Furthermore, it is not necessary that the potential
$\Vv$ is a sum~(\ref{eq:perturbation-MS})
of a geometric potential $V$ and an abstract
perturbation $v_a$. All we need is that
$\Vv$ satisfies axioms~{\rm (V0)--(V3)}
in~\cite{weber:2013b} and is
\emph{Morse-Smale below the regular level $a$}
in the functional analytic sense
of~\cite[\S 1]{weber:2013b}.
Any $\Vv$ that satisfies~{\rm (V0)--(V3)}
gives rise to a $C^1$ \emph{semi}-flow
\begin{equation}\label{eq:varphi-a}
     \varphi:(0,\infty)\times\Lambda^a M
     \to\Lambda^a M
     ,\qquad
     \Lambda^a M:=\{\Ss_\Vv\le a\},
\end{equation}
which extends continuously to zero;
see e.g.~\cite{weber:2010a-link}.

In what follows we construct
the natural isomorphism for
the semi-flow~(\ref{eq:varphi-a}).
For simplicity think of $\Vv$
as given by~(\ref{eq:perturbation-MS}).
To avoid overusing the word
'continuous' all maps are assumed
to be continuous
unless specified differently.

\subsection{Morse filtration}\label{sec:Morse filtration}

Assume $\Vv$ is a perturbation
that satisfies axioms~{\rm (V0)--(V3)}
in~\cite{weber:2013b} and $\Ss_\Vv$ is
Morse-Smale below the regular level $a$.
We construct a Morse filtration
$\Ff=\left( F_k\right)$ associated to
$\Ss_\Vv:\Lambda^a M\to\R$ such that,
in addition, each set $F_k$ is \emph{open}
and \emph{semi-flow invariant}.

Consider the closed ball $B_x^\rho$
of radius $\rho>0$ about $x$ with respect
to the $W^{1,2}$ metric on $\Lambda M$.
Since $a$ is a regular value and
the critical points are nondegenerate
there is a sufficiently small radius
$\rho=\rho(a)>0$ such that 
\begin{equation}\label{eq:rho}
     B_x^\rho\subset\Lambda^a M,\qquad
     B_x^\rho\cap B_y^\rho=\emptyset,
\end{equation}
for any two distinct elements $x$ and $y$
of the finite set $\Crit^a$.
The Morse-Smale condition guarantees that
there are no flow lines from one critical
point to another one of equal or larger
Morse index.
The following lemma
generalizes this principle, firstly,
to small neighborhoods
(cf.~\cite[Lemma~2.5]{abbondandolo:2006a})
and, secondly, to semi-flows.
More precisely, the lemma guarantees
that the Morse index strictly
decreases whenever there is a flow trajectory
from $B_x^\rho$ to $B_y^\rho$
and $\rho>0$ is sufficiently small.
We postpone proofs.

\begin{lemma}[Morse-Smale on neighborhoods]
\label{le:34}
There is a constant 
$\rho=\rho(a)>0$ such that
the pre-images ${\varphi_s}^{-1}B_y^\rho$ satisfy
\begin{equation}\label{eq:34}
      B_x^\rho\cap
     {\varphi_s}^{-1}B_y^\rho
     =\emptyset, \quad\forall s\ge 0,
\end{equation}
for all pairs of distinct critical
points $x,y\in\Crit^a$ 
with $\IND_\Vv(x)\le\IND_\Vv(y)$.
\end{lemma}

\begin{hypothesis}\label{hyp:rho-eps-tau}
Assume the perturbation $\Vv$
satisfies~{\rm (V0)--(V3)}
in~\cite{weber:2013b}
and the Morse-Smale condition holds
below the regular level $a$ of $\Ss_\Vv$.
\begin{enumerate}
\item[(H5)]
Fix a constant $\rho=\rho(a)>0$
sufficiently small such 
that~(\ref{eq:rho}) and~(\ref{eq:34}) hold true
and such that for each critical point
$x\in\Crit^a$ the local coordinate chart
$(\Phi,\Phi(B^u\times B^+))$ about
$x\in\Lambda M$ covers the ball $B_x^{2\rho}$.
Here $B^u\times B^+\subset X^-\oplus X^+$
is a product of balls contained in $\Bb_{\rho_0}$
with $B^u\subset W^u$; see
Hypothesis~\ref{hyp:local-setup-N}~(H1).
Pick constants $\eps>0$ sufficiently small
and $\tau>0$ sufficiently
large\footnote{
   In the notation
   of Theorem~\ref{thm:Conley-pair}
   pick $\eps\in(0,\mu(a)]$ and
   $\tau>\tau_0(a)$.
   }
such that for each $x\in\Crit^a$
Theorem~\ref{thm:inv-fol}
(Invariant stable foliation)
and Theorem~\ref{thm:Conley-pair} (Conley pair)
hold true. In particular, every $x\in\Crit^a$
admits a Conley pair, namely
$(N_x,L_x)=(N_x^{\eps,\tau},L_x^{\eps,\tau})$
defined by~(\ref{eq:Conley-set-NEW})
and~(\ref{eq:L_x}). By Theorem~\ref{thm:inv-fol}
part~d) we assume that $N_x\subset B_x^\rho$.
Consequently $N_x\cap N_y=\emptyset$
whenever $x\not= y$.
\end{enumerate}
\end{hypothesis}

From now on we assume
Hypothesis~\ref{hyp:rho-eps-tau}
and use the notation
\begin{equation}\label{eq:N_k}
     N_k
     :=\bigcup_{x\in\Crit^a_k} N_x,
     \qquad
     L_k
     :=\bigcup_{x\in\Crit^a_k} L_x,
     \qquad k\in\Z.
\end{equation}
By definition a union over the empty set
is the empty set. Since $N_x\subset B_x^\rho$
both unions are unions of disjoint sets
by~(\ref{eq:rho}).
We denote the {\bf maximal Morse index}
among the critical points
{\bf\boldmath below level $a$} by
\begin{equation}\label{eq:max-morse-index}
     m=m(a):=\max_{x\in\Crit^a} \IND_\Vv(x).
\end{equation}
Observe that $\Crit_0^a\not=\emptyset$
since the action $\Ss_\Vv$ is bounded below.
For such a critical point $x$ of Morse index $0$
the Conley index pair $(N_x,L_x)$
consists of the ascending disk
$N_x=N_x(x)=W^s_\eps(x)$
by Theorem~\ref{thm:inv-fol} part~a)
and the empty exit set $L_x=\emptyset$.
Note that the ascending disk
$
     W^s_\eps(x):=W^s(x)\cap\{\Ss_\Vv<\Ss_\Vv(x)+\eps\}
$
is open and semi-flow invariant.
Hence $N_0$ is a finite union
of (open and semi-flow invariant)
disjoint ascending
disks and $L_0=\emptyset$.
Next observe that for each $T>0$ the set
$
     F_0=F_0(T):={\varphi_T}^{-1}N_0
$
is semi-flow invariant. By continuity
of $\varphi_T$ it is also open.
Assume $k>0$ is the next larger realized Morse
index, that is $k$ is the minimal Morse
index among the elements of
$\Crit^a\setminus\Crit^a_0$.
Consider the unstable manifold
of a critical point $x_k$ of Morse index $k$.
Each element $\gamma\not= x_k$
moves in finite time $T_\gamma$ into the
neighborhood $N_0$ of $\Crit_0$
by existence of the asymptotic forward
limit~\cite[Thm.~9.14]{weber:2010a-link}.
The Morse-Smale condition
guarantees that the Morse index of the
asymptotic forward limit is strictly less than $k$,
thus indeed zero by minimality of $k$.
Hence $\gamma\in {\varphi_{T_\gamma}}^{-1}N_0$.
In fact, a much stronger
statement is true: There is a time $T_k>0$
such that the pre-image
${\varphi_{T_k}}^{-1}N_0$ contains
all elements $\gamma$ of the
infinite dimensional exit set $L_k$ of $N_k$.

\begin{proposition}[Uniform time]
\label{prop:entrance-time}
Given Hypothesis~\ref{hyp:rho-eps-tau},
suppose $A$ is an open semi-flow invariant subset
of $\Lambda^a M$ containing all critical points 
of Morse index less or equal to $k$ and no others.
In the case $k<m(a)$
there is a time $T_{k+1}\ge 0$ such that 
$L_{k+1}\subset {\varphi_{T_{k+1}}}^{-1}A$.
If $L_{k+1}=\emptyset$, set $T_{k+1}:=0$.
In the case $k=m(a)$ of maximal Morse index
there is a time $T_{m+1}\ge0$ such that
$
     \Lambda^a M
     ={\varphi_{T_{m+1}}}^{-1}A
$.
\end{proposition}

\subsubsection*{Definition of the Morse filtration}
The first step in the construction
of the Morse filtration $\Ff=\left( F_k\right)_{k\in\Z}$
associated to $\Ss_\Vv:\Lambda^a M\to\R$ is
to set $F_k:=\emptyset$ whenever $k<0$.
Now consider the time $T_1$ given by
Proposition~\ref{prop:entrance-time}
for $A=N_0$. It provides the crucial inclusion
$$
     L_1\subset{\varphi_{T_1}}^{-1}N_0=:F_0
$$
illustrated by Figure~\ref{fig:fig-Morse-filtration}.
\begin{figure}
  \centering
  \includegraphics{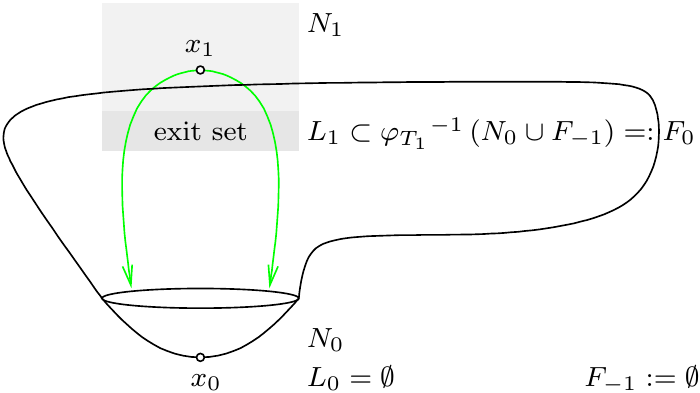}
  \caption{
           Morse filtration $\Ff=\left(\emptyset\subset
           F_0\subset
           F_1\subset\dots\subset F_{m}=\Lambda^a M\right)$
           }
  \label{fig:fig-Morse-filtration}
\end{figure}
Because the exit set $L_1$ of $N_1$ is contained 
in the semi-flow invariant set $F_0$,
the union $N_1\cup F_0$ is semi-flow invariant as well.
Trivially it is also open.
Next consider the time $T_2$ provided by
Proposition~\ref{prop:entrance-time}
for $A=N_1\cup F_0$. Hence
$$
     L_2\subset{\varphi_{T_2}}^{-1}\left(N_1\cup F_0\right)=:F_1
$$
and $F_1$ is open and semi-flow invariant
by the same reasoning as above.
Note that if there are no critical points
of Morse index $1$, then
$F_1={\varphi_{0}}^{-1}(\emptyset\cup F_0)=F_0$.
Proceeding iteratively we obtain a sequence of
open semi-flow invariant subsets
$$
     \emptyset=F_{-1}\subset F_0\subset F_1\subset\ldots\subset
     F_m=\Lambda^a M.
$$
More precisely, recalling that
$\varphi_T:\Lambda^a M\to\Lambda ^a M$
for any $T\ge 0$
we set
\begin{equation}\label{eq:F_k}
     F_{k}
     :={\varphi_{T_{k+1}}}^{-1}\left(N_k\cup F_{k-1}\right)
     \supset L_{k+1},\qquad
     k=0,\ldots,m-1,
\end{equation}
and
\begin{equation}\label{eq:F_m}
     F_m
     :={\varphi_{T_{m+1}}}^{-1}\left(N_m\cup F_{m-1}\right)
     =\Lambda^a M.
\end{equation}
Here $T_{k+1}$ is the time associated by
Proposition~\ref{prop:entrance-time}
to the set $A=N_k\cup F_{k-1}$.
Note that if there are no critical points
whose Morse index is $k$ or $k+1$,
then $F_k=F_{k-1}$ and $F_{k+1}={\varphi_{T_{k+2}}}^{-1}(F_{k-1})$.
Set $F_\ell:=\Lambda^a M$ whenever $\ell>m$.

\subsubsection*{Proofs}

The proof of Theorem~\ref{thm:cellular-filtration}
uses Proposition~\ref{prop:entrance-time} (Uniform time)
which relies on Lemma~\ref{le:34} (Morse-Smale on
neighborhoods). So we start with the

\begin{proof}[Proof of Lemma~\ref{le:34}
(Morse-Smale on neighborhoods)]
Assume the lemma is not true. 
Then there are critical points $x\not= y$
below level $a$ with $\IND_\Vv(x)\le\IND_\Vv(y)$,
sequences of constants
$\rho_\nu\searrow 0$ and $s_\nu\ge0$,
and a sequence of loops $\gamma^\nu\in B_x^{\rho_\nu}$
such that $\varphi_{s_\nu}\gamma^\nu\in B_y^{\rho_\nu}$.
Thus $\gamma^\nu$ converges to $x$ and 
$\varphi_{s_\nu}\gamma^\nu$ to $y$
in the $W^{1,2}$ topology, as $\nu\to\infty$.
Moreover, it follows that
$s_\nu\to\infty$, as $\nu\to\infty$.
To see the latter assume by contradiction
that the sequence $s_\nu$ is bounded.
Then there is a subsequence, still denoted by $s_\nu$,
such that $s_\nu$ converges to a constant $T\ge 0$.
By continuity of the semi-flow $\varphi$
we conclude that $\varphi_{s_\nu}\gamma^\nu$
converges in $W^{1,2}$ to $\varphi_T x$,
as $\nu\to\infty$. But $\varphi_T x=x$ since
critical points are fixed points.
Since $\varphi_{s_\nu}\gamma^\nu$
converges also to $y$ in $W^{1,2}$
we obtain the contradiction $x=y$.

Now consider the sequence of heat flow trajectories
$u^\nu:[0,s_\nu]\times S^1\to M$,
$$
     u^\nu(s,t):=\left(\varphi_s\gamma^\nu\right)(t).
$$
Since the action is nonincreasing
along heat flow trajectories and since
$\gamma^\nu\in B_x^{\rho_\nu}\subset \Lambda^a M$
it follows that
$$
     \max_{s\in[0,s_\nu]}\Ss_\Vv\left( u^\nu(s,\cdot)\right)
     \le \Ss_\Vv\left(\gamma^\nu \right)
     \le a.
$$
So we have a uniform action bound
on compact subcylinders of $[0,\infty)\times S^1$
for the sequence $u^\nu$ of heat flow trajectories.
By the arguments used to
prove~\cite[Prop.~3]{weber:2013b}
(Convergence on compact sets)
and~\cite[Le.~4]{weber:2013b}
(Compactness up to broken trajectories)
we obtain critical points $x=x_0,\dots,x_\ell=y$,
where $\ell\ge 1$,
and for each $k\in\{1,\ldots,\ell\}$
a connecting trajectory
$u_k\in\Mm(x_{k-1},x_k;\Vv)$
with $\p_s u_k\not\equiv 0$.
By the Morse-Smale condition
the Morse index of $x_k$ is strictly smaller than
the Morse index of $x_{k-1}$.
Thus $\IND_\Vv(x_0)>\IND_\Vv(x_m)$.
Contradiction.
\end{proof}

\begin{remark}\label{rmk:L2-grad}
The action functional 
$\Ss_\Vv:\Lambda M\to\R$,
$\gamma\mapsto\frac12\norm{\dot\gamma}_2^2
-\Vv(\gamma)$, is continuously differentiable.
To see this observe that
$$
     d\Ss_\Vv(\gamma)\xi
     =\langle\dot\gamma,\Nabla{t}\xi\rangle_{L^2}
     -\langle\grad\Vv(\gamma),\xi\rangle_{L^2}
$$
for all $\gamma\in\Lambda M$ and
$\xi\in W^{1,2}(S^1,\gamma^*TM)$.
Continuity of the first term is obvious and
for the second term it
follows from axioms~{\rm (V0)--(V1)}.
By definition the $L^2$-gradient of $\Ss_\Vv$
is determined by the identity
$
     d\Ss_\Vv(\gamma)\xi
     =\langle\grad\Ss_\Vv(\gamma),
     \xi\rangle_{L^2}
$
for all $\gamma\in\Lambda M$
and $\xi\in W^{1,2}(S^1,\gamma^*TM)$.
If $\gamma$ is of higher regularity $W^{2,2}$,
then we can carry out integration by parts
and $\grad\Ss_\Vv$ becomes a continuous section
of the Hilbert space bundle over $W^{2,2}(S^1,M)$
whose fiber over $\gamma$
is given by the Hilbert space $L^2(S^1,\gamma^*TM)$
of $L^2$ vector fields along $\gamma$.
In this case we obtain the explicit representation
$$
     \grad\Ss_\Vv(\gamma)
     =-\Nabla{t}\p_t\gamma-\grad\Vv(\gamma)
$$
whenever $\gamma\in W^{2,2}(S^1,\gamma^*TM)$.
\end{remark}

\begin{proof}
[Proof of Proposition~\ref{prop:entrance-time} (Uniform time).]
Key ingredients will be
Palais-Smale, Morse-Smale on neighborhoods,
and the fact that the action functional $\Ss_\Vv$
is bounded from below.
Recall Hypothesis~\ref{hyp:rho-eps-tau} on the
choices of $\Vv$, $\rho$, $\eps$, and $\tau$.

Fix $k<m(a)$ and pick an open semi-flow 
invariant subset $A\subset\Lambda^a M$
which contains $\Crit^a_{\le k}$ but no other
critical points.
Assume $L_{k+1}\not= \emptyset$, 
otherwise we are done by setting $T_{k+1}=0$.
Now assume by contradiction that there
is no time $T\ge0$ such that 
$\varphi_T L_{k+1}\subset A$.
In this case there are sequences
of positive reals $s_\nu\to\infty$
and of elements $\gamma^\nu$ of $L_{k+1}$ such that
$\varphi_{s_\nu}\gamma^\nu\notin A$ for
every $\nu\in\N$.
Choosing subsequences, still denoted by 
$s_\nu$ and $\gamma^\nu$, we may assume that
all $\gamma^\nu$ lie in the same
path connected component $L_x$ of $L_{k+1}$
for some $x\in\Crit^a_{k+1}$.
Here we use that $\Crit^a_{k+1}$ is a finite set since
$\Ss_\Vv$ is Morse below level $a$;
see~\cite {weber:2002a}.

Now consider the open neighborhood of
$\Crit^a$ in $\Lambda^a M$ defined by
$$
     U:=A\cup\left(N_x\setminus L_x\right)
     \cup 
     \bigcup_{y\in\Crit^a_{\ge k+1}\setminus\{x\}}
     N_y. 
$$
Indeed $A$ is open by assumption and so are the
neighborhoods $N_x$ and $N_x\setminus L_x$ of $x$
by Theorem~\ref{thm:Conley-pair}
and Definition~\ref{def:index-pair} of a Conley pair.
Note that
$$
     \kappa
     :=\inf_{\gamma\in\Lambda^a M\setminus U}
     \Norm{\grad\Ss_\Vv(\gamma)}_2
     >0
$$
is strictly positive.
To see this assume by contradiction
that $\kappa=0$. Then there is a sequence
$z^i$ in $\Lambda^a M\setminus U$
such that $\norm{\grad\Ss_\Vv(z^i)}_2\to 0$,
as $i\to\infty$.
So by Palais-Smale a subsequence
converges to some critical point 
in the closed set $\Lambda^a M\setminus U$.
But all critical points below level $a$ lie in the
open set $U$. Contradiction.

None of the elements
$\varphi_{s_\nu}\gamma^\nu$ of $\Lambda^a M$
lies in $U$: Indeed $\varphi_{s_\nu}\gamma^\nu\notin A$
by assumption. Furthermore, such an element cannot
lie in the union of the $N_y$'s,
because otherwise we would have a flow line
from $N_x\subset B_x^\rho$ to $N_y\subset B_y^\rho$
thereby contradicting Lemma~\ref{le:34}
(Morse-Smale on neighborhoods)
since $\IND_\Vv(x)\le\IND_\Vv(y)$.
It remains to check that
$\varphi_{s_\nu}\gamma^\nu\notin N_x\setminus L_x$.
To see this set $c:=\Ss_\Vv(x)$ and
recall that $\gamma^\nu$ lies in $L_x$
which is positively invariant in $N_x$ by
Definition~\ref{def:index-pair}~(iii).
Assume that the semi-flow trajectory
through $\gamma^\nu$ leaves $L_x$,
thus simultaneously $N_x$, say at a time $s_*$.
(Otherwise, if it stayed inside $L_x$ forever,
we are done.)
By definition of $N_x=N_x^{\eps,\tau}$
and the downward gradient property
the point $\varphi_{s_*}\gamma^\nu$
reaches the action level $c-\eps$
precisely after time $\tau$, that is
$\Ss_\Vv(\varphi_\tau(\varphi_{s_*}\gamma^\nu))
=c-\eps$. Since the action decreases along
heat flow trajectories we conclude that
$\Ss_\Vv(\varphi_\tau(\varphi_{s_*+s}\gamma^\nu))
\le c-\eps$ whenever $s\ge 0$.
Thus the semi-flow line through
$\varphi_{s_*}\gamma^\nu$ cannot
re-enter $N_x$ (nor its subset $L_x$).
To summarize we know that
$\varphi_{[0,s_*)}\gamma^\nu\subset L_x$ and
$\varphi_{[s_*,\infty)}\gamma^\nu\cap N_x=\emptyset$.
But this proves that 
$\varphi_{[0,\infty)}\gamma^\nu\cap 
\left( N_x\setminus L_x\right)=\emptyset$.

More generally, it even holds that
$\varphi_s\gamma^\nu\notin U$
whenever $s\in[0,s_\nu]$ and $\nu\in\N$:
Indeed $\varphi_s\gamma^\nu$
cannot lie in $A$, since $A$ is semi-flow
invariant by assumption and
$\varphi_{s_\nu}\gamma^\nu\notin A$. That
$\varphi_s\gamma^\nu\notin N_x\setminus L_x$
has been proved in the previous paragraph.
The statement for the union of the $N_y$'s follows by
the same Morse-Smale argument given
in the previous paragraph for $s=s_\nu$.

To finally derive a contradiction
use the fact that $\varphi_s$ is the semi-flow
generated by the negative $L^2$-gradient of
$\Ss_\Vv$ to obtain that
\begin{equation*}
\begin{split}
     \Ss_\Vv(\gamma^\nu)-\Ss_\Vv(\varphi_{s_\nu}\gamma^\nu)
    &=\int_{s_\nu}^0\frac{d}{ds}
     \Ss_\Vv(\varphi_s\gamma^\nu)\, ds\\
    &=\int_{s_\nu}^0d\Ss_\Vv|_{\varphi_s(\gamma^\nu)}
      \circ \left(\frac{d}{ds} \varphi_s\gamma^\nu\right)
      ds\\
    &=\int_0^{s_\nu}
     \Norm{\grad\Ss_\Vv
     \left(\varphi_s\gamma^\nu\right)}_2^2
     \, ds\\
    &\ge \kappa^2s_\nu
\end{split}
\end{equation*}
where the inequality uses the definition of $\kappa$
and the fact that $\varphi_s\gamma^\nu\notin U$
whenever $s\in[0,s_\nu]$.
Since $\kappa>0$, we get that
$$
     \Ss_\Vv(\varphi_{s_\nu}\gamma^\nu)
     \le\Ss_\Vv(\gamma^\nu)-\kappa^2s_\nu
     \le a-\kappa^2s_\nu\longrightarrow-\infty,\quad
     \text{as $\nu\to\infty$.}
$$
But this contradicts
the fact that $\Ss_\Vv$
is bounded from below by $-C_0$ where
$C_0$ is the constant in axiom~{\rm (V0)}.
This concludes the proof of the
case $k<m$.

In the case $k=m$
pick an open semi-flow
invariant subset $A\subset\Lambda^a M$
which contains $\Crit^a$.
Assume by contradiction
that there is no time $T\ge 0$
such that $\varphi_T(\Lambda^a M)\subset A$.
Then there are sequences $s_\nu\to\infty$
and $\gamma^\nu$ in
$(\Lambda^a M)\setminus A $ such that
$\varphi_{s_\nu}\gamma^\nu\notin A$ for $\nu\in\N$.
Now repeat for the much simpler $U:=A$ the
argument given in the case $k<m$.
This proves
Proposition~\ref{prop:entrance-time}.
\end{proof}

\begin{proof}[Proof of Theorem~\ref{thm:cellular-filtration} (Morse
  filtration and chain group isomorphism)]
First we pick an integer $k\in\{0,\ldots,m(a)\}$
where $m(a)$ is the maximal Morse
index~(\ref{eq:max-morse-index}) among
the (finitely many) elements of $\Crit^a$.
Observe that a set
{\bf\boldmath $A$ is semi-flow invariant},
that is $\varphi_T A\subset A$ for every time
$T\ge 0$, if and only if $A\subset {\varphi_T}^{-1}(A)$
for every time $T\ge 0$.
This observation for $A=N_k\cup F_{k-1}$
and the definition of $F_k$,
see~(\ref{eq:F_k}) and~(\ref{eq:F_m}), show that
\begin{equation}\label{eq:inclusion-F_k}
     F_{k-1}\subset\left(N_k\cup F_{k-1}\right)
     \subset {\varphi_{T_{k+1}}}^{-1}
     \left(N_k\cup F_{k-1}\right)
     =:F_k.
\end{equation}
This proves~(i)
in Definition~\ref{def:cellular-filtration}
of a cellular filtration.
Because $F_m=\Lambda^a M$ by~(\ref{eq:F_m}),
condition~(ii) is obviously true.
Thus to prove that $\Ff(\Lambda^a M)=(F_k)$
is a cellular filtration of $\Lambda^a M$ 
it remains to verify condition~(iii)
in Definition~\ref{def:cellular-filtration}.

Putting together the individual
isomorphisms given by~(\ref{eq:nat-iso})
for each critical point $x$
provides the isomomorphism
\begin{equation*}
\begin{split}
     \Theta_k:
     \CM_k^a(\Ss_\Vv)
   &\to 
     \bigoplus_{x\in\Crit^a_k}\Ho_k(N_x,L_x)
   \\
     \langle x\rangle
   &\mapsto
     \Bigl(0,\dots,0,
     \underbrace{
         \left(\iota\circ\vartheta^x\right)_*
         (\sigma_{\langle x\rangle} a_k)
     }_{\text{$=[D^u_{\langle x\rangle}]$
         by~(\ref{eq:vartheta-or})}}
     ,0,\dots,0\Bigr)
\end{split}
\end{equation*}
between abelian groups.
It is well defined since
$\sigma_{\langle x\rangle}\in\{\pm 1\}$
defined by~(\ref{eq:sign-vartheta})
changes sign when replacing the
orientation $\langle x\rangle$ of the
unstable manifold of $x$ by the opposite
orientation $-\langle x\rangle$.

By~(\ref{eq:inclusion-F_k})
and~(\ref{eq:F_k}) there is the inclusion of pairs
$\iota:(N_k,L_k)\hookrightarrow(F_k,F_{k-1})$.
Further below we will prove that it
induces an isomorphism on homology
\begin{equation}\label{eq:iso-F_k}
     \iota_*:\Ho_*(N_k,L_k)
     \stackrel{\cong}{\longrightarrow}
     \Ho_*(F_k,F_{k-1}).
\end{equation}
Recall from~(\ref{eq:N_k})
that $N_k=\cup_x N_x$
is a union of disjoint subsets.
Therefore
$$
     \oplus \iota^x_*:
     \bigoplus_{x\in\Crit^a_k}\Ho_\ell(N_x,L_x)
     \stackrel{\cong}{\longrightarrow}
     \Ho_\ell(N_k,L_k)
$$
is an isomorphism for each $\ell\in\Z$; see
e.g.~\cite[III Proposition~4.12]{dold:1995a}.
Now if $\ell\not= k$, then (each summand of)
the left hand side is zero by
Theorem~\ref{thm:Conley-index}.
Hence $\Ho_\ell(F_k,F_{k-1})=0$
by~(\ref{eq:iso-F_k}), that is condition~(iii)
in Definition~\ref{def:cellular-filtration}
holds true, and $\Ff(\Lambda^a M)=(F_k)$ is a
cellular filtration of $\Lambda ^a M$.
If $\ell=k$, then again by
Theorem~\ref{thm:Conley-index} each
group $\Ho_k(N_x,L_x)$ is generated
by the homology class of the disk
$D^u_x\subset W^u(x)$.
By~(\ref{eq:iso-F_k})
this shows that $\Ff(\Lambda^a M)$
is a Morse filtration.

Next assume $b\le a$ is also a regular value.
It's a first impulse to take as
$\Ff(\Lambda^b M)=(F_k^b)$ the sequence
of intersections $(F_k\cap \Lambda^b M)$.
But then how to calculate $\Ho_\ell
(F_k\cap\Lambda^b M,
F_{k-1}\cap\Lambda^b M)$?
Let's start differently with the simple
observations that $\Crit^b\subset\Crit^a$
and that the sets $N_k$ and $L_k$ defined
by~$($\ref{eq:N_k}$)_a$ contain, respectively,
the sets $N_k^b$ and $L_k^b$ given
by~$($\ref{eq:N_k}$)_b$.
Now define the sets
\begin{equation}\label{eq:F_k-b}
     \Ff(\Lambda^b M)=\left( F_k^b\right)
\end{equation}
iteratively by~$($\ref{eq:F_k}$)_b$
using the sets $N_k^b$ and $F_{k-1}^b$
and taking pre-images with respect
to the semi-flow on $\Lambda^b M$.
However, concerning the new times $T_{k+1}^b$
observe that setting $T_{k+1}^b$ equal to
the \emph{old time} $T_{k+1}=T_{k+1}(a)$
is absolutely fine to satisfy the crucial
condition $F_k^b\supset L_{k+1}^b$.
The proof that $\Ff(\Lambda^b M)$
defined this way is a Morse filtration
is no different from the proof
for $\Ff(\Lambda^a M)$.\footnote{
  Note that the sets $F_k^b$
  are equal to the intersections
  $F_k\cap\Lambda^b M$...
  }

To complete the proof it remains to
establish the isomorphism~(\ref{eq:iso-F_k}).
Similarly as in~(\ref{eq:idea-incl})
the idea is to establish a number of
consecutive isomorphisms
\begin{equation}\label{eq:seq-isos}
\begin{split}
     \Ho_\ell(F_k,F_{k-1})
   &\stackrel{1}{\cong} 
     {\rm H}_\ell(N_k\cup F_{k-1},F_{k-1})
   \\
   &\stackrel{2}{\cong}
     {\rm H}_\ell(N_k,N_k\cap F_{k-1})
   \\
   &\stackrel{3}{\cong} 
     \Ho_\ell(N_k,L_k)
\end{split}
\end{equation}
and show that each generator $[D^u_x]$ is invariant
under the composition of these isomorphisms.
So the image under $\iota_*$
of any basis of $\Ho_*(N_k,L_k)$ consisting
of such elements $[D^u_x]$, one for
each $x\in\Crit^a_k$,
is an isomorphic image of that same basis.
Hence $\iota_*$ takes bases in bases and
therefore it is an isomorphism;
cf.~(\ref{eq:idea-incl}).
\begin{figure}
  \centering
  \includegraphics{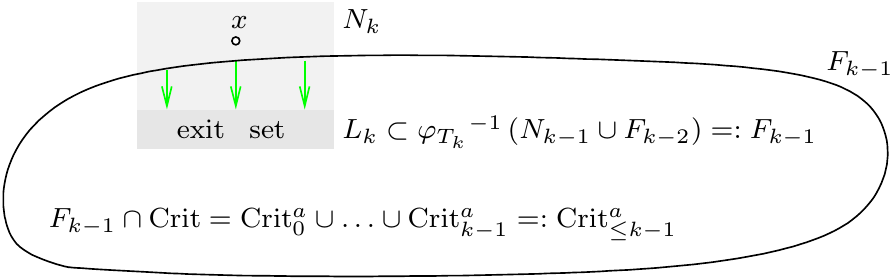}
  \caption{
           The sets $L_k\subset N_k$ and $F_{k-1}$
           }
  \label{fig:fig-set-N_k-and-F_k-1}
\end{figure}

{\bf The first isomorphism} uses the fact
that the open semi-flow invariant sets
$$
     X:=F_k:={\varphi_{T_{k+1}}}^{-1}(N_k\cup F_{k-1})
     ,\qquad
     A:=N_k\cup F_{k-1},
$$
are homotopy equivalent:
Reciprocal homotopy
equivalences are given by
\begin{equation}\label{eq:homot-eq}
     r:X\to A,\;
     \gamma\mapsto\varphi_{T_{k+1}}\gamma
     ,\qquad
     \iota:A\hookrightarrow
     X={\varphi_{T_{k+1}}}^{-1}(A),
\end{equation}
where $\iota$ denotes inclusion. Indeed
$\iota\circ r$ is homotopic to $id_X$ via the
homotopy $\{h_\lambda:X\to X,\gamma
\mapsto\varphi_{\lambda T_{k+1}}
\gamma\}_{\lambda\in[0,1]}$
and $r\circ \iota$ is homotopic to $id_A$ via the
homotopy $\{f_\lambda:A\to A,\gamma
\mapsto\varphi_{\lambda T_{k+1}}
\gamma\}_{\lambda\in[0,1]}$.
Now by homotopy equivalence of the sets $X$
and $A$ their singular homology groups are
isomorphic; see e.g. Corollary~5.3
in~\cite[III]{dold:1995a}.
Hence $\mathrm{H}_*(X,A)=0$ by
the homology sequence
of the pair $(X,A)$, see \emph{loc.~cit.}~(3.2),
and this implies the first isomorphism
(use the homology sequence
of the triple $B\subset A\subset X$ for $B=F_{k-1}$;
\emph{loc.~cit.}~(3.4)).

Alternatively, observe that $\iota$ and $r$ are
reciprocal homotopy equivalences as maps of
pairs $r:(X,B)\to (A,B)$ and $\iota:(A,B)\to (X,B)$
since both homotopies $h_\lambda$ and
$f_\lambda$ preserve the semi-flow invariant set
$B=F_{k-1}$. Thus the induced map on homology
$r_*:\Ho_*(X,B)\to\Ho_*(A,B)$ is an isomorphism
with inverse $\iota_*$; see e.g.
Corollary~5.3 in~\cite[Chapter~III]{dold:1995a}.

Since $r=\varphi_{T_{k+1}}$ leaves the parts
$\INT D^u_x$ of the disks $D^u_x$ outside $L_k$
invariant (as sets) it holds that
$[r(D^u_x)]=[D^u_x]$ as elements of
$\Ho_*(N_k,L_k)$.

{\bf The second isomorphism} uses 
the excision axiom.
Consider the topological space
$X:=N_k\cup F_{k-1}$ and its subset
$A:=F_{k-1}$ which is open in $X$ by
openness of $F_{k-1}$ in $\Lambda^a M$.
For the same reason $N_k$ is open in $X$.
Therefore $N_k\cap F_{k-1}$ is open in $X$.
Observe that
$$
     X=N_k\setminus(N_k\cap F_{k-1})
     \cup (N_k\cap F_{k-1})
     \cup F_{k-1}\setminus(N_k\cap F_{k-1})
$$
is a union of three disjoint sets of which the
second one is open. Thus the complement
of set two is closed and consists of the disjoint
sets one and three. Hence each of them is closed
in $X$. Note that set three is equal to
$B:=F_{k-1}\setminus N_k$.
Since $\cl B=B\subset A=\INT A$ we are in
position to apply the excision axiom
in order to cut off $B$ from $X$ and from $A$
without changing relative homology;
see Figure~\ref{fig:fig-set-N_k-and-F_k-1}.
and e.g. Corollary~7.4 in~\cite[III]{dold:1995a}.

Note that all disks $D^u_x$ are disjoint
from the cut off set $B$. Therefore excision
does not affect any of these disks.

{\bf The third isomorphism}
is based on the fact that there is a
strong deformation retraction
$r:A:=N_k\cap F_{k-1}\to L_k=:B$ as illustrated by
Figure~\ref{fig:fig-set-N_k-and-F_k-1}.
Hence the singular homology 
groups of $A$ and $B$ are isomorphic;
see e.g. Corollary~5.3 in~\cite[III]{dold:1995a}.
Thus $\mathrm{H}_*(A,B)=0$ by
the homology sequence of the pair
$(A,B)$, see \emph{loc.~cit.}~(3.2),
which implies existence of the third isomorphism
$\Ho_*(N_k,A)\cong\Ho_*(N_k,B)$
in~(\ref{eq:seq-isos}) --
to see this use the homology sequence
of the triple $B\subset A\subset N_k$;
see \emph{loc.~cit.}~(3.4).
Because $r$ is defined (below) by flowing
points forward until $L_k$ is reached,
the disks $D^u_x\subset W^u(x)$
are invariant (as sets) under $r$ and therefore
$[r(D^u_x)]=[D^u_x]$ as elements of
$\Ho_*(N_k,L_k)$.

To construct the strong deformation
retraction $r:A\to B$ consider the
{\bf entrance time function}
\begin{equation}\label{eq:entrance-time}
\begin{split}
     \Tt=\Tt_{L_k}:N_k\cap F_{k-1}
    &\to[0,\infty)
     \\
     \gamma
    &\mapsto
     \inf\{s\ge 0\mid\varphi_s\gamma\in L_k\}
\end{split}
\end{equation}
associated to the subset $L_k$ of $N_k\cap F_{k-1}$.
We use the convention $\inf\emptyset=\infty$.
Concerning the target $[0,\infty)$
as opposed to $[0,\infty]$ observe that
the semi-flow moves any element
$\gamma\in N_k\cap F_{k-1}$ into $L_k$ in
some finite time:
By~\cite[Thm.~9.14]{weber:2010a-link}
which uses that $\Ss_\Vv$ is Morse below
level $a$, the asymptotic forward limit
$$
     \gamma_\infty
     :=\lim_{s\to\infty}\varphi_s\gamma
     \in \Crit^a\cap F_{k-1}=\Crit^a_{\le k-1}
$$
exists and is some critical point below level $a$.
Concerning the right hand side we used that
$F_{k-1}$ is semi-flow invariant and contains
precisely the critical points (below level $a$)
of Morse index less or equal to $k-1$.
Hence $\gamma_\infty\notin N_k$, because
the critical points inside $N_k$ are of Morse index $k$.
This shows that the trajectory with initial point
$\gamma$ leaves $N_k$. But doing so it has to run
through the exit set $L_k$ of $N_k$ by
Definition~\ref{def:index-pair}~(iv). Thus the
entrance time $\Tt(\gamma)$ in $L_k$ is finite.

Note that the infimum in~(\ref{eq:entrance-time})
is actually taken on by (relative)
closedness of $L_k$.
Below we prove that $\Tt$ is continuous.
Consequently the map defined by
\begin{equation*}
\begin{split}
     r:A=N_k\cap F_{k-1}
   &\to L_k=B
   \\
     \gamma
   &\mapsto\varphi_{\Tt(\gamma)}\gamma
\end{split}
\end{equation*}
takes values in $B$ and is continuous.
But $r\circ \iota=id_B$ and
$\iota\circ r=h_1$ is homotopic to
$id_A=h_0$ via the homotopy $\{ h_\lambda:A\to A,\,
\gamma\mapsto\varphi_{\lambda \Tt(\gamma)}
\gamma\}_{\lambda\in[0,1]}$. Thus
$r$ is a strong deformation retraction
and it only remains to check continuity of
$\Tt$.\footnote{
  In such situations the Kat\v{e}tov-Tong
  insertion Theorem~\cite{katetov:1951a,tong:1952a}
  can be very useful: Given functions
  $u\le\ell:X\to\R$ on a normal
  topological space with $u$ upper
  and $\ell$ lower semi-continuous. 
  Then there exists a continuous function
  $f:X\to\R$ in between, that is
  $u\le f\le \ell$.
  }

{\boldmath\bf The entrance time
function $\Tt$ is continuous}:
Lemma~2.10 in~\cite{abbondandolo:2006a}
tells that the entrance time function
associated to a \emph{closed/open} subset
is \emph{lower/upper} semi-continuous.
Thus $\Tt=\Tt_{L_k}$ is lower
semi-continuous by closedness of $L_k$ in
$N_k\cap F_{k-1}$.
So it remains to prove upper semi-continuity.
Although $L_k$ is not open, it behaves like an open
set under the \emph{forward} semi-flow.
Namely, any element of $L_k$
remains inside $L_k$ for sufficiently small times
by openness of $N_k$ and because $L_k$ is
\emph{positively invariant} in $N_k$.
More precisely, choose
$\gamma_0\in N_k\cap F_{k-1}$ and $\delta>0$.
Recall from~(\ref{eq:N_k})
that $\gamma_0\in N_x\cap F_{k-1}$ for some
path connected component 
$N_x=N_x^{\eps,\tau}$ of $N_k$.
As we saw above $\Tt(\gamma_0)$ is finite
and $\varphi_{\Tt(\gamma_0)}\gamma_0$ lies
in the boundary of $L_x$ relative $N_x$, that is
$$
     \varphi_{\Tt(\gamma_0)}\gamma_0\in
     \p L_x=
     \left(\left(\varphi_{2\tau}\right)^{-1}
     \{\Ss_\Vv=c-\eps\}\right)
     \cap\{\Ss_\Vv<c+\eps\}
      ,\quad c:=\Ss_\Vv(x),
$$
although not yet in its interior
$$
     \INT L_x=
     \left(\left(\varphi_{(\tau,2\tau)}\right)^{-1}
     \{\Ss_\Vv=c-\eps\}\right)
     \cap\{\Ss_\Vv<c+\eps\}.
$$
By continuity of $\varphi$ there is a time
$T\in\left(\Tt(\gamma_0),\Tt(\gamma_0)+\delta\right)$
such that (the possibly small) forward flow segment
$\varphi_{[0,T]}\gamma_0$ is still contained in the
open subset $N_x\subset\Lambda^a M$.\footnote{
  Necessarily $T<\Tt(\gamma_0)+\tau$ since
  already $\varphi_{\Tt(\gamma_0)+\tau}\gamma_0
  =\varphi_\tau(\varphi_{\Tt(\gamma_0)}\gamma_0)$
  lies outside $N_x$.
  }
Thus $\varphi_T\gamma_0\in L_x$
by positive invariance of $L_x$ in $N_x$, see
Definition~\ref{def:index-pair}~(iii),
and $\varphi_T\gamma_0\in\INT L_x$
since $\Tt(\gamma_0)<T<\Tt(\gamma_0)+\tau$.
Thus by continuity of
$\varphi$ in the loop variable $\gamma$
there is a neighborhood $U$ of $\gamma_0$
in the open subset
$N_k\cap F_{k-1}\subset\Lambda^a M$
such that its image $\varphi_T(U)$ is contained
in the open neighborhood $\INT L_x$ of
$\varphi_T\gamma_0$ in $\Lambda^a M$.
Thus, given any $\gamma\in U$,
time $T$ lies in the set whose
infimum~(\ref{eq:entrance-time})
is $\Tt(\gamma)$ and therefore
\begin{equation}\label{eq:analog-args}
     \Tt(\gamma)
     \le T
     <\Tt(\gamma_0)+\delta.
\end{equation}
This shows that $\Tt$ is upper semi-continuous
at any $\gamma_0\in N_k\cap F_{k-1}$ and
concludes the proof that $\Tt$ is continuous.
The proof of
Theorem~\ref{thm:cellular-filtration}
is complete.
\end{proof}

\subsection{Cellular and singular homology}

\begin{theorem}\label{thm:cellular=singular}
Assume $\Ss_\Vv$ is Morse-Smale
below regular values $b\le a$ and
consider the Morse filtrations $\Ff(\Lambda^b M)
\hookrightarrow\Ff(\Lambda^a M)$ provided
by Theorem~\ref{thm:cellular-filtration}.
Then there are natural isomorphisms
\begin{equation}
\label{eq:isomorphism-cellular-singular}
     \Ho_*\Ff\left(\Lambda^b M\right)
     \cong \Ho_*\left(\Lambda^b M\right)
     ,\qquad
     \Ho_*\Ff\left(\Lambda^a M\right)
     \cong \Ho_*\left(\Lambda^a M\right)
\end{equation}
which commute with the inclusion induced
homomorphisms 
$\Ho_*\Ff\left(\Lambda^b M\right)
\to\Ho_*\Ff\left(\Lambda^a M\right)$ and
$\Ho_*\left(\Lambda^b M\right)\to
\Ho_*\left(\Lambda^a M\right)$.
\end{theorem}

\begin{proof}
Apply~\cite[V Prop,~1.3]{dold:1995a}
to the cellular map provided by inclusion.
\end{proof}

\begin{remark}
Obviously for $k$ negative or larger than
the maximal Morse index $m(a)$ on
$\Lambda^a M$ there are no critical points
of Morse index $k$. Thus there are no
generators of $\Co_k\Ff(\Lambda^a M)$
by Theorem~\ref{thm:cellular-filtration} and
therefore
${\rm H}_k(\Lambda^a M)$ is trivial for such $k$
by~(\ref{eq:isomorphism-cellular-singular}).
\end{remark}

\subsection{Cellular and Morse chain complexes}
In Theorem~\ref{thm:cellular-filtration} we
established isomorphisms
$$
     \Theta_k=\Theta_k(\vartheta):
     \CM^a_k\left(\Ss_\Vv\right)
     \to \Co_k\Ff:=\Ho_k\left(F_k,F_{k-1}\right),
     \quad k\in\{0,\dots,m(a)\},
$$
between the Morse complex associated to
the Morse function $\Ss_\Vv$ on $\Lambda^a M$
and the cellular complex associated to
the Morse filtration
$\Ff =\left( F_k\right)_{k=-1}^m$ of
$\Lambda^a M$ defined by~(\ref{eq:F_k}).
On the other hand,
by~(\ref{eq:isomorphism-cellular-singular}) 
there is a natural isomorphism
between cellular homology and 
singular homology of $\Lambda^a M$.
So in order to establish the isomorphism
in Theorem~\ref{thm:main}
between Morse homology and singular homology
it suffices to prove that
the isomorphisms $\Theta_k$ intertwine
the Morse and the triple boundary
operators.\footnote{
  In this case both \emph{chain complexes} --
  the Morse complex of $\Ss_\Vv$
  and the cellular complex of the Morse
  filtration $\Ff$ -- are \emph{equal} (under the
   identifications provided by $\Theta_k$).
  }
Remarkably, in this very last step also
the \emph{forward} $\lambda$-Lemma
enters.

\begin{proof}[Proof of Theorem~\ref{thm:Morse=triple}]
For $k=0$ both boundary operators are trivial.
Fix $k\in\{1,\dots,m(a)\}$.
Given the key Theorem~\ref{thm:cellular-filtration},
the proof of~\cite[Theorem~2.11]{abbondandolo:2006a}
essentially carries over modulo the little new
twists caused by the present use of
push-forward orientations
and the forward $\lambda$-Lemma.
For convenience of the reader
we recall the proof and add further
details. 

\vspace{.1cm}\noindent
{\bf Idea of proof 
(cf. Figure~\ref{fig:fig-flow-down-sphere}).}
  In the unstable manifold $W^u(x)$
  one picks a certain disk $D^u_x$
  about $x$ with bounding sphere
  $S^u_x=\alpha^x(\SS^{k-1})$ 
  in the exit set $L_x\subset F_{k-1}$.
  For large times $T$ the forward flow
  $\varphi_T S^u_x=\beta^x(\SS^{k-1})$ largely
  enters $F_{k-2}$ -- except for center parts
  of embedded balls
  $B_1^T,\dots,B_N^T$
  which get stuck near critical points $y$ of Morse
  index $k-1$. The center of each ball
  corresponds to a connecting trajectory
  $u^\ell$ from $x$ to some $y$. In this case the
  center is $u^\ell(T)$ and $y=u^\ell(\infty)$.
  Homologically the splitting of the $(k-1)$-sphere
  provided by isolated flow lines emanating
  from $x$ is encoded by
  identity~(\ref{eq:Abbo-Maj-Exc}).
  A relevant part of each thickened flow line
  $B_\ell^T$ is isotopic to the disk
  $D^u_y=\vartheta^y(\D^{k-1})$
  thereby transporting a given orientation
  $\langle x\rangle$ of $W^u(x)$
  down to an orientation of $W^u(y)$
  denoted by $u_*\langle x\rangle$.

\vspace{.1cm}
Fix an oriented critical point $\langle x\rangle$
of Morse index $k$ and below level $a$
and consider the commutative diagram
in which all maps whose notation involves
$\iota$ or $i$ are inclusion induced.
\begin{equation*}
\xymatrix{
   &
     \Ho_*(N_x,L_x)
     \ar[r]^{\iota^x_*}
   &
     \Ho_*(N_k,L_k)
     \ar[d]^{\iota_*}_\cong
   \\
     {\underbracket{\Ho_k(\D^k,\SS^{k-1})}_{\sigma_{\langle x\rangle}[\D^k_{\langle\mathrm{can}\rangle}]}}
     \ar[r]^{\vartheta^{x}_*}_\cong
     \ar[d]_{\p}^\cong
   &
     \Ho_k(D^u_x,S^u_x)
     \ar[u]^{\iota_*}_\cong
     \ar[d]_{\p}^\cong
     \ar[r]^{i_*^x}
   &
     {\underbracket{\Ho_k(F_k,F_{k-1})}_{{\color{blue}{\Theta_k\langle x\rangle=}}[D^u_{\langle x\rangle}]}}
     \ar[d]^{\p}
   \\
          {\underbracket{\Ho_{k-1}(\SS^{k-1})}_{\sigma_{\langle x\rangle} [\SS^{k-1}_{\langle\mathrm{can}\rangle}]}}
     \ar[r]^{\;\;\;\alpha^x_* =(\vartheta^x|)_*}_{\cong}
     \ar[dd]_{J_*}
     \ar[dr]_{\beta^x_*=(\varphi_T\alpha^x)_*}
   &
     \Ho_{k-1}(S^u_x)
     \ar[r]^{(i^x|)_*}
     \ar[d]_{(\varphi_T)_*}^\cong
   &
     {\underbracket{\Ho_{k-1}(F_{k-1})}
     _{[S^u_{\langle x\rangle}]=[\varphi_T S^u_{\langle x\rangle}]}}
     \ar[dd]^{j_*}
   \\
   &
     \Ho_{k-1}(\varphi_T S^u_x)
     \ar[ru]_{\iota_*}
     \ar[d]_{j_*}
   \\
     {\underbracket{\Ho_{k-1}(\SS^{k-1},\SS^*)}_{{\stackrel{(\ref{eq:Abbo-Maj-Exc})}{=}}\sigma_{\langle x\rangle}\sum_\ell [B_\ell]}}
     \ar[r]^{\beta^x_*}_\cong
   &
     \Ho_{k-1}(\varphi_T S^u_x,\varphi_T S_x^*)
     \ar[r]^{\iota_*}
   &
     {\underbracket{
        \overbracket{\Ho_{k-1}(F_{k-1},F_{k-2})}
             ^{\sigma_{\langle x\rangle}\sum_\ell[\varphi_T\alpha^x(B_\ell)]}}
       _{=\sum_\ell[D^u_{u^\ell_*\langle
           x\rangle}]\color{blue}{=\sum_\ell\Theta(u^\ell_*\langle x\rangle)}}}
   \\
     \bigoplus_\ell
     {\underbracket{\Ho_{k-1}(B_\ell,\p B_\ell)}
             _{ \sigma_{\langle x\rangle} [B_\ell]}}
     \ar[u]^{\oplus_\ell \, \iota^\ell_*}_\cong
   &
     \bigoplus_\ell
     {\underbracket{\Ho_{k-1}(\D^{k-1},\SS^{k-2})}
     _{\sigma_{\langle x\rangle} a_{k-1}=\sigma_{\langle x\rangle} [\D^{k-1}_{\langle\mathrm{can}\rangle}]}}
     \ar[l]_{\diag(\theta^\ell_* )\;\;}^\cong
     \ar[ur]_{\;\;\oplus_\ell\, \bar\vartheta^{y(\ell)}_*}
   &
     , \underbracket{y=y(\ell):=u^\ell(+\infty)}
     _{u^\ell(0)=\alpha^x(\xi_\ell),\, \xi_\ell\in B_\ell\subset\SS^{k-1}} 
}
\end{equation*}
The elements of the homology groups shown
above/below the horizontal brackets
are mapped to one another by the maps
labelling the arrows. The diffeomorphism
$\vartheta^x:\D^k\to D^u_x:=\varphi_{-2\tau}
\overline{W^u_\eps(x)}\subset N_x$,
see~(\ref{eq:nat-iso}) and
Figures~\ref{fig:fig-D_u_x} 
and~\ref{fig:fig-flow-down-sphere},
is the one corresponding to $x$
in the sequence $\vartheta$ chosen to
define $\Theta_k$
and $\alpha^x=\vartheta^x|$ denotes
restriction to the boundary $\SS^{k-1}$.
The maps $j$ and $J$ are the usual projection
maps in their respective short exact sequence
of pairs.
The rectangle in row one commutes, simply
because all maps are inclusions.
The two squares in row two commute by
naturality of long exact sequences of pairs and
so do the two (nonrectangular) squares in row
three. The left triangle commutes by definition
of $\beta^x$ in~(\ref{eq:beta-x})
and the right one as the embedded
$(k-1)$-spheres
\begin{equation}\label{eq:S^u_x-T_k}
     S^u_x:=\alpha^x(\SS^{k-1})
     \subset L_x\subset L_k\subset F_{k-1}
     :={\varphi_{T_k}}^{-1}(N_{k-1}\cup F_{k-2})
\end{equation}
and $\varphi_T S^u_x$ of $W^u(x)$ are not only
diffeomorphic but even isotopic inside
the (semi-flow invariant) set $F_{k-1}$.
Commutativity of the final row
uses an isotopy provided by the
\emph{forward} $\lambda$-Lemma;
see~(\ref{eq:gamma}).
\\
For now ignore the last two lines of the
diagram. However, for later use let us mention
right away that we abbreviated relevant ball
complements by
$$
     \SS^*
     :=\SS^{k-1}\setminus{\cup_\ell}\,\INT B_\ell
     ,\qquad
     S_x^*
     :=S^u_x\setminus {\cup_\ell}\,
     \alpha^x(\INT B_\ell).
$$
These punched spheres are given by
the complement of open balls $\INT B_\ell$
in $\SS^{k-1}$ and the complement of the
corresponding open balls $\alpha^x(\INT B_\ell)$
in the corresponding sphere
$\alpha^x(\SS^{k-1})=S^u_x$, respectively.
\\
Recall the canonical orientations
of $\D^k$ and $\SS^{k-1}$ and the {\bf positive
generators} $a_k=[\D^k_{\langle\mathrm{can}\rangle}]$
and $b_{k-1}=[\SS^{k-1}_{\langle\mathrm{can}\rangle}]$
of $\Ho_k(\D^k,\SS^{k-1})$ and
$\Ho_{k-1}(\SS^{k-1})$, respectively, introduced
in Definition~\ref{def:canonical-orientation}. With
these conventions the connecting homomorphism
$\p:\Ho_k(\D^k,\SS^{k-1})\to \Ho_{k-1}(\SS^{k-1})$
maps $a_k$ to
$[\p\D^k_{\langle\mathrm{can}\rangle}] =b_{k-1}$.

\begin{figure}
  \centering
  \includegraphics{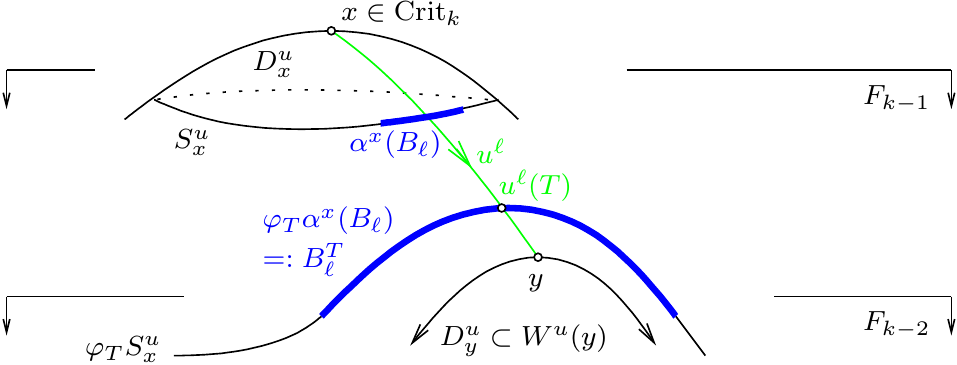}
  \caption{Isolated flow lines $u^\ell$
                and embedded balls $\alpha^x(B_\ell)$
           }
  \label{fig:fig-flow-down-sphere}
\end{figure}

{\bf The task at hand} is to express
the action of the triple boundary operator
on a generator 
$$
     \Theta_k\langle x \rangle 
     :=\bar\vartheta^x_*(\sigma_{\langle x\rangle} a_k)
     =[D^u_{\langle x\rangle}]
     \in\Ho_k(F_k,F_{k-1})=\Co_k\Ff
$$
of $\Co_k\Ff$ in terms of generators
$
     [D^u_{y}]\in\Co_{k-1}\Ff
$
where the $D^u_y\subset W^u(y)$ are
appropriately oriented disks -- one for
each flow trajectory connecting $x$ to
some $y\in\Crit^a_{k-1}$.
Recall that $\alpha^x=\vartheta^x|:
\SS^{k-1}\to S^u_x$
is a diffeomorphism. Abbreviate
$$
     \bar\alpha^x:=i^x\circ\alpha^x :
     \SS^{k-1}\to S^u_x\hookrightarrow F_{k-1} 
$$
and
\begin{equation}\label{eq:beta-x}
     \beta^x:=\varphi_T\circ\alpha^x ,
     \qquad
     \bar\beta^x:=\iota\circ\beta^x :
     \SS^{k-1}\to\varphi_T S^u_x
     \hookrightarrow F_{k-1} ,
\end{equation}
where $T\ge 1$ will be defined
in~(\ref{eq:time-AAAA}) below.
Use the definition~(\ref{eq:Theta-k})
of $\Theta_k$, the
identity~(\ref{eq:triple-boundary}) for
$\p_k^{trip}$, and commutativity of the huge
diagram above to obtain the following identities
\begin{equation}\label{eq:triple-boundary-2}
\begin{split}
     \left(\p_k^{trip}\Theta_k\right)\langle x\rangle
   &=
     \left(j_*\p i_*^x\vartheta^{x}_*\right)
     (\sigma_{\langle x\rangle} a_k)
   \\
   &=\sigma_{\langle x\rangle}
     \left( j_*\bar\alpha^x_*\right)(b_{k-1})
   \\
   &=\sigma_{\langle x\rangle}
     \left( j_*\bar\beta^x_*\right)(b_{k-1})
   \\
   &=\sigma_{\langle x\rangle}
     \left(\bar\beta^x_* J_*\right)(b_{k-1})
   \\
   &=\sum_{y\in\Crit^a_{k-1}}\sum_{u\in m_{xy}}
     \underbrace{\left( i_*^y\vartheta^y_*\right)
        \left(\sigma_{u_*\langle x\rangle} a_{k-1}\right)}
     _{\Theta_{k-1}(u_*\langle x\rangle)}
\end{split}
\end{equation}
among which only the final one remains to be
proved. To start with observe that by the
Morse-Smale condition together with index
difference one the pre-image
$$
     \{\xi_1,\dots,\xi_N\}
     :=(\alpha^x)^{-1}
     \left(\bigcup_{y\in\Crit_{k-1}^a} W^s(y)\right)
     \cong\bigcup_{y\in\Crit_{k-1}^a} m_{xy}
$$
is a finite subset of $\SS^{k-1}$
which parametrizes\footnote{
  Note that $\alpha^x(\SS^{k-1})\cap W^s(y)
  =S^u_x\cap W^s(y)
  \cong \varphi_{2\tau} S^u_x \cap W^s(y)
  =S^u_\eps(x)\cap W^s(x)
  \cong m_{xy}$
  where $S^u_\eps(x)$ is contained in a level set;
  both diffeomorphisms arise
  by restricting the heat flow to unstable manifolds;
  cf. Remark~\ref{rmk:diffeo-unstable-NEW}.
  }
the set of (unparametrized) heat flow lines
running from $x$ to some critical point of
Morse index $k-1$; cf.~(\ref{eq:m_xy})
and~\cite[Prop.~1]{weber:2013b}.
We denote by
{\boldmath $u^\ell$} the (unique) {\bf heat flow
trajectory} which passes at time $s=0$ through
the point $\alpha^x(\xi_\ell)\in W^u(x)\cap W^s(y)$
where $y=y(\ell):=u^\ell(\infty)$ is
the corresponding critical point of Morse index
$k-1$; see Figure~\ref{fig:fig-flow-down-sphere}.
Pick a time $s_\ell>0$ such that the point
$u^\ell(s_\ell)=\varphi_{s_\ell}\alpha^x(\xi_\ell)$
already lies in the 
ball $B_y^{\rho/2}$ about $y$ where the
radius $\rho>0$ only depends on the
action value $a$; see Lemma~\ref{le:34}
(Morse-Smale on neighborhoods).

By asymptotic forward
existence~\cite[Thm.~9.14]{weber:2010a-link}
and strictly decreasing Morse index
along connecting orbits
due to the Morse-Smale condition,
Lemma~\ref{le:34},
all elements of the punctured sphere
$
     \SS^{k-1}\setminus\{\xi_1,\dots,\xi_N\}
$
are mapped under $\alpha^x$ to points of
$W^u(x)$ which asymptotically converge in
forward time to some critical point $z$
below level $a$ and of
Morse index strictly smaller than $k-1$.
But such critical points are contained in
$F_{k-2}$; see Definition~\ref{def:Morse-filtration}.
Fix $N$ {\bf\boldmath pairwise
disjoint closed balls}
$\iota^\ell: B_\ell\hookrightarrow\SS^{k-1}$
centered in $\xi_\ell\in\SS^{k-1}$
and sufficiently small such that
\begin{equation}\label{eq:MS-B_ell}
     \varphi_{s_\ell}\alpha^x(B_\ell)
     \subset B^\rho_{y(\ell)}
     \qquad, \ell=1,\dots,N
     =\sum_{y\in\Crit^a_{k-1}}\abs{m_{xy}}.
\end{equation}
The canonical orientation of 
$\SS^{k-1}$ induces a {\bf canonical orientation}
of $B_\ell$.\footnote{
  For $k=1$ the sphere $\SS^0$
  consists precisely of the $N=2$ points
  $\xi_1=-1$ and $\xi_2=+1$, whose complement
  is empty. The two 0-balls
  are given by $B_\ell=\{\xi_\ell\}$ and
  $F_{k-1}=F_{-1}=\emptyset$.
  }
Throughout we denote by {\boldmath $B_\ell$}
the ball equipped with its canonical orientation.

Associated to the closed subset
$\overline{F_{k-2}}\subset\Lambda^a M$,
see~(\ref{eq:entrance-time}), there is the
continuous\footnote{
  Lower semi-continuity holds by closedness
  of the subset and upper semi-continuity
  follows from the fact that $F_{k-1}$ is
  positively invariant by the arguments
  which led to (\ref{eq:analog-args}).
  }
entrance time function
$\Tt_{\overline{F_{k-2}}}:\Lambda^a M\to[0,\infty]$.
The function
\begin{equation}\label{eq:max-T}
\begin{split}
     f:\SS^*=
     \SS^{k-1}\setminus\cup_{\ell}\,\INT B_\ell
   &\to[0,\infty) ,
   \\
     \xi
   &\mapsto
     \Tt_{\overline{F_{k-2}}}\left(\alpha^x(\xi)\right) 
\end{split}
\end{equation}
is continuous and also pointwise
  finite.\footnote{
  As observed earlier for each $\xi\in\SS^*$
  the point $\alpha^x(\xi)$ lies on a trajectory
  which connects $x$ with some
  $z\in\Crit^a_{\le k-2}\subset F_{k-2}$.
  Thus $\alpha^x(\xi)$ reaches the open
  set $F_{k-2}$ in finite time.
  }
Hence by compactness of its domain, that is the
punched sphere $\SS^*$, the function
$f$ admits a maximum. (Note that
$F_{k-2}=F_{-1}=\emptyset$ in the case $k=1$.)
Consider the instants of time
\begin{equation}\label{eq:time-AAAA}
     T:=\max\left\{T_k,s_x,1+\max f\right\}
     ,\qquad
     s_x:=\max\{s_1,\dots,s_N\} ,
\end{equation}
which come with the following consequences.
Firstly, by~(\ref{eq:S^u_x-T_k})
there is the crucial inclusion $\varphi_T S^u_x\subset
N_{k-1}\cup F_{k-2}$. This inclusion, together
with~(\ref{eq:rho}), (\ref{eq:34}),
(\ref{eq:MS-B_ell}), and the facts that
$N_{k-1}=\cup_z N_z$ and
$N_z\subset B_z^\rho$, implies that
\begin{equation}\label{eq:time-TTT}
     u^\ell(T) \in N_{y(\ell)}
      ,\qquad
     B^T_\ell:=\varphi_T\alpha^x(B_\ell)\subset
     N_{y(\ell)}\cup F_{k-2}.
\end{equation}
Secondly, the image $\varphi_T(S^u_x)$
of the map $\bar\beta^x$
largely lies downtown in $F_{k-2}$ except for
(small neighborhoods of) the points $u^\ell(T)$
each of which gets stuck at a critical point
$y=y(\ell):=u^\ell(+\infty)\in\Crit^a_{k-1}$;
see Figure~\ref{fig:fig-flow-down-sphere}.
Via the isotopy
$\{\varphi_{\lambda T}\circ\bar\alpha^x\}
_{\lambda\in[0,1]}$
the map $\bar\alpha^x$ is homotopic
to $\bar\beta^x$ in $F_{k-1}$.
Thus $[S^u_x]=\bar\alpha^x_*([\SS^{k-1}])
=\bar\beta^x_*([\SS^{k-1}])=[\varphi_T S^u_x]$
as elements of $\Ho_{k-1}(F_{k-1})$ by the
homotopy axiom of singular homology.
Most importantly, the map
$\bar\beta^x$ is well defined as a map
between the pairs of spaces indicated
in the following diagram. 

Fix for every $\ell$ an {\bf\boldmath
orientation preserving diffeomorphism}
$\theta^\ell:\D^{k-1}_{\langle\mathrm{can}\rangle}\to B_\ell$
and consider the commutative diagram of
maps of pairs
\begin{equation}\label{eq:cruc-id}
\begin{gathered}
\xymatrix{
     (\D^{k-1},\SS^{k-2})
     \ar[d]_{\theta^\ell}
    &
     \SS^{k-1}
     \ar[d]_J
     \ar[r]^{\bar\beta^x=\iota\varphi_T\alpha^x}
    &
     F_{k-1}
     \ar[d]^j
 \\
     \left(B_\ell,\p B_\ell\right)
     \ar[r]^{\iota^\ell\qquad\quad}
    &
     \left(
     \SS^{k-1},\SS^{k-1}\setminus\cup_\ell
     \,\INT B_\ell
     \right)
     \ar[r]^{\quad\bar\beta^x}
    &
     (F_{k-1},F_{k-2}).
}
\end{gathered}
\end{equation}
Here $J$ and $j$ denote inclusions
of pairs $X=(X,\emptyset)\mapsto(X,A)$.
The identity
\begin{equation}\label{eq:Abbo-Maj-Exc}
     J_*(b_{k-1}) 
     =\sum_{\ell=1}^N\bar\theta^\ell_*(a_{k-1})
     ,\qquad
     \bar\theta^\ell:=\iota^\ell\theta^\ell,
\end{equation}
provided
by~\cite[Exc.~2.12]{abbondandolo:2006a}
proves the first of the two identities
\begin{equation}\label{eq:gamma}
\begin{split}
     \sigma_{\langle x\rangle}\cdot
     \bigl(\bar\beta^x_* J_*\bigr) (b_{k-1})
   &=\sigma_{\langle x\rangle}\cdot\sum_{\ell=1}^N
      \left(\bar\beta^x\bar\theta^\ell\right)_*(a_{k-1})
   \\
   &=\sum_{\ell=1}^N
     \sigma_{u^\ell_*\langle x\rangle}\cdot
     \bar\vartheta^y_*(a_{k-1}).
\end{split}
\end{equation}
To conclude the proof
of~(\ref{eq:gamma}), thus
of~(\ref{eq:triple-boundary-2}), 
hence of Theorem~\ref{thm:Morse=triple},
it remains to prove that the maps
\begin{equation}\label{eq:id-homol}
     \sigma_{\langle x\rangle}\cdot
     (\bar\beta^x\bar\theta^\ell)_*
     \;\text{and}\;
     \sigma_{u^\ell_*\langle x\rangle}\cdot
     \bar\vartheta^y_*
     :\Ho_{k-1}(\D^{k-1},\SS^{k-2})
     \to\Ho_{k-1}(F_{k-1},F_{k-2})
\end{equation}
coincide on the
\emph{positive} generator $a_{k-1}$.
By definition~(\ref{eq:sign-kappa})
of the orientation reversing diffeomorphism
$\mu=\diag(-1,1,\dots,1)\in\Ll(\R^{k-1})$
and $\kappa_{\langle x\rangle}\in\{0,1\}$
this holds true if the by $\mu$
pre-composed maps of pairs\footnote{
  Changing the sign of the image of a homology
  class corresponds to pre-composing the map
  with an orientation preserving diffeomorphism
  such as $\mu$. Certainly $\mu=\mu^1$ and
  $\mu^0:=\1$.
  }
(illustrated by Figure~\ref{fig:fig-orientations})
$$
     \bar\beta^x\bar\theta^\ell
     \mu^{\kappa_{\langle x\rangle}}
     \;\;\text{and}\;\;
     \bar\vartheta^y
     \mu^{\kappa_\ell} 
     \;:\;(\D^{k-1},\SS^{k-2})\to(F_{k-1},F_{k-2})
     ,\qquad
     \kappa_\ell:=\kappa_{u^\ell_*\langle x\rangle},
$$
are isotopic, thus homotopic among
orientation preserving maps.\footnote{
  It suffices to show that the first map
  takes the canonically oriented disk
  $\D^{k-1}$
  to a disk isotopic to $D^u_{y(\ell)}$
  endowed with the transported orientation
  $u^\ell_*\langle x\rangle$
  as the latter is
  $\bar\vartheta^y\mu^{\kappa_\ell}
  (\D^{k-1}_{\langle\mathrm{can}\rangle})$.
  }
The proof takes two steps. First we
isotop (a relevant part of)
the map $\bar\beta^x\bar\theta^\ell$ to
$\bar\vartheta^y$, then in step two we verify
that all chosen orientations are preserved.
\begin{figure}
  \centering
  \includegraphics{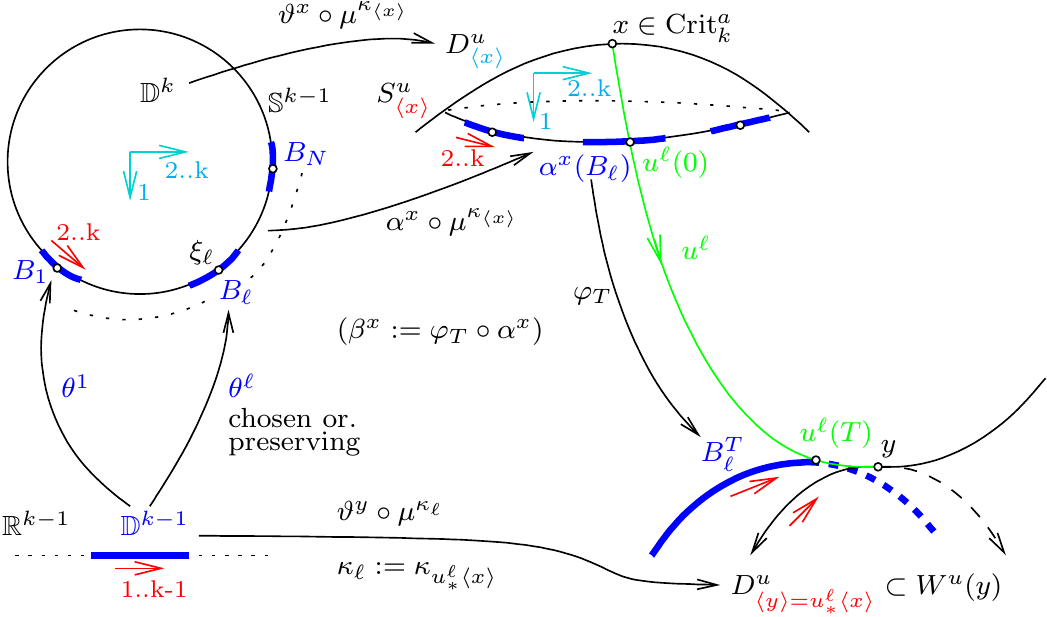}
  \caption{All maps are orientation preserving
                by choice of the exponents $\kappa$
           }
  \label{fig:fig-orientations}
\end{figure}

\vspace{.1cm}
\noindent
{\bf Step~1 (Isotopy).}
We construct an isotopy
of maps of pairs
$$
     (\D^{k-1},\SS^{k-2})
     \to(N_y\cup F_{k-2},F_{k-2})
     \subset(F_{k-1},F_{k-2})
$$
between (relevant parts of) the embedded
disks 
$$
     \bar\beta^x\bar\theta^\ell(\D^{k-1})
     =\varphi_T\alpha^x(B_\ell)
     =:B^T_\ell
     \quad\text{and}\quad
     \bar\vartheta^y(\D^{k-1})=D^u_y 
     \quad\text{where $y=u^\ell(\infty)$.}
$$
Remarkably at this very late stage of
the whole project eventually the
\emph{forward} analogue of the Backward
$\lambda$-Lemma~\cite[Thm.~1]{weber:2014a}
enters as a crucial
tool.\footnote{Since all dynamics takes
  place locally near $y$ in the closure of the
  unstable manifold of $x$ even
  the standard finite dimensional
  $\lambda$-Lemma,
  see e.g.~\cite[Ch.~2 \S7]{palis:1982a},
  serves our purposes.
  }
This is a local result valid in a neighborhood
of a hyperbolic fixed point.\footnote{
  Alternatively, apply the
  hyperbolic tools used in~\cite[Proof of
  Theorem~2.11]{abbondandolo:2006a}.
  }
We assume without loss of generality
that the forward $\lambda$-Lemma
applies on the whole domain
of our usual local coordinates
$\Phi^{-1}$ near any of the finitely many
critical points on
$\Lambda^a M$.\footnote{
  Otherwise, start with a smaller
  radius $\rho_0$ in
  Hypothesis~\ref{hyp:rho-eps-tau}.
  }
From now on we fix a local
parametrization
$\Phi:T_y\Lambda M=X=X^-\oplus X^+
\supset B^u\times B^+\to\Lambda M$ 
near $y=y(\ell)$ and use our
usual conventions concerning local
notations; see 
Hypothesis~\ref{hyp:local-setup-N}
and Figure~\ref{fig:fig-forward-lambda}.
In particular, the local flow is
denoted by $\phi$ and $S^u_\eps$
abbreviates the descending sphere $S^u_\eps(y)$.
However, we will not change notations
such as $N_y$, $L_{k-1}$, $F_{k-2}$ etc.
\begin{figure}
  \centering
  \includegraphics{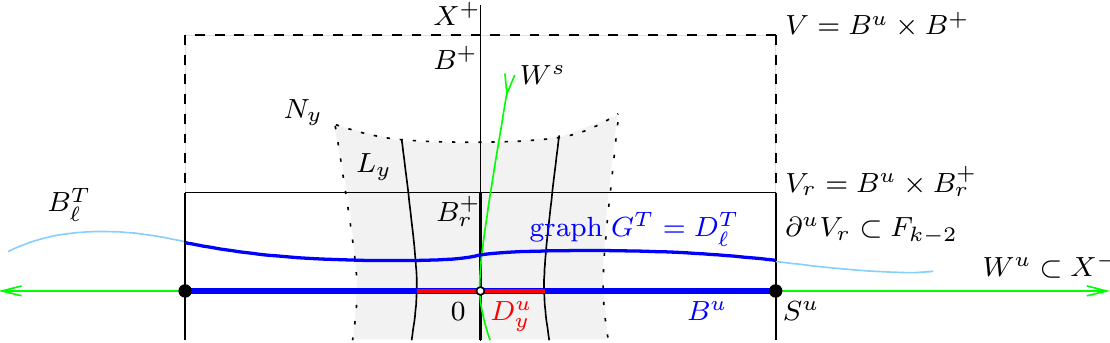}
  \caption{Isotopy $\{\graph\:
                \lambda G^T\}_{\lambda\in[0,1]}$ between
                $D^T_\ell$ and $B^u\supset D^u_y$
           }
  \label{fig:fig-forward-lambda}
\end{figure}
Observe that
$\phi_{-2\tau} S^u_\eps\subset L_y
\subset L_{k-1}\subset F_{k-2}$ where
the three inclusions hold by
Proposition~\ref{prop:N-L}, (\ref{eq:N_k}),
and~(\ref{eq:F_k}), respectively.
Thus
\begin{equation}\label{eq:hbhsjcbvajhs}
     W^u\setminus D^u_y
     =\phi_{(-2\tau,\infty)} S^u_\eps
     =\phi_{(0,\infty)} (\phi_{-2\tau} S^u_\eps)
     \subset F_{k-2}
\end{equation}
by semi-flow invariance of $F_{k-2}$.
Because $N_y\subset B^\rho_y\subset B^{2\rho}_y
\subset B^u\times B^+$ by 
Hypothesis~\ref{hyp:rho-eps-tau},
the $(k-1)$-sphere
$S^u:=\p B^u\subset W^u$
is disjoint to $N_y$, thus to $D^u_y$.
In fact, the distance between $S^u$ and $N_y$
is at least $\rho$. Consequently $S^u\subset
W^u\setminus D^u_u\subset F_{k-2}$
by~(\ref{eq:hbhsjcbvajhs}).
Therefore by openness of $F_{k-2}$ and
compactness of its subset $S^u$
there is a radius $r\in(0,1)$
such that the family $S^u\times B^+_r$
of radius $r$ balls $B^+_r$ about $0\in X^+$
is contained in $F_{k-2}$. To summarize
\begin{equation}\label{eq:pu-Vr}
     \p^u V_r:=S^u\times B^+_r\subset F_{k-2}
     ,\qquad
     \p^u V_r\cap N_y=\emptyset.
\end{equation}
The forward $\lambda$-Lemma asserts
that for every sufficiently large time $T$ the part
$$
     D^T_\ell:=B^T_\ell\cap
     \left(B^u\times B^+\right)
     =\graph\; G^T
     ,\qquad
     G^T\in C^1(B^u,B^+),
$$  
of the disk $B^T_\ell=\varphi_T\alpha^x(B_\ell)
=\beta^x\theta^\ell(\D^{k-1})$ inside
$B^u\times B^+$ is the graph of a $C^1$ map
$G^T:B^u\to B^+$ whose $C^1$ norm
converges to zero, as $T\to\infty$.
Thus choose $T$ in~(\ref{eq:time-AAAA})
larger, if necessary, to obtain that
$\norm{G^T}_{C^1}<r$.
Then, as elements of $\Ho_{k-1}(F_{k-1},F_{k-2})$,
the following classes are equal
\begin{equation*}
\begin{split}
     \left(\bar\beta^x\bar\theta^\ell\right)_*[\D^{k-1}]
     =[B^T_\ell]
     =[D^T_\ell]
     =[B^u]
     =[D^u_y]
     =\left(\bar\vartheta^y\right)_*[\D^{k-1}].
\end{split}
\end{equation*}
Here the first identity is just by definition
of the maps. The class of $B^T_\ell$ is
well defined in relative homology
by~(\ref{eq:cruc-id}) building on
definition~(\ref{eq:time-AAAA}) of $T$.
The part of the disk $B^T_\ell$ in
$V$ is $D^T_\ell=G^T(B^u)$ whose boundary
lies in $\p^u V_r$, hence in $F_{k-2}$
by~(\ref{eq:pu-Vr}). So $D^T_\ell$
is a cycle relative $F_{k-2}$. On the other
hand, its complement
$B^T_\ell\setminus D^T_\ell$ lies outside $V$,
hence outside $N_y$, and therefore in $F_{k-2}$
by~(\ref{eq:time-TTT}). Consequently the classes
of $B^T_\ell$ and $D^T_\ell$ coincide relative
$F_{k-2}$. Concerning identity three observe
that $D^T_\ell$ and $B^u$ are isotopic through
the embedded disks $\graph\:\lambda G^T$,
for $\lambda\in[0,1]$, whose boundaries
lie in $\p^u V_r\subset F_{k-2}$.
Identity four uses that $B^u\setminus D^u_y
\subset W^u\setminus D^u_y\subset F_{k-2}$
by~(\ref{eq:hbhsjcbvajhs}).
The final identity five holds by
choice of the diffeomorphism $\vartheta^x$
in~(\ref{eq:vartheta-2}).
\\
This proves~(\ref{eq:id-homol}) modulo
signs. So it only remains to study orientations.

\vspace{.1cm}
\noindent
{\bf Step~2 (Orientations).}
To prove~(\ref{eq:id-homol})
recall the definition of the transport
$u^\ell_*\langle x\rangle$ of the
orientation $\langle x\rangle$ of $W^u(x)$
along the heat flow trajectory $u^\ell$
between the critical points
$x$ and $y:=u^\ell(+\infty)$
towards an orientation of $W^u(y)$.
By Lemma~\ref{le:asc-disk} for small
$\eps>0$ the ascending disk $W^s_\eps(y)$
is a codimension $(k-1)$ submanifold
of $\Lambda^a M$.
Choosing $T$ larger, if necessary,
the point $p_\ell:=u^\ell(T)$
which anyway lies on the trajectory
$u^\ell$ from $x$ to $y$ moves closer
to $y$ and eventually lies in $W^s_\eps(y)$.
By the Morse-Smale condition
the orthogonal\footnote{
  with respect to the Hilbert structure
  of $\Lambda M$
  }
complement 
$T_{p_\ell}W^s_\eps(y)^\perp$ is a subspace
of $T_{p_\ell}W^u(x)$. The latter splits as
a direct sum of subspaces
\begin{equation}\label{eq:orient-transport}
     T_{p_\ell}W^u(x)
     =\R\left(\tfrac{d}{ds}\varphi_s p_\ell\right)\oplus
     T_{p_\ell}W^s_\eps(y)^\perp
     ,\qquad
     p_\ell:=u^\ell(T).
\end{equation}
Since two of the three vector spaces are
oriented, namely by $\langle x\rangle$ and
by the downward flow, the third space
inherits an orientation as well. Thereby
providing a co-orientation along all of the
(contractible) ascending disk $W^s_\eps(y)$,
in particular, at the point $y$ itself.
But $T_yW^s_\eps(y)^\perp=T_yW^u(y)$, so the
co-orientation determines an orientation
of  the unstable manifold $W^u(y)$ called
the {\boldmath\bf push-forward orientation
of $\langle x\rangle$ along the
flow line $u^\ell$} and denoted by
{\boldmath $u^\ell_*\langle x\rangle$}.

On the other hand, the boundary
orientation of $\SS^{k-1}$
is determined by an outward pointing
vector field and the canonical orientation
of $\D^k$. Given the orientation
$\langle x\rangle$ of $W^u(x)$,
the boundary orientation of
the $(k-1)$-sphere
$S^u_x=\p D^u_x\subset W^u(x)$
arises the same way using
the (outward pointing) downward gradient
vector field.
But the sign $\sigma_{\langle x\rangle}$ of
the diffeomorphism $\vartheta^x$ has
been chosen in~(\ref{eq:sign-vartheta})
precisely to make
$\vartheta^x\circ\mu^{\kappa_{\langle x\rangle}}$
and its restriction to the boundary
preserve these orientations. In particular,
there is the oriented direct sum
\begin{equation}\label{eq:orient-gamma}
     \left\langle T_{p_\ell} W^u(x)
     \right\rangle_{\langle x\rangle}
     =\left\langle\R\left(\tfrac{d}{ds}\varphi_s p_\ell
     \right)\right\rangle_{\text{flow}}
     \oplus
     \left\langle T_{p_\ell} B_\ell^T
     \right\rangle_{\varphi_T\alpha^x
     \mu^{\kappa_{\langle x\rangle}}}.
\end{equation}
Compare these orientations
with the ones in~(\ref{eq:orient-transport}),
which determine $u^\ell_*\langle x\rangle$,
to obtain that $(\varphi_T\bar\alpha^x
\mu^{\kappa_{\langle x\rangle}})_*
(\D^k_{\langle\mathrm{can}\rangle})
=u^\ell_*\langle x\rangle=
(\bar\vartheta^y\mu^{\kappa_\ell})_*
(\D^k_{\langle\mathrm{can}\rangle})$ where
$\kappa_\ell=\kappa_{u^\ell_*\langle x\rangle}$
and where the second identity
holds by the very definition of the sign
$\sigma_{u^\ell_*\langle x\rangle}$.
\end{proof}

\subsection{The natural isomorphism on homology}
\begin{theorem}\label{thm:morse}
Suppose $M$ is simply connected.
Assume $\Vv:\Ll M\to\R$ is a perturbation
that satisfies~{\rm (V0)--(V3)}
in~\cite{weber:2013b} and
$\Ss_\Vv$ is Morse-Smale below
a regular value $a\in\R$.
Then there is a natural isomorphism
$$
     \Psi^a_*:
     \HM^a_*(\Lambda M,\Ss_\Vv) 
     \to
     \mathrm{H}_*(\Lambda^aM)
$$
which commutes with the homomorphisms
$\HM^b_*(\Lambda M,\Ss_\Vv)\to
\HM^a_*(\Lambda M,\Ss_\Vv)$
and $\mathrm{H}_*(\Lambda^bM)
\to\mathrm{H}_*(\Lambda^aM)$ for $b<a$.
\end{theorem}

\begin{proof}[Proof of Theorem~\ref{thm:morse}]
Suppose $\Ss_\Vv$ is Morse-Smale below level $a$
and $b\le a$ are regular values.
Consider the Morse filtrations $\Ff(\Lambda^b M)
\hookrightarrow\Ff(\Lambda^a M)$
provided by~(\ref{eq:F_k}) and~(\ref{eq:F_k-b}).
Then the desired natural isomorphism
is the composition
of the two horizontal natural
isomorphisms in the following diagram.
\begin{equation*}
\begin{split}
\xymatrix{
     \Psi^a_*\,:\,
     \HM_*^a(\Lambda M,\Ss_\Vv)
     \ar[r]_{\hspace{.4cm}\cong}^{\hspace{.35cm}[\Theta_*^a]}
    &
     \Ho_*\Ff\left(\Lambda^a M\right)
     \ar[r]_{\cong \hspace{.3cm}}
       ^{(\ref{eq:isomorphism-cellular-singular})\hspace{.4cm}}
    &
     \mathrm{H}_*(\{\Ss_\Vv\le a\})
 \\
     \Psi^b_*\, : \,
     \HM_*^b(\Lambda M,\Ss_\Vv)
     \ar[u]^{\iota_*}
     \ar[r]_{\hspace{.4cm}\cong}^{\hspace{.35cm}[\Theta_*^b]}
    &
     \Ho_*\Ff\left(\Lambda^b M\right)
     \ar[r]_{\cong \hspace{.3cm}}
       ^{(\ref{eq:isomorphism-cellular-singular}) \hspace{.4cm}}
     \ar[u]^{\iota_*}
    &
     \mathrm{H}_*(\{\Ss_\Vv\le b\})
     \ar[u]^{\iota_*}
}
\end{split}
\end{equation*}
Concerning the left rectangle observe
that already both \emph{chain complexes},
underlying $\HM_*$ and $\Ho_*\Ff$,
are naturally identified for each regular
level $b\le a$ by the chain complex
isomorphism $\Theta_*^b$ --
see Theorem~\ref{thm:cellular-filtration}
and Theorem~\ref{thm:Morse=triple} --
which we actually established above for the present
class of abstract potentials $\Vv$.
Each of the two vertical maps $\iota_*$ is
induced by the inclusion of the subcomplex
associated to $b$. Thus the left rectangle
already commutes on the chain level.
The right rectangle is due to and commutes by
Theorem~\ref{thm:cellular=singular}.
\end{proof}

\begin{proof}[Proof of Theorem~\ref{thm:main}]
Consider the Morse function
$\Ss_V$ in Theorem~\ref{thm:main}
and pick a regular value $a$.
Then the transversality 
theorem~\cite[\S 1.2 Thm.~8]{weber:2013b}
provides, for each regular perturbation
$v\in\Oo^a_{reg}$, the second
of the two natural isomorphisms
\begin{equation}\label{eq:jhihi}
     \HM_*^a(\Lambda M,\Ss_{V+v})
     \stackrel{\Psi^a_*}{\cong}
     \mathrm{H}_*(\{\Ss_{V+v}\le a\})
     \cong
     \mathrm{H}_*(\{\Ss_{V}\le a\})
\end{equation}
where, of course, the notation $\Ss_{V+v}$
is slightly abusive. The first isomorphism
$\Psi^a_*$ is due to Theorem~\ref{thm:morse}
and the second one
to~\cite[\S 5.2 Prop.~8]{weber:2013b}.
Concerning $\Psi^a_*$ it is crucial
that $\Ss_{V+v}$ is Morse-Smale
below level $a$ -- which holds by
regularity of $v$ -- and concerning
the second isomorphism that $v$ lies
in the radius $r_a$ ball $\Oo^a$
defined by~\cite[(62)]{weber:2013b}.
This proves~(\ref{eq:gy6g5}), thus the
first part of Theorem~\ref{thm:main}.

Now assume that $a<b$ are regular values
of $\Ss_V$. The set of admissible
perturbations $\Oo^b$ given
by~\cite[(62)]{weber:2013b} is a closed ball
about zero in a separable Banach space.
Pick a regular perturbation
$v\in\Oo^b_{reg}\subset\Oo^b$
whose norm is bounded from above by
the constant $\delta^a/2$ given
by~\cite[(61)]{weber:2013b}.
In this case $v$ is in the set $\Oo^a$ 
by~\cite[\S 5.2 Rmk.~4]{weber:2013b}
and therefore enjoys the properties stated
in~\cite[\S 5.2 Prop.~8]{weber:2013b}
for both values $a$ and $b$;
see also the transversality
theorem~\cite[\S 1.2 Thm.~8]{weber:2013b}.
Of course, as the perturbed action
$\Ss_{V+v}$ is Morse-Smale below level $b$,
it is so below level~$a$.
Hence $v\in\Oo^a_{reg}\cap \Oo^b_{reg}$
and therefore we obtain, just as above, the
horizontal isomorphisms in the diagram
\begin{equation}\label{eq:gy656g}
\begin{split}
\xymatrix{
     \HM_*^b(\Lambda M,\Ss_{V+v})
     \ar[r]^{\Psi^b_*}
     &
     \mathrm{H}_*(\{\Ss_{V+v}\le b\})
     \ar[r]^{(\ref{eq:comm-tri-h})_b}
     &
     \mathrm{H}_*(\{\Ss_{\Vv}\le b\})
   \\
     \HM_*^a(\Lambda M,\Ss_{V+v})
     \ar[r]^{\Psi^a_*}
     \ar[u]^{\iota_*}
     &
     \mathrm{H}_*(\{\Ss_{V+v}\le a\})
     \ar[r]^{(\ref{eq:comm-tri-h})_a}
     \ar[u]^{\iota_*}
     &
     \mathrm{H}_*(\{\Ss_{\Vv}\le a\}).
     \ar[u]^{\iota_*}
}
\end{split}
\end{equation}
Here the left rectangle commutes
by Theorem~\ref{thm:morse}.
To see that the rectangle on the
right commutes use commutativity
of diagram~(\ref{eq:comm-tri-h}) for $a$
and for $b$ together with the inclusion
induced homomorphisms between both
diagrams and functoriality of singular
homology. This proves
Theorem~\ref{thm:main} when $a<\infty$.
The case $a=\infty$ follows from
functoriality and a direct limit argument.
\end{proof}

\begin{remark}
Consider part~II) of the proof
of~\cite[\S 5.2 Prop.~8]{weber:2013b}.
The resulting two homomorphisms --
one injection and one surjection --
fit into the (by functoriality
of singular homology) commutative
rectangle
\begin{equation}\label{eq:comm-tri-h}
\begin{split}
\xymatrix{
     {\rm H}_*(\left\{
     \Ss_{\Vv+v_\lambda}\le a\right\})
     \ar[r]^{\hspace{.1cm}\iota_*}_{\hspace{.1cm}\text{surj.}}
     \ar[dr]
   &
     {\rm H}_*(\left\{
     \Ss_{\Vv}\le a_+\right\})
   \\
     {\rm H}_*(\left\{
     \Ss_{\Vv+v_\lambda}\le a_-\right\})
     \ar[u]_\cong^{\iota_*}
     \ar[r]^{\hspace{.1cm}\iota_*}_{\hspace{.1cm}\text{inj.}}
   &
     {\rm H}_*(\left\{
     \Ss_{\Vv}\le a\right\})
     \ar[u]^\cong_{\iota_*}.
}
\end{split}
\end{equation}
of four inclusion induced
homomorphisms, all denoted by $\iota_*$.
Consequently both horizontal maps
are isomorphisms and this
defines the isomorphism indicated by the
diagonal arrow which divides the square
into two commutative triangles.
\end{remark}

\vspace{.2cm}
\noindent
{\bf Acknowledgements.} {\small
For extremely useful and pleasant discussions
the author is indebted to
Alberto Abbondandolo and Klaus Mohnke.
Many thanks to both of them.
\\
The present paper was announced
in previous publications under different
titles, for instance in~\cite{weber:2013b} as 
{\it Stable foliations associated to level sets and the homology of the loop space}
and in~\cite{weber:2014a,weber:2014b} as
{\it Stable foliations and the homology
of the loop space}.
}

\bibliography{library_math}{}

\begin{thebibliography}{10}

\bibitem{abbondandolo:2006a}
A.~Abbondandolo and P.~Majer.
\newblock Lectures on the {M}orse complex for infinite-dimensional manifolds.
\newblock In {\em Morse theoretic methods in nonlinear analysis and in
  symplectic topology}, volume 217 of {\em NATO Sci. Ser. II Math. Phys.
  Chem.}, pages 1--74. Springer, Dordrecht, 2006.

\bibitem{abbondandolo:2013b}
A.~Abbondandolo and M.~Schwarz.
\newblock The role of the {L}egendre transform in the study of the {F}loer
  complex of cotangent bundles.
\newblock {\em \href{http://arxiv.org/abs/1306.4087}{{\rm arXiv 1306.4087}}},
  2013.

\bibitem{conley:1978a}
C.~Conley.
\newblock {\em Isolated invariant sets and the {M}orse index}, volume~38 of
  {\em CBMS Regional Conference Series in Mathematics}.
\newblock American Mathematical Society, Providence, R.I., 1978.

\bibitem{dold:1995a}
A.~Dold.
\newblock {\em Lectures on algebraic topology}.
\newblock Classics in Mathematics. Springer-Verlag, Berlin, 1995.
\newblock Reprint of the 1972 edition.

\bibitem{henry:1981a}
D.~Henry.
\newblock {\em Geometric theory of semilinear parabolic equations}, volume 840
  of {\em Lecture Notes in Mathematics}.
\newblock Springer-Verlag, Berlin, 1981.

\bibitem{katetov:1951a}
M.~Kat{\v{e}}tov.
\newblock On real-valued functions in topological spaces.
\newblock {\em Fund. Math.}, 38:85--91, 1951.

\bibitem{kell:2012a}
M.~Kell.
\newblock {P}rivate communication after the talk ''{C}onley theory and the heat
  flow'' by {J}. {W}eber.
\newblock Workshop "Topologia e Din{\^a}mica", UFF Niter\'oi, Brazil, 10
  February 2012.

\bibitem{milnor:1963a}
J.~Milnor.
\newblock {\em Morse theory}.
\newblock Based on lecture notes by M. Spivak and R. Wells. Annals of
  Mathematics Studies, No. 51. Princeton University Press, Princeton, N.J.,
  1963.

\bibitem{milnor:1965a}
J.~Milnor.
\newblock {\em Lectures on the {$h$}-cobordism theorem}.
\newblock Notes by L. Siebenmann and J. Sondow. Princeton University Press,
  Princeton, N.J., 1965.

\bibitem{munkres:1984a}
J.~R. Munkres.
\newblock {\em Elements of algebraic topology}.
\newblock Addison-Wesley Publishing Company, Menlo Park, CA, 1984.

\bibitem{palais:1966a}
R.~S. Palais.
\newblock Homotopy theory of infinite dimensional manifolds.
\newblock {\em Topology}, 5:1--16, 1966.

\bibitem{palais:1969a}
R.~S. Palais.
\newblock The {M}orse lemma for {B}anach spaces.
\newblock {\em Bull. Amer. Math. Soc.}, 75:968--971, 1969.

\bibitem{palis:1982a}
J.~Palis, Jr. and W.~de~Melo.
\newblock {\em Geometric theory of dynamical systems}.
\newblock Springer-Verlag, New York, 1982.
\newblock An introduction, Translated from the Portuguese by A. K. Manning.

\bibitem{po-zniak:1991a}
M.~Po\'{z}niak.
\newblock The {M}orse complex, {N}ovikov homology, and {F}redholm theory.
\newblock Preprint, University of Warwick, 1991.

\bibitem{rybakowski:1987a}
K.~P. {Rybakowski}.
\newblock {\em {The homotopy index and partial differential equations}}.
\newblock Universitext. Springer-Verlag, Berlin, 1987.

\bibitem{salamon:1990a}
D.~Salamon.
\newblock Morse theory, the {C}onley index and {F}loer homology.
\newblock {\em Bull. London Math. Soc.}, 22(2):113--140, 1990.

\bibitem{salamon:2006a}
D.~Salamon and J.~Weber.
\newblock Floer homology and the heat flow.
\newblock {\em Geom. Funct. Anal.}, 16(5):1050--1138, 2006.

\bibitem{schwarz:1993a}
M.~Schwarz.
\newblock {\em Morse homology}, volume 111 of {\em Progress in Mathematics}.
\newblock Birkh{\"a}user Verlag, Basel, 1993.

\bibitem{smale:1960a}
S.~Smale.
\newblock Morse inequalities for a dynamical system.
\newblock {\em Bull. Amer. Math. Soc.}, 66:43--49, 1960.

\bibitem{smale:1961a}
S.~Smale.
\newblock On gradient dynamical systems.
\newblock {\em Ann. of Math. (2)}, 74:199--206, 1961.

\bibitem{thom:1949a}
R.~Thom.
\newblock Sur une partition en cellules associ{\'e}e {\`a} une fonction sur une
  vari{\'e}t{\'e}.
\newblock {\em C. R. Acad. Sci. Paris}, 228:973--975, 1949.

\bibitem{tong:1952a}
H.~Tong.
\newblock Some characterizations of normal and perfectly normal spaces.
\newblock {\em Duke Math. J.}, 19:289--292, 1952.

\bibitem{weber:2014d}
J.~Weber.
\newblock Global stable foliations for the heat flow.
\newblock In preparation.

\bibitem{weber:heat-book}
J.~Weber.
\newblock The heat flow and the homology of the loop space.
\newblock Book in preparation.

\bibitem{weber:1993a-link}
J.~Weber.
\newblock Der {M}orse-{W}itten {K}omplex.
\newblock
  \href{http://www.math.sunysb.edu/~joa/PUBLICATIONS/1993-madip.pdf}{{M}aster's
  thesis}, {TU}~{B}erlin, February 1993.

\bibitem{weber:2002a}
J.~Weber.
\newblock Perturbed closed geodesics are periodic orbits: index and
  transversality.
\newblock {\em Math. Z.}, 241(1):45--82, 2002.

\bibitem{weber:2010a-link}
J.~Weber.
\newblock The heat flow and the homology of the loop space.
\newblock
  \href{http://www.math.sunysb.edu/~joa/PUBLICATIONS/2010-habilitation.pdf}{{H}abilitation
  monograph}, {HU}~{B}erlin, February 2010.

\bibitem{weber:2013b}
J.~Weber.
\newblock Morse homology for the heat flow.
\newblock {\em Math. Z.}, 275(1-2):1--54, 2013.

\bibitem{weber:2013a}
J.~Weber.
\newblock {M}orse homology for the heat flow -- {L}inear theory.
\newblock {\em Math. Nachr.}, 286(1):88--104, 2013.

\bibitem{weber:2014b}
J.~Weber.
\newblock The {B}ackward $\lambda$-{L}emma and {M}orse {F}iltrations.
\newblock {\em Progress in Nonlinear Differential Equations and Their
  Applications}, 85:457--466, 2014.

\bibitem{weber:2014a}
J.~Weber.
\newblock A backward $\lambda$-lemma for the forward heat flow.
\newblock {\em Math. Ann.}, 359(3-4):929--967, 2014.

\bibitem{witten:1982a}
E.~Witten.
\newblock Supersymmetry and {M}orse theory.
\newblock {\em J. Differential Geom.}, 17(4):661--692, 1982.

\end{thebibliography}
\bibliographystyle{abbrv}


\end{document}